\documentclass[11pt]{article}
\usepackage[left=2.7cm,right=2.7cm,top=2.2cm,bottom=2.2cm]{geometry}
\usepackage{amssymb}
\usepackage{amsmath}
\usepackage{amsthm}
\usepackage{enumerate}
\usepackage{mathrsfs}
\usepackage{stmaryrd}
\usepackage{bbm}
\usepackage{xcolor}
\usepackage{mathtools}
\usepackage[hyperfootnotes=false]{hyperref}

\theoremstyle{definition}
\newtheorem{definition}{Definition}[section]

\newtheorem{example}[definition]{Example}
\newtheorem{assumption}[definition]{Assumption}

\theoremstyle{plain}
\newtheorem{proposition}[definition]{Proposition}
\newtheorem{theorem}[definition]{Theorem}
\newtheorem{lemma}[definition]{Lemma}
\newtheorem{corollary}[definition]{Corollary}

\theoremstyle{remark}
\newtheorem{remark}[definition]{Remark}

\numberwithin{equation}{section}

\newcommand{\E}{\mathbb{E}}
\newcommand{\G}{\mathbb{G}}
\newcommand{\N}{\mathbb{N}}
\renewcommand{\P}{\mathbb{P}}

\newcommand{\R}{\mathbb{R}}

\newcommand{\U}{\mathbb{U}}

\newcommand{\X}{\mathbb{X}}

\newcommand{\Z}{\mathbb{Z}}

\newcommand{\cB}{\mathcal{B}}

\newcommand{\cD}{\mathcal{D}}

\newcommand{\cF}{\mathcal{F}}
\newcommand{\cG}{\mathcal{G}}

\newcommand{\cL}{\mathcal{L}}

\newcommand{\cP}{\mathcal{P}}

\newcommand{\cU}{\mathcal{U}}
\newcommand{\cV}{\mathcal{V}}

\newcommand{\bG}{\mathbf{G}}

\newcommand{\bV}{\mathbf{V}}
\newcommand{\bX}{\mathbf{X}}
\newcommand{\bY}{\mathbf{Y}}
\newcommand{\bZ}{\mathbf{Z}}

\newcommand{\1}{\mathbf{1}}
\renewcommand{\d}{\mathrm{d}}
\newcommand{\dd}{\,\mathrm{d}}

\renewcommand{\epsilon}{\varepsilon}

\newcommand{\edot}{\boldsymbol{\cdot}}

\newcommand{\tp}{\tilde{p}}

\newcommand{\tY}{\widetilde{Y}}
\newcommand{\tN}{\widetilde{N}}
\newcommand{\tbX}{\widetilde{\mathbf{X}}}
\newcommand{\tM}{\widetilde{M}}
\newcommand{\tX}{\widetilde{X}}
\newcommand{\tW}{\widetilde{W}}
\newcommand{\tV}{\widetilde{V}}
\newcommand{\tbbX}{\widetilde{\mathbb{X}}}

\newcommand{\D}{\mathrm{D}}
\newcommand{\sV}{\mathscr{V}}

\newcommand{\ty}{\widetilde{y}}
\newcommand{\tP}{\widetilde{\mathbb{P}}}
\newcommand{\tbV}{\widetilde{\bV}}

\begin{document}

\title{Rough SDEs and Robust Filtering for Jump-Diffusions}
\author{Andrew L.~Allan\footnote{Department of Mathematical Sciences, Durham University, andrew.l.allan{\fontfamily{ptm}\selectfont @}durham.ac.uk} \hspace{-2pt}, Jost Pieper\footnote{Department of Mathematical Sciences, Durham University, jost.pieper{\fontfamily{ptm}\selectfont @}durham.ac.uk} \hspace{-0.5pt} and Josef Teichmann\footnote{Department of Mathematics, ETH Z\"urich, josef.teichmann{\fontfamily{ptm}\selectfont @}math.ethz.ch}\\
~\vspace{-5pt}\\
}
\date{\today}
\maketitle

\vspace{-15pt}

\begin{abstract}
We investigate the existence of a robust, i.e., continuous, representation of the conditional distribution in a stochastic filtering model for multidimensional correlated jump-diffusions. Even in the absence of jumps, it is known that in general such a representation can only be continuous with respect to rough path topologies, leading us naturally to express the conditional dynamics as a rough stochastic differential equation with jumps. Via the analysis of such equations, including exponential moments, Skorokhod continuity, and randomisation of the rough path, we establish several novel robustness results for stochastic filters.
\end{abstract}

\vspace{8pt}

\noindent
Keywords: c\`adl\`ag rough paths, rough stochastic differential equations, John--Nirenberg inequality, robust stochastic filtering.

\noindent
MSC 2020 classification:
60L20, 
60G35. 

\tableofcontents

\section{Introduction}

Stochastic filtering is concerned with determining the conditional distribution of an unobserved signal process $X$ from observations of another process $Y$, as both processes evolve continuously in time. Problems involving the filtering of stochastic systems arise frequently in numerous applications, including in finance, economics, defence and aerospace, and such problems have been studied extensively in a wide variety of settings. For a detailed account, we refer to some of the many excellent monographs on the subject, such as \cite{BainCrisan2009}, \cite{CrisanRozovskiu2011}, \cite{Kallianpur1980} and \cite{LipsterShirayev1977}.

At least as far back as the late seventies, Clark \cite{Clark1978} pointed out that it would be natural and desirable, particularly in the context of real-world applications, to obtain a \emph{robust} representation of a stochastic filter, by which one typically means a \emph{continuous} function of the observed path (or some features thereof), which coincides with the conditional distribution. More precisely, for a suitable class of functions $f$, one seeks a continuous adapted function $\Theta^f$ on the pathspace of $Y$, such that
\begin{equation*}
\Theta^f_t(Y) = \E [f(X_t,Y_t) \,|\, \cF^Y_t]
\end{equation*}
holds almost surely, where we write $\cF^Y_t$ for the $\sigma$-algebra generated by $(Y_s)_{s \leq t}$. Such a representation is important for at least two reasons. First, it guarantees that, provided the model we suppose for the observation noise is close (in a suitable weak sense) to its true law, the resulting estimation error will be correspondingly small. In particular, it ensures that small errors in our dynamical model and in our observed data should only result in small errors in the resulting filter, and moreover that approximating the observation process by discrete-time data---as is inevitably the case in practice---leads to an accurate approximation of the conditional distribution. Second, it allows one to apply learning techniques with pointwise errors in mind (in contrast to mean squared errors) to stochastic filters, where the filter is trained as a neural network of relevant features (typically related to the topology with respect to which the filter is continuous) on training data.

In this paper, we investigate the existence of such a robust representation of the conditional distribution in a stochastic filtering model for correlated multidimensional jump-diffusion processes of the form
\begin{equation}\label{eq: intro- SDE}
\begin{split}
\d X_t &= b_1(t,X_t,Y_t) \dd t + \sigma_0(t,X_t,Y_t) \dd B_t + \sigma_1(t,X_t,Y_t) \dd W_t\\
&\quad + \int_{\U_1} f_1(t,X_{t-},Y_{t-},u) \, \tN_1(\d t,\d u) + \int_{\U_2} f_2(t,X_{t-},Y_{t-},u) \, \tN_2(\d t,\d u),\\
\d Y_t &= b_2(t,X_t,Y_t) \dd t + \sigma_2(t,Y_t) \dd W_t + \int_{\U_2} f_3(t,Y_{t-},u) \, \tN_2(\d t,\d u),
\end{split}
\end{equation}
under essentially optimal regularity assumptions on the coefficients, where here $B$ and $W$ are independent Brownian motions, $N_1$ is a Poisson random measure, and $N_2$ is an integer-valued random measure whose compensator depends on $X$. In particular, we consider a \emph{proper} filtering problem, i.e., an equivalent change of measure exists which makes the signal disappear from the dynamics of the observation process. The case of improper filtering problems is more delicate (and generally easier).

In \cite{Clark1978}, Clark considered $X$ and $Y$ as diffusions driven by independent Brownian motions, and expressed the conditional distribution as a ratio of expectations via the Kallianpur--Striebel formula, under a new probability measure. The associated change of measure introduces a term involving the exponential of a stochastic integral, and finding a robust representation in the uncorrelated noise case fundamentally boils down to finding a robust version of this term. Clark suggested such a version via a formal integration by parts, and the continuity of this representation with respect to the uniform topology was also confirmed by Kushner \cite{Kushner1979} under appropriate locally uniform exponential integrability assumptions, and Clark and Crisan \cite{ClarkCrisan2005} provided a rigorous treatment of the measurability issues that arose therein. We refer to \cite[Ch.~5]{BainCrisan2009} for a comprehensive overview of these results.

In the correlated noise case, one needs to find a robust version of the signal itself, in terms of the observation path, in addition to the exponential terms in the Kallianpur--Striebel formula. For scalar observation, Davis used a Doss--Sussman type flow transformation to express $X$ as the composition of a flow of a deterministic ODE driven by the sample path of the observation $Y$, and an SDE with coefficients which are continuous in $Y$. Paired with another formal integration by parts, this provides a robust version of the conditional distribution; for details, see \cite{Davis1980}, \cite{Davis1982}, \cite{Davis2005} and \cite{DavisSpathopoulos1987}. These results were extended to a multidimensional setting by Elliott and Kohlmann \cite{ElliottKohlmann1981}, under some additional commutativity assumptions on the vector fields.

In a general continuous multidimensional correlated noise setting, however, it is known that there cannot exist a robust representation of the conditional distribution on the space of continuous paths endowed with the uniform topology (see \cite[Example~1]{CrisanDiehlFrizOberhauser2013}). Instead, Crisan, Diehl, Friz and Oberhauser \cite{CrisanDiehlFrizOberhauser2013} showed that, in such a setting, upon lifting the observation path to a rough path, there exists a representation of the filter which is continuous with respect to a suitable rough path topology.

Initiated by Lyons \cite{Lyons1998}, the theory of rough paths generalizes classical notions of integration and controlled ODEs to handle highly oscillatory multidimensional paths. A \emph{rough path} may be viewed as a path $X$ which has been enhanced with its iterated integrals, which is sufficient to ensure the well-posedness and stability of solutions to nonlinear differential equations driven by such a path. In the context of stochastic analysis, such rough differential equations (RDEs) provide a robust version of the solution to the corresponding stochastic differential equation (SDE). Unlike its stochastic counterpart, such an RDE is well-posed for a fixed realisation of the driving noise, and the solution is a continuous function of the driving rough path; see, e.g., \cite{FrizHairer2020} or \cite{FrizVictoir2010} for a comprehensive exposition of the theory.

Upon conditioning on the observation $\sigma$-algebra $\cF^Y_t$, it is natural to fix a lifted (in the rough path sense) realization of $(Y_s)_{s \in [0,t]}$, and consider the resulting rough dynamics. However, since the unconditioned noise also remains, the resulting equation is actually driven by both rough \emph{and} stochastic noise. To make sense of such equations, in \cite{CrisanDiehlFrizOberhauser2013} the authors used a flow transformation to make sense of such mixed \emph{rough stochastic differential equations} (rough SDEs), driven by a H\"older continuous rough path and a Brownian motion. They then established locally uniform exponential moments to make sense of the exponential terms appearing in the Kallianpur--Striebel formula, and hence also a locally Lipschitz continuous version of the conditional distribution.

Due to the use of flow transformations, this solution theory for rough SDEs does not provide an intrinsic solution to the equation, and comes with excessive regularity requirements on its coefficients. Over the last decade, alternative approaches to the analysis of rough SDEs have been proposed, such as the \emph{random rough path} approach of, e.g., Diehl, Oberhauser and Riedel \cite{DiehlOberhauserRiedel2015}, in which the rough SDE is treated as an RDE driven by the random joint rough path lift of the stochastic noise together with the rough path. See also Friz and Zorin-Kranich \cite{FrizKranich2023} for a more recent and extensive analysis of the random rough path approach.

Another recent approach was introduced by Friz, Hocquet and L\^e \cite{frizHocLe2021}, in which the stochastic sewing lemma is used to give intrinsic meaning to rough SDEs under optimal regularity requirements on the coefficients. Since its inception, the theory of rough SDEs has received a number of theoretical extensions, such as Malliavin differentiability \cite{BuginiCoghiNilssen2024}, parameter dependent rough SDEs \cite{BuginiFrizStannat2024}, and rough SDEs with jumps \cite{AllanPieper2026}, along with various applications to, e.g., pathwise stochastic control \cite{FrizLeZhang2024}, \cite{HorstZhang2025}, stochastic volatility models \cite{BankBayerFrizPelizzari2025}, and rough McKean--Vlasov equations in a filtering setting \cite{CoghiNilssenNüskenReich2023}.

Motivated in particular by applications in mathematical finance (e.g., \cite{FreyRunggaldier2010}, \cite{MeyerProske2004}), stochastic filtering for jump-diffusions has been studied extensively, and the corresponding filtering equations have been derived in a number of settings (\cite{Poklukar2006}, \cite{PopaSritharan2009}, \cite{QiaoDuan2015}). In particular, cases with correlated noise have been treated in both scalar (\cite{CeciColaneri2012}, \cite{CeciColaneri2014}, \cite{FernandoHausenblas2018}) and multidimensional settings (\cite{DavieGermGyongy2024}, \cite{GermGyongy2025PartI}, \cite{GermGyongy2025}, \cite{GrigelionisMikulevicius1982}, \cite{KliemannKochMarchetti1990}, \cite{Qiao2021}, \cite{Qiao2023}).

Extensions of the robustness result to filtering settings with jumps have been developed by Kushner \cite{Kushner1997} for independent noise, and by Grigelionis and Mikulevicius \cite{GrigelionisMikulevicius1982} in a multidimensional correlated noise setting under a commutativity assumption on the signal dynamics. However, a robustness result for general multidimensional correlated jump-diffusions of the form \eqref{eq: intro- SDE}, in the spirit of \cite{CrisanDiehlFrizOberhauser2013}, has not yet been established. In the present paper we fill this gap, by utilizing the recent theory of rough SDEs with jumps, as developed in \cite{AllanPieper2026}.

A key step in the procedure detailed in both \cite{CrisanDiehlFrizOberhauser2013} and the present work is to condition on one of the two noise components, in what one may refer to as a ``doubly stochastic DE'', to obtain a (random) rough SDE. While it is intuitively clear that the solutions to these equations should coincide, making this rigorous requires one to establish measurability of the solution to the ``randomised'' rough SDE.\footnote{In \cite{CrisanDiehlFrizOberhauser2013} this argument was omitted, but it has become apparent that this is a substantially more delicate task than the authors of \cite{CrisanDiehlFrizOberhauser2013} suggest. Indeed, this problem appears to have been resolved satisfactorily for the first time only very recently in \cite{FrizLeZhang2024}. See Section~\ref{sec: Consistency RSDEs and SDEs} below for an alternative argument.}

Concurrently with the present work, Friz, L\^e and Zhang derived a resolution to this problem in \cite{FrizLeZhang2024} and \cite{FrizLeZhang2025}. Considering solutions to rough SDEs with the driving rough path as a parameter, they establish measurability of the solution with respect to the parameter space by the application of suitable measurable selection theorems, thus obtaining the desired measurability of the randomised rough SDE. Although this approach can undoubtedly be adapted to incorporate jumps, we instead provide an alternative perspective. Rather than resorting to measurable selection, we directly adapt the classical procedure for proving well-posedness of rough SDEs to handle random rough paths, by ``randomising'' the metric space on which we establish an invariant contraction. This allows us to construct the solution to both the randomised rough SDE and to the corresponding doubly stochastic DE \emph{simultaneously}, thus establishing that both solutions are well-defined and measurable, and that they coincide. The proof is very much in the spirit of the classical $\alpha$-slicing construction for the existence of solutions to SDEs driven by c\`adl\`ag semimartingales (see, e.g., \cite[Ch.~16]{Cohen2015}).

Another key ingredient we require is integrability of the exponential moments which appear in the Kallianpur--Striebel formula. The existence of exponential moments for solutions to rough SDEs has been addressed for H\"older continuous noise in \cite{frizHocLe2021}, and a more general result of L\^e \cite{Le2022a} provides exponential moments for BMO processes via a novel version of the John--Nirenberg inequality. We extend the results of \cite{Le2022a} to obtain bounds on so-called $\textup{BMO}^{p\textup{-var}}$ processes, which in particular include c\`adl\`ag rough stochastic integrals.

With these ingredients in place, we establish the existence of a robust version of the conditional distribution associated with the filtering model in \eqref{eq: intro- SDE}. In general, continuity is established with respect to the rough path lift of the relevant noise components (Theorem~\ref{theorem: robustness result}), and in the case of additive noise in the observation we obtain a robustness result akin to and generalizing those in \cite{CrisanDiehlFrizOberhauser2013} and \cite{DiehlOberhauserRiedel2015} (Corollary~\ref{corollary: additive noise robustness requires only continuous and jump part}). Moreover, by considering the filter as a function on the space of random rough paths, we obtain a novel result on the robustness of stochastic filters with respect to model uncertainty (Theorem~\ref{theorem: Lipschitz continuity of Theta in L^m}).

The paper is organised as follows. In Section~\ref{sec: Preliminaries} we discuss the necessary preliminaries from the theory of rough SDEs, as well as extensions of this theory to incorporate integrals against integer-valued random measures, and continuity with respect to Skorokhod topology. We also provide in Theorem~\ref{theorem: a p-variation kolmogorov result} a Kolmogorov continuity-type criterion for processes with c\`adl\`ag paths. In Section~\ref{sec: Consistency RSDEs and SDEs} we establish the consistency between solutions to doubly stochastic DEs and their rough SDE counterparts, and exponential moments are then discussed in Section~\ref{sec: John-Nirenberg}. Finally, in Section~\ref{sec: application to filtering} we combine the ideas developed in the earlier sections to construct a robust representation of the conditional distribution in our stochastic filtering model.

\bigskip

\noindent
\textbf{Acknowledgement:} The authors would like to thank Jannis Dause and Peter Friz for helpful discussions, particularly on Theorem~\ref{theorem: a p-variation kolmogorov result} and Example~\ref{example: p<q is sharp for kolmogorov}.

\section{Preliminaries}\label{sec: Preliminaries}

\subsection{Frequently used notation}

For $0 < T < \infty$, we write $\Delta_{[0,T]} = \{(s,t) \in [0,T]^2 : s \leq t\}$. We call a function $w \colon \Delta_{[0,T]} \to [0,\infty)$ a \emph{control} on $\Delta_{[0,T]}$ if it is superadditive, in the sense that $w(s,u) + w(u,t) \leq w(s,t)$ for all $s \leq u \leq t$. Note that for any control $w$, the map $s \mapsto w(s,t)$ is non-increasing, $t \mapsto w(s,t)$ is non-decreasing, and $w(t,t) = 0$ for all $t$.

We write $w(s,t+) := \lim_{u \searrow t} w(s,u)$ and
\begin{equation*}
w(s,t-) := \begin{cases}
\lim_{u \nearrow t} w(s,u) & \text{if } \, s < t,\\
0 & \text{if } \, s = t,
\end{cases}
\end{equation*}
and define $w(s+,t)$ and $w(s-,t)$ analogously.

\smallskip

By a \emph{partition} $\cP$ of a given interval $[s,t]$, we mean a finite sequence of times $\cP = \{s = t_0 < t_1 < \cdots < t_n = t\}$. We also denote by $|\cP| := \max_{0 \leq i < n} |t_{i+1} - t_i|$ the mesh size of a partition $\cP$. We will also sometimes abuse notation slightly by writing $[u,v] \in \cP$ to identify two consecutive times $u, v \in \cP$, i.e., when $u = t_i$ and $v = t_{i+1}$ for some $i$.

\smallskip

For $p \in [1,\infty)$, given a two-parameter function $F \colon \Delta_{[0,T]} \to E$, taking values in any normed vector space $(E,|\cdot|)$, the \emph{$p$-variation} of $F$ over the interval $[s,t] \subseteq [0,T]$ is defined as
\begin{equation*}
\|F\|_{p,[s,t]} := \bigg(\sup_{\cP \subset [s,t]} \sum_{[u,v] \in \cP} |F_{u,v}|^p\bigg)^{\hspace{-2pt}\frac{1}{p}},
\end{equation*}
where the supremum is taken over all possible partitions $\cP$ of the interval $[s,t]$. We say that $F$ has finite $p$-variation if $\|F\|_{p,[0,T]} < \infty$. We will sometimes also consider $p$-variation over general intervals, e.g., $\|F\|_{p,[s,t)} := \lim_{u \nearrow t} \|F\|_{p,[s,u]}$.

Given a path $A \colon [0,T] \to E$, we write $\delta A$ for its increment function, so that
\begin{equation*}
\delta A_{s,t} = A_t - A_s
\end{equation*}
for all $(s,t) \in \Delta_{[0,T]}$. The $p$-variation of $A$ over $[s,t]$ is defined as the $p$-variation of its increment process, i.e., $\|A\|_{p,[s,t]} := \|\delta A\|_{p,[s,t]}$, and we write $V^p = V^p([0,T];E)$ for the space of (deterministic) c\`adl\`ag paths with finite $p$-variation.

\smallskip

Given a c\`adl\`ag path or process $Y = (Y_t)_{t \in [0,T]}$, we will write $\Delta Y_t := Y_t - Y_{t-}$ for the jump of $Y$ at time $t$.

\smallskip

For $k \in \N$ and a function $f$, we write $\D^k f$ for the $k$-th order Fr\'echet derivative of $f$. We write $C^n_b$ for the space of functions $f$ which are $(n-1)$-times continuously differentiable, such that $f$ and all its derivatives up to order $n-1$ are bounded, and such that the $(n-1)$-th order derivative $\D^{n-1} f$ is Lipschitz continuous. We denote the corresponding norm by $\|f\|_{C^n_b}$.

We will also write, e.g., $\cL(\R^\ell;\R^m)$ for the space of linear maps from $\R^\ell \to \R^m$.

\smallskip

During proofs, we will often use the symbol $\lesssim$ to indicate inequality up to a multiplicative constant. When deriving estimates, this implicit constant will depend on the same variables as the constant specified in the statement of the corresponding estimate.

\subsection{Rough stochastic analysis}

We let $(\Omega,\cF,(\cF_t)_{t \in [0,T]},\P)$ be a filtered probability space, and we will always assume that the filtration satisfies the usual conditions. For $r \in [1,\infty]$, we write $\|\cdot\|_{L^r}$ for the standard Lebesgue norm on $(\Omega,\cF,\P)$. We also adopt the shorthand $\E_s [\, \cdot \,] := \E [\, \cdot \,|\, \cF_s]$ for the conditional expectation at time $s \in [0,T]$.

\smallskip

As in \cite{AllanPieper2026}, for $q \in [1,\infty)$, $r \in [q,\infty]$, $s \in [0,T]$ and a random variable $Y$, we write
\[ \|Y\|_{q,r,s} := \big\| \E_s [|Y|^q]^{\frac{1}{q}} \big\|_{L^r} \]
and we write $L^{q,r}_s$ for the space of random variables $Y$ such that $\|Y\|_{q,r,s} < \infty$.

\smallskip

For convenience, we recall the following properties of the $\|\cdot\|_{q,r,s}$ norm, the proofs of which may be found in \cite[Proposition~2.2]{AllanPieper2026}.

\begin{itemize}
\item $(L^{q,r}_s,\|\cdot\|_{q,r,s})$ is a Banach space.
\item We have that $\|Y\|_{L^q} \leq \|Y\|_{q,r,s} \leq \|Y\|_{L^r}$ and $\| \E_s[Y] \|_{L^r} \leq \|Y\|_{q,r,s}$.
\item The norm $\|Y\|_{q,r,s}$ is non-decreasing in each of the variables $q, r$ and $s$.
\item The following version of H\"older's inequality,
\begin{equation}\label{eq: Holder's inequality for q,r,s norm}
\big\|\E_s[|Y||Z|]\big\|_{L^\ell} \leq \big\|\E_s[|Y|^p]^{\frac{1}{p}}\big\|_{L^{p \ell}} \big\|\E_s[|Z|^q]^{\frac{1}{q}}\big\|_{L^{q \ell}} = \|Y\|_{p,p \ell,s} \|Z\|_{q,q \ell,s},
\end{equation}
holds for any $\ell \in [1,\infty]$ whenever $p, q \in (1,\infty)$ with $\frac{1}{p} + \frac{1}{q} = 1$.
\end{itemize}

For $p, q \in [1,\infty)$, $r \in [q,\infty]$, and a two-parameter stochastic process $F = (F_{s,t})_{(s,t) \in \Delta_{[0,T]}}$, we write
\begin{equation*}
\|F\|_{p,q,r,[s,t]} := \bigg(\sup_{\cP \subset [s,t]} \sum_{[u,v] \in \cP} \|F_{u,v}\|_{q,r,u}^p\bigg)^{\hspace{-2pt}\frac{1}{p}} = \bigg(\sup_{\cP \subset [s,t]} \sum_{[u,v] \in \cP} \big\|\E_u[|F_{u,v}|^q]^{\frac{1}{q}}\big\|_{L^r}^p\bigg)^{\hspace{-2pt}\frac{1}{p}}
\end{equation*}
for $(s,t) \in \Delta_{[0,T]}$. For a stochastic process $Y = (Y_t)_{t \in [0,T]}$, we then let
\begin{equation*}
\|Y\|_{p,q,r,[s,t]} := \|\delta Y\|_{p,q,r,[s,t]}.
\end{equation*}

We write $V^p L^{q,r} = V^p L^{q,r}([0,T];E)$ for the space of stochastic processes $Y \colon \Omega \times [0,T] \to E$ which are c\`adl\`ag almost surely, and are such that $Y_0 \in L^q$ and $\|Y\|_{p,q,r,[0,T]} < \infty$.

In the special case when $r = q$, we also write
\begin{equation}\label{eq: defn p,infty,[0,T] norm}
\|F\|_{p,q,[s,t]} := \|F\|_{p,q,q,[s,t]} = \bigg(\sup_{\cP \subset [s,t]} \sum_{[u,v] \in \cP} \|F_{u,v}\|_{L^q}^p\bigg)^{\hspace{-2pt}\frac{1}{p}},
\end{equation}
with $\|Y\|_{p,q,[s,t]} := \|\delta Y\|_{p,q,[s,t]}$ and $V^p L^q := V^p L^{q,q}$. Of course, we can also define $\|\cdot\|_{p,\infty,[s,t]}$ and $V^p L^\infty$, by simply replacing the $L^q$ norm in \eqref{eq: defn p,infty,[0,T] norm} with an $L^\infty$ norm.

\smallskip

We consider pairs $\bX = (X,\X)$, consisting of a c\`adl\`ag path $X \colon [0,T] \to \R^d$ and a c\`adl\`ag two-parameter function $\X \colon \Delta_{[0,T]} \to \R^{d \times d}$ (where here c\`adl\`ag is understood for each time parameter separately), such that $\|X\|_{p,[0,T]} < \infty$ and $\|\X\|_{\frac{p}{2},[0,T]} < \infty$. For such pairs, and any $(s,t) \in \Delta_{[0,T]}$, we use the seminorm\footnote{We choose this slightly non-standard definition of the rough path norm for the convenient property that the map $(s,t) \mapsto \|\bX\|_{p,[s,t]}^p$ is a control.}
\begin{equation*}
\|\bX\|_{p,[s,t]} := \big( \|X\|_{p,[s,t]}^p + \|\X\|_{\frac{p}{2},[s,t]}^p \big)^{\frac{1}{p}},
\end{equation*}
which induces the pseudometric
\begin{equation*}
(\bX,\tbX) \, \mapsto \, \|\bX - \tbX\|_{p,[s,t]} = \big( \|X - \tX\|_{p,[s,t]}^p + \|\X - \tbbX\|_{\frac{p}{2},[s,t]}^p \big)^{\frac{1}{p}}
\end{equation*}
for pairs $\bX = (X,\X)$ and $\tbX = (\tX,\tbbX)$.

As defined in \cite{FrizZhang2018}, for a given $p \in [2,3)$, a \emph{c\`adl\`ag rough path} is such a pair $\bX = (X,\X)$, such that $\|\bX\|_{p,[0,T]} < \infty$, and such that Chen's relation $\X_{s,t} = \X_{s,u} + \X_{u,t} + \delta X_{s,u} \otimes \delta X_{u,t}$ holds for all $0 \leq s \leq u \leq t \leq T$. We write $\sV^p = \sV^p([0,T];\R^d)$ for the space of c\`adl\`ag rough paths.

\smallskip

We will sometimes write $\Delta \X_t := \X_{t-,t}$ for the ``jump'' of $\X$ at time $t$, and we will also use the shorthand $|\Delta \bX_t| := |\Delta X_t| + |\Delta \X_t|$.

\smallskip

Given a two-parameter process $A = (A_{s,t})_{(s,t) \in \Delta_{[0,T]}}$, we write $\E_{\edot} A$ for the two-parameter process given by
\begin{equation*}
(\E_{\edot} A)_{s,t} := \E_s [A_{s,t}]
\end{equation*}
for every $(s,t) \in \Delta_{[0,T]}$.

\begin{definition}[Definition~4.1 in \cite{AllanPieper2026}]\label{defn: stochastic controlled path}
Let $p \in [2,3)$, $q \in [2,\infty)$ and $r \in [q,\infty]$, and let $X \in V^p$. We call a pair of processes $(Y,Y')$ a \emph{stochastic controlled path} (relative to $X$), if $Y$ and $Y'$ are both adapted, $Y \in V^p L^{q,r}$, $Y' \in V^p L^{q,r}$, $\sup_{s \in [0,T]} \|Y'_s\|_{L^r} < \infty$, and $\|\E_{\edot} R^Y\|_{\frac{p}{2},r,[0,T]} < \infty$, where the two-parameter process $R^Y = (R^Y_{s,t})_{(s,t) \in \Delta_{[0,T]}}$ is defined by
\begin{equation*}
\delta Y_{s,t} = Y'_s \delta X_{s,t} + R^Y_{s,t}
\end{equation*}
for every $(s,t) \in \Delta_{[0,T]}$.

We write $\cV^{p,q,r}_X$ for the space of stochastic controlled paths relative to $X$.
\end{definition}

Given a rough path $\bX \in \sV^p$ and a stochastic controlled path $(Y,Y') \in \cV^{p,q,r}_X$, one can consider the \emph{rough stochastic integral} of $(Y,Y')$ with respect to $\bX$, as defined in the following lemma.

\begin{lemma}[Lemmas~4.3 and 4.6 in \cite{AllanPieper2026}]\label{lemma: rough stochastic integral}
Let $p \in [2,3)$, $q \in [2,\infty)$ and $r \in [q,\infty]$. Let $\bX = (X,\X) \in \sV^p$ be a c\`adl\`ag rough path, and let $(Y,Y') \in \cV^{p,q,r}_X$ be a stochastic controlled path. Then there exists an $L^q$-integrable adapted c\`adl\`ag process $\int_0^\cdot Y_u \dd \bX_u$, such that, for every $(s,t) \in \Delta_{[0,T]}$,
\begin{equation*}
\lim_{|\cP| \to 0} \bigg\|\int_s^t Y_u \dd \bX_u - \sum_{[u,v] \in \cP} \big(Y_u \delta X_{u,v} + Y'_u \X_{u,v}\big)\bigg\|_{q,r,s} = 0,
\end{equation*}
where the limit holds along partitions $\cP$ of the interval $[s,t]$ as the mesh size tends to zero.
\end{lemma}

\subsection{Rough SDEs with random measures}

In our application to stochastic filtering in Section~\ref{sec: application to filtering}, it will be convenient to express the stochastic noise generated by a pure jump process in the language of random measures. We refer to, e.g., \cite[Ch.~13]{Cohen2015} or \cite[Ch.~II]{JacodShiryaev2003} for a detailed exposition on random measures and their associated stochastic integration. For convenience, in this section we will begin by fixing some standard notation, and recalling a few fundamental properties.

We will write $(\U,\cU)$ to denote a given Blackwell space, and consider an integer-valued random measure $N$ on $[0,T] \times \U$. Let us write $\nu$ for the predictable compensator of $N$. We recall that there exists a predictable, non-decreasing and integrable process $A$, and a kernel $K$, such that
\begin{equation}\label{eq: decomposition of nu}
\nu(\d t,\d u) = K(t,\d u) \dd A_t.
\end{equation}
In general the process $A$ here is c\`adl\`ag, but, for our purposes, it will always be assumed to be continuous. In particular, this means that, almost surely, $\nu(\{t\} \times \U) = 0$ for each $t \in [0,T]$.

If $N$ is a Poisson measure, then its compensator $\nu$ is deterministic, and if $N$ is a homogeneous Poisson measure, then $\nu(\d t,\d u) = F(\d u) \dd t$ for some (deterministic) $\sigma$-finite measure $F$.

\begin{lemma}[Ch.~II, Proposition~1.14 in \cite{JacodShiryaev2003}]
Let $N$ be an integer-valued random measure on $[0,T] \times \U$. Then there exists an optional $\U$-valued process $\beta$, and a thin random set $D$, such that, for almost every $\omega \in \Omega$,
\begin{equation}\label{eq: definition D_i and beta^i}
N(\omega;\d t,\d u) = \sum_{0 < s \leq T} \1_D(\omega,s) \delta_{(s,\beta_s(\omega))}(\d t,\d u),
\end{equation}
where $\delta_{(s,u)}$ denotes the Dirac measure at the point $(s,u)$.
\end{lemma}

In this setting (in particular with $A$ assumed to be continuous), given an integer-valued random measure $N$, we denote by $G_{\textup{loc}}(N)$ the set of predictable functions $\zeta$ on $\Omega \times [0,T] \times \U$ such that the process $(\sum_{s \leq \cdot} |\zeta(\omega;s,\beta_s(\omega))|^2 \1_D(\omega,s))^{\frac{1}{2}}$ is locally integrable. Given $\zeta \in G_{\textup{loc}}(N)$, we call the \emph{stochastic integral} of $\zeta$ with respect to $\widetilde{N} := N - \nu$, denoted $\int_0^\cdot \int_{\U} \zeta(s,u) \, \widetilde{N}(\d s,\d u)$, the purely discontinuous local martingale whose jumps coincide with the process $\zeta(\cdot,\beta) \1_D$.

\begin{assumption}\label{assumption: Regularity of measure integral for RSDE}
Let $\nu$ be a predictable random measure on $[0,T] \times \U$, and suppose that $A$ and $K$ are a process and kernel such that \eqref{eq: decomposition of nu} holds. We assume that the process $A$ is almost surely continuous, and that $A \in V^{\frac{p}{2}} L^{\frac{q}{2},\infty} \cap V^{\frac{p}{q}} L^{1,\infty}$, for some $2 \leq q \leq p < 3$. Let $g$ be a measurable function on $[0,T] \times \R^m \times \U$. We assume that
\begin{equation*}
\bigg\| \sup_{s \in [0,T]} \int_{\U} \big( |g(s,0,u)|^2 \vee |g(s,0,u)|^q \big) \, K(s,\d u) \bigg\|_{L^\infty} < \infty,
\end{equation*}
and, additionally, that there exists a constant $C > 0$ such that, for each $\ell \in \{2,q\}$, any $y, \ty \in \R^m$ and any $s \in [0,T]$,
\begin{equation*}
\int_{\U} |g(s,y,u) - g(s,\ty,u)|^\ell \, K(s,\d u) \leq C (1 \wedge |y - \ty|^\ell)
\end{equation*}
holds $\P$-almost surely.
\end{assumption}

In particular, if $N$ is an integer-valued random measure with compensator $\nu$, and if $g$ is a measurable function on $[0,T] \times \R^m \times \U$, then, under Assumption~\ref{assumption: Regularity of measure integral for RSDE}, it is straightforward to see that $(g(s,Z_s,\cdot))_{s \in [0,T]} \in G_{\textup{loc}}(N)$ for any $\R^m$-valued predictable process $Z$.

\smallskip

We consider the rough SDE given by
\begin{equation}\label{eq: RSDE with measure}
Y_t = y_0 + \int_0^t b(Y_s) \dd s + \int_0^t \sigma(Y_{s-}) \dd M_s + \int_0^t \int_{\U} g(s,Y_{s-},u) \, \tN(\d s,\d u) + \int_0^t f(Y_s) \dd \bX_s
\end{equation}
for $t \in [0,T]$, where $M$ is a c\`adl\`ag martingale, $\widetilde{N} = N - \nu$ is a compensated integer-valued random measure, and $\bX$ is a c\`adl\`ag rough path. In particular, the second integral is an It\^o integral against $M$, the third integral is a stochastic integral against $\widetilde{N}$ in the sense of random measures, and the last integral is the rough stochastic integral of $(f(Y),\D f(Y) Y')$ against $\bX$, in the sense of Lemma~\ref{lemma: rough stochastic integral}.

The main result of \cite{AllanPieper2026} was to establish existence, uniqueness and stability of solutions to rough SDEs driven by a c\`adl\`ag martingale and a c\`adl\`ag rough path. It is straightforward to extend this result to include an integer-valued random measure, which we state precisely in the theorem below. Since the proof is just a straightforward adaptation of the proof of \cite[Theorem~5.2]{AllanPieper2026}, using the estimate in part~(ii) of Lemma~\ref{lemma: application of conditional BDG for jump measures}, we omit the proof here for brevity.

\begin{theorem}\label{theorem: existence and estimates for solutions to RSDEs with measures}
Let $2 \leq q \leq p < 3$, $b \in C^1_b$, $\sigma \in C^1_b$, $f \in C^3_b$ and $g \colon [0,T] \times \R^m \times \U \to \R^m$. Let $y_0 \in L^q$ be $\cF_0$-measurable, $\bX = (X,\X) \in \sV^p$ be a c\`adl\`ag rough path, and let $M \in V^p L^{q,\infty}$ be a c\`adl\`ag martingale. Further, let $N$ be an integer-valued random measure on $[0,T] \times \U$ for some Blackwell space $(\U,\cU)$, with compensator $\nu$, write $\tN = N - \nu$ for the corresponding compensated random measure, and suppose that $\nu$ and $g$ satisfy Assumption~\ref{assumption: Regularity of measure integral for RSDE}.

Then there exists a process $Y$, which is unique up to indistinguishability, such that $Y$ has almost surely c\`adl\`ag sample paths, $(Y,Y') \in \cV^{p,q,\infty}_X$ is a stochastic controlled path, where $Y' = f(Y)$, and such that, almost surely, \eqref{eq: RSDE with measure} holds for every $t \in [0,T]$.

Moreover, if $y_0, \ty_0 \in L^q$ are $\cF_0$-measurable, $\bX, \tbX \in \sV^p$ are two c\`adl\`ag rough paths, and $M, \tM \in V^p L^{q,\infty}$ are c\`adl\`ag martingales, such that the norms $\|\bX\|_{p,[0,T]}$, $\|\tbX\|_{p,[0,T]}$, $\|M\|_{p,q,\infty,[0,T]}$ and $\|\tM\|_{p,q,\infty,[0,T]}$ are all bounded by some constant $L > 0$, and if $Y, \tY$ are the solutions of \eqref{eq: RSDE with measure} corresponding to the data $(y_0,M,N,\bX)$ and $(\ty_0,\tM,N,\tbX)$ respectively, then we have that
\begin{equation}\label{eq: Lipschitz continuity of solution map}
\begin{split}
\|Y &- \tY\|_{p,q,[0,T]} + \|Y' - \tY'\|_{p,q,[0,T]} + \|\E_{\edot} (R^Y - R^{\tY})\|_{\frac{p}{2},q,[0,T]}\\
&\leq C \big(\|y_0 - \ty_0\|_{L^q} + \|M - \tM\|_{p,q,[0,T]} + \|\bX - \tbX\|_{p,[0,T]}\big),
\end{split}
\end{equation}
where the constant $C$ depends only on $p, q, \|b\|_{C^1_b}, \|\sigma\|_{C^1_b}, \|f\|_{C^3_b}, g, \nu, T$ and $L$.
\end{theorem}

\subsection{Skorokhod continuity for rough SDEs}

We denote by $\Lambda$ the set of increasing bijective functions from $[0,T] \to [0,T]$. Given a $\lambda \in \Lambda$ and a rough path $\bX = (X,\X) \in \sV^p$, we write $\bX \circ \lambda := (X \circ \lambda, \X \circ (\lambda,\lambda))$ for the rough path obtained from $\bX$ by the ``time change'' $\lambda$.

As in \cite{FrizZhang2018}, given rough paths $\bX, \bZ \in \sV^p$, we define the $p$-variation J1-Skorokhod distance by
\begin{equation} \label{eq: defn skorokhod metric}
\sigma_{p,[0,T]}(\bX,\bZ) := \inf_{\lambda \in \Lambda} \big\{ |\lambda| \vee \|\bX \circ \lambda - \bZ\|_{p,[0,T]} \big\},
\end{equation}
where $|\lambda| := \sup_{t \in [0,T]} |\lambda(t) - t|$.

\begin{proposition}\label{prop: Skorokhod continuity for RSDEs}
Let $2 \leq q \leq p < 3$, $b \in C^1_b$, $\sigma \in C^1_b$, $f \in C^3_b$ and $g \colon [0,T] \times \R^m \times \U \to \R^m$. Let $\bX$ and $(\bX^n)_{n \in \N}$ be c\`adl\`ag rough paths, let $y_0 \in L^q$ be $\cF_0$-measurable, let $N$ be an integer-valued random measure with compensator $\nu$, such that $\nu$ and $g$ satisfy Assumption~\ref{assumption: Regularity of measure integral for RSDE}, and let $M \in V^p L^{q,\infty}$ be a c\`adl\`ag martingale which is continuous at deterministic times.

Let us write $Y$ and $Y^n$ for the solutions to the rough SDE \eqref{eq: RSDE with measure} with data $(y_0,M,N,\bX)$ and $(y_0,M,N,\bX^n)$ respectively. If $\sigma_{p,[0,T]}(\bX^n,\bX) \to 0$ as $n \to \infty$, then $Y^n_T \to Y_T$ in $L^q$ as $n \to \infty$.
\end{proposition}

\begin{remark}
We note that the assumption in Proposition~\ref{prop: Skorokhod continuity for RSDEs} that $M$ is continuous at deterministic times, and the assumption that the process $A$ in \eqref{eq: decomposition of nu} is continuous, are both necessary. Indeed, the result of Proposition~\ref{prop: Skorokhod continuity for RSDEs} is not true in general if either $M$ or $N$ has a positive probability of having a jump at any deterministic time $t \in (0,T]$. Essentially, in this case, time-changing the driving rough path $\bX$ can change the \emph{order of events}, which fundamentally alters the solution, violating the desired continuity, as illustrated in the following example. Of course, this assumption on $M$ is satisfied, for instance, by any centred L\'evy process.
\end{remark}

\begin{example}
Let $\xi$ be an $\cF_{t_0}$-measurable integrable random variable for some $t_0 \in (0,T)$, such that $\E [\xi \,|\, \cF_{t_0-}] = 0$, so that the process $M = (M_t)_{t \in [0,T]}$ given by $M_t := \xi \1_{[t_0,T]}(t)$ is a martingale. Further, let $X_t = \1_{[t_0,T]}(t)$ and $X^n_t = \1_{[t_0 - \frac{1}{n},T]}(t)$ so that, in particular, $X^n \to X$ in Skorokhod topology. Let $Y^n$ be the solution to the SDE
\begin{equation*}
\d Y^n_t = Y^n_{t-} \dd M_t + \d X^n_t,
\end{equation*}
and let $Y$ be the solution to the same SDE with $X^n$ replaced by $X$. One can then directly check that $Y_T = 1$, while $Y^n_T = 1 + \xi$ for every $n \in \N$, so that $Y^n_T$ does not converge to $Y_T$. This example illustrates that, if $M$ (or $N$) has a jump at a deterministic time, then we cannot expect the solution to a rough SDE to be continuous in the sense of Proposition~\ref{prop: Skorokhod continuity for RSDEs}.
\end{example}

\begin{proof}[Proof of Proposition~\ref{prop: Skorokhod continuity for RSDEs}]
We first note that, since $\Delta_{[0,T]}$ is compact, it follows from our assumptions that the function $\Delta_{[0,T]} \ni (s,t) \mapsto \|M\|_{p,q,[s,t]} + \|A\|_{\frac{p}{2},\frac{q}{2},[s,t]} + \|A\|_{\frac{p}{q},1,[s,t]}$ is uniformly continuous, so that, in particular,
\begin{equation}\label{eq: condition on martingale for Skorokhod continuity}
\lim_{h \searrow 0} \, \sup_{|t - s| < h} \big(\|M\|_{p,q,[s,t]} + \|A\|_{\frac{p}{2},\frac{q}{2},[s,t]} + \|A\|_{\frac{p}{q},1,[s,t]}\big) = 0.
\end{equation}

Let $C > 0$ be the Lipschitz constant in \eqref{eq: Lipschitz continuity of solution map}, which may be chosen such that \eqref{eq: Lipschitz continuity of solution map} holds for $\bX$ and $\bX^n$ for every $n \in \N$.

Since $\sigma_{p,[0,T]}(\bX^n,\bX) \to 0$ as $n \to \infty$, there exists a sequence $(\lambda_n)_{n \in \N} \subset \Lambda$ such that
\[|\lambda_n| \to 0 \qquad \text{and} \qquad \|\bX^n - \bX \circ \lambda_n\|_{p,[0,T]} \to 0 \quad \text{as} \quad n \to \infty.\]

Let $\epsilon > 0$, and let $Y^{(\lambda_n)}$ be the solution to the rough SDE \eqref{eq: RSDE with measure} with data $(y_0,M,N,\bX \circ \lambda_n)$. Then
\begin{equation}\label{eq: Y^n_T - Y^lambda_T est}
\|Y^n_T - Y^{(\lambda_n)}_T\|_{L^q} \leq \|Y^n - Y^{(\lambda_n)}\|_{p,q,[0,T]} \leq C \|\bX^n - \bX \circ \lambda_n\|_{p,[0,T]} < \epsilon
\end{equation}
for all sufficiently large $n$.

Given a partition $\cP = \{0 = t_0 < t_1 < \cdots < t_N = T\}$ of the interval $[0,T]$, we write $\bX^\cP$ for the piecewise constant approximation of $\bX$ along $\cP$. That is, $\bX^\cP = (X^\cP,\X^\cP)$, where $X^\cP_u = X_{t_i}$ and $\X^\cP_{u,v} = \X_{t_i,t_j}$ whenever $u \in [t_i,t_{i+1})$ and $v \in [t_j,t_{j+1})$.

It is straightforward to see\footnote{To see this, one can simply take a partition $\cP = \{t_i\}$ of $[0,T]$ such that $\|\bX\|_{p,[t_i,t_{i+1})}$ is small for each $i$, which allows one to make the uniform distance $\|\bX^\cP - \bX\|_{\infty,[0,T]}$ arbitrarily small, and then use interpolation of $p$-variation norms to obtain a bound on $\|\bX^\cP - \bX\|_{p',[0,T]}$ for $p' > p$.} that, for any $p' \in (p,3)$, there exists a partition $\cP = \{0 = t_0 < t_1 < \cdots < t_N = T\}$ of the interval $[0,T]$ such that
\[\|\bX^\cP - \bX\|_{p',[0,T]} < \frac{\epsilon}{C}.\]
Since $p$-variation is invariant under time changes, we have that $\|\bX^\cP \circ \lambda_n - \bX \circ \lambda_n\|_{p',[0,T]} < \frac{\epsilon}{C}$ also holds for each $n \in \N$. Writing $Y^\cP$ and $Y^{\lambda_n^{-1}(\cP)}$ for the solutions to the rough SDE with data $(y_0,M,N,\bX^\cP)$ and $(y_0,M,N,\bX^{\cP} \circ \lambda_n)$ respectively, we then have that
\begin{equation}\label{eq: Y^cP_T - Y_T estimate}
\|Y^\cP_T - Y_T\|_{L^q} \leq \|Y^\cP - Y\|_{p',q,[0,T]} \leq C \|\bX^\cP - \bX\|_{p',[0,T]} < \epsilon
\end{equation}
and, similarly,
\begin{equation}\label{eq: Y^lambdacP_T - Y^lambda_T estimate}
\|Y^{\lambda_n^{-1}(\cP)}_T - Y^{(\lambda_n)}_T\|_{L^q} \leq \|Y^{\lambda_n^{-1}(\cP)} - Y^{(\lambda_n)}\|_{p',q,[0,T]} \leq C \|\bX^{\cP} \circ \lambda_n - \bX \circ \lambda_n\|_{p',[0,T]} < \epsilon.
\end{equation}

We now take $n \in \N$ sufficiently large such that $|\lambda_n| < \frac{1}{2} \min_{0 \leq i < N} |t_{i+1} - t_i|$. This means in particular that $t_{i-1} \vee \lambda_n^{-1}(t_{i-1}) < t_i \wedge \lambda_n^{-1}(t_i)$ for each $i$. Let us consider the solutions $Y^\cP$ and $Y^{\lambda_n^{-1}(\cP)}$ of the rough SDE driven by the piecewise constant rough paths $\bX^\cP$ and $\bX^\cP \circ \lambda_n$ respectively. In particular, $\bX^\cP$ only jumps at the times $t_i$ for $i = 1, \ldots, N$, and $\bX^\cP \circ \lambda_n$ only jumps at the times $\lambda_n^{-1}(t_i)$ for $i = 1, \ldots, N$, and the jump sizes are given by
\[\Delta X^\cP_{t_i} = \Delta (X^\cP \circ \lambda_n)_{\lambda_n^{-1}(t_i)} = \delta X_{t_{i-1},t_i} \quad \text{and} \quad \Delta \X^{\cP}_{t_i} = \Delta (\X^{\cP} \circ (\lambda_n,\lambda_n))_{\lambda_n^{-1}(t_i)} = \X_{t_{i-1},t_i}.\]

For some $i$, let us suppose that $t_i < \lambda_n^{-1}(t_i)$. We then have that
\begin{align*}
Y^\cP_{\lambda_n^{-1}(t_i)} &= Y^\cP_{t_i-} + f(Y^\cP_{t_i-}) \delta X_{t_{i-1},t_i} + \D f(Y^\cP_{t_i-}) f(Y^\cP_{t_i-}) \X_{t_{i-1},t_i}\\
&\quad + \int_{t_i}^{\lambda_n^{-1}(t_i)} b(Y^\cP_s) \dd s + \int_{t_i}^{\lambda_n^{-1}(t_i)} \sigma(Y^\cP_s) \dd M_s + \int_{t_i}^{\lambda^{-1}_n(t_i)} \int_{\U} g(s,Y^\cP_s,u) \, \tN(\d s,\d u)
\end{align*}
and
\begin{align*}
&Y^{\lambda_n^{-1}(\cP)}_{\lambda_n^{-1}(t_i)} = Y^{\lambda_n^{-1}(\cP)}_{t_i-} + f\big(Y^{\lambda_n^{-1}(\cP)}_{\lambda_n^{-1}(t_i)-}\big) \delta X_{t_{i-1},t_i} + \D f\big(Y^{\lambda_n^{-1}(\cP)}_{\lambda_n^{-1}(t_i)-}\big) f\big(Y^{\lambda_n^{-1}(\cP)}_{\lambda_n^{-1}(t_i)-}\big) \X_{t_{i-1},t_i}\\
&+ \int_{t_i}^{\lambda_n^{-1}(t_i)} b(Y^{\lambda_n^{-1}(\cP)}_s) \dd s + \int_{t_i}^{\lambda_n^{-1}(t_i)} \sigma(Y^{\lambda_n^{-1}(\cP)}_s) \dd M_s + \int_{t_i}^{\lambda^{-1}_n(t_i)} \int_{\U} g(s,Y^{\lambda_n^{-1}(\cP)}_s,u) \, \tN(\d s,\d u).
\end{align*}

By the BDG inequality and \cite[Lemma~2.12]{AllanPieper2026}, it is straightforward to see that
\begin{equation*}
\bigg\|\int_{t_i}^{\lambda_n^{-1}(t_i)} b(Y^\cP_s) \dd s + \int_{t_i}^{\lambda_n^{-1}(t_i)} \sigma(Y^\cP_s) \dd M_s\bigg\|_{L^q} \lesssim (\lambda_n^{-1}(t_i) - t_i) + \|M\|_{p,q,[t_i,\lambda_n^{-1}(t_i)]},
\end{equation*}
and similarly, applying part~(ii) of Lemma~\ref{lemma: application of conditional BDG for jump measures} with $r = q$, that
\begin{equation*}
\bigg\|\int_{t_i}^{\lambda^{-1}_n(t_i)} \int_{\U} g(s,Y^\cP_s,u) \, \tN(\d s,\d u)\bigg\|_{L^q} \lesssim \|A\|_{\frac{p}{2},\frac{q}{2},[t_i,\lambda^{-1}_n(t_i)]}^{\frac{1}{2}} + \|A\|_{\frac{p}{q},1,[t_i,\lambda^{-1}_n(t_i)]}^{\frac{1}{q}}.
\end{equation*}
Clearly, the same holds true with $Y^\cP$ replaced by $Y^{\lambda_n^{-1}(\cP)}$, and, by the condition in \eqref{eq: condition on martingale for Skorokhod continuity}, the right-hand side above can be made arbitrarily small by taking $n$ sufficiently large.

We also have that
\begin{align*}
&\big\|f(Y^\cP_{t_i-}) \delta X_{t_{i-1},t_i} + \D f(Y^\cP_{t_i-}) f(Y^\cP_{t_i-}) \X_{t_{i-1},t_i}\\
&\quad - f\big(Y^{\lambda_n^{-1}(\cP)}_{\lambda_n^{-1}(t_i)-}\big) \delta X_{t_{i-1},t_i} - \D f\big(Y^{\lambda_n^{-1}(\cP)}_{\lambda_n^{-1}(t_i)-}\big) f\big(Y^{\lambda_n^{-1}(\cP)}_{\lambda_n^{-1}(t_i)-}\big) \X_{t_{i-1},t_i}\big\|_{L^q}\\
&\lesssim \big\|Y^\cP_{t_i-} - Y^{\lambda_n^{-1}(\cP)}_{\lambda_n^{-1}(t_i)-}\big\|_{L^q} \big(|\delta X_{t_{i-1},t_i}| + |\X_{t_{i-1},t_i}|\big)\\
&\leq \bigg\|Y^\cP_{t_i-} - Y^{\lambda_n^{-1}(\cP)}_{t_i-} - \int_{t_i}^{\lambda_n^{-1}(t_i)} b(Y^{\lambda_n^{-1}(\cP)}_s) \dd s\\
&\qquad - \int_{t_i}^{\lambda_n^{-1}(t_i)} \sigma(Y^{\lambda_n^{-1}(\cP)}_s) \dd M_s - \int_{t_i}^{\lambda^{-1}_n(t_i)} \int_{\U} g(s,Y^{\cP}_s,u) \, \tN(\d s,\d u)\bigg\|_{L^q} \|\bX\|_{p,[0,T]}.
\end{align*}
It follows that
\begin{equation*}
\big\|Y^\cP_{\lambda_n^{-1}(t_i)} - Y^{\lambda_n^{-1}(\cP)}_{\lambda_n^{-1}(t_i)}\big\|_{L^q} \leq c \big\|Y^\cP_{t_i-} - Y^{\lambda_n^{-1}(\cP)}_{t_i-}\big\|_{L^q} + \eta,
\end{equation*}
where $c > 0$ is a constant which depends only on $\|f\|_{C^3_b}$ and $\|\bX\|_{p,[0,T]}$, and $\eta > 0$ is a constant which can be chosen arbitrarily small by taking $n$ sufficiently large. The cases when $t_i > \lambda_n^{-1}(t_i)$ and when $t_i = \lambda_n^{-1}(t_i)$ may be treated similarly, and hence, in every case, we have that
\begin{equation}\label{eq: bound with X and eta}
\big\|Y^\cP_{t_i \vee \lambda_n^{-1}(t_i)} - Y^{\lambda_n^{-1}(\cP)}_{t_i \vee \lambda_n^{-1}(t_i)}\big\|_{L^q} \leq c \big\|Y^\cP_{(t_i \wedge \lambda_n^{-1}(t_i))-} - Y^{\lambda_n^{-1}(\cP)}_{(t_i \wedge \lambda_n^{-1}(t_i))-}\big\|_{L^q} + \eta.
\end{equation}
Since $\bX^\cP$ and $\bX^\cP \circ \lambda_n$ are constant on the interval $[t_{i-1} \vee \lambda_n^{-1}(t_{i-1}),t_i \wedge \lambda_n^{-1}(t_i))$ for each $i$, by again using the stability of solutions to rough SDEs, we can simply estimate
\begin{equation}\label{eq: bound with C}
\big\|Y^\cP_{(t_i \wedge \lambda_n^{-1}(t_i))-} - Y^{\lambda_n^{-1}(\cP)}_{(t_i \wedge \lambda_n^{-1}(t_i))-}\big\|_{L^q} \leq C \big\|Y^\cP_{t_{i-1} \vee \lambda_n^{-1}(t_{i-1})} - Y^{\lambda_n^{-1}(\cP)}_{t_{i-1} \vee \lambda_n^{-1}(t_{i-1})}\big\|_{L^q}.
\end{equation}

Since $Y^\cP_0 - Y^{\lambda_n^{-1}(\cP)}_0 = 0$, by combining the estimates in \eqref{eq: bound with X and eta} and \eqref{eq: bound with C} for each $i = 1, \ldots, N$, we deduce that $\|Y^\cP_T - Y^{\lambda_n^{-1}(\cP)}_T\|_{L^q} \lesssim \eta$, where the implicit multiplicative constant depends only on $C$, $\|f\|_{C^3_b}$ and $\|\bX\|_{p,[0,T]}$. Thus, by choosing $n$ sufficiently large, we can ensure that $\eta$ is sufficiently small to guarantee that $\|Y^\cP_T - Y^{\lambda_n^{-1}(\cP)}_T\|_{L^q} < \epsilon$. By combining this with \eqref{eq: Y^n_T - Y^lambda_T est}, \eqref{eq: Y^cP_T - Y_T estimate} and \eqref{eq: Y^lambdacP_T - Y^lambda_T estimate}, we deduce that $\|Y^n_T - Y_T\|_{L^q} < 4 \epsilon$ for all sufficiently large $n$, so that $Y^n_T \to Y_T$ in $L^q$ as $n \to \infty$.
\end{proof}

\begin{remark}\label{remark: Skorokhod continuity for rough stochastic integrals}
Using the same approach, the proof of Proposition~\ref{prop: Skorokhod continuity for RSDEs} can be extended slightly to also obtain the convergence of rough stochastic integrals with respect to the Skorokhod topology. More precisely, under the assumptions of Proposition~\ref{prop: Skorokhod continuity for RSDEs}, for any $h \in C^2_b$, if $\sigma_{p,[0,T]}(\bX^n,\bX) \to 0$ as $n \to \infty$, then $\int_0^T h(Y^n_s) \dd \bX^n_s \to \int_0^T h(Y_s) \dd \bX_s$ in $L^q$ as $n \to \infty$.
\end{remark}

\subsection{A Kolmogorov criterion for c\`adl\`ag processes}

We conclude this section with a variant of the Kolmogorov continuity criterion adapted to the setting of discontinuous paths with finite $p$-variation. Such a result appears to be unexplored in the rough path literature, but the authors are aware of an analogous result in a discrete-time setting in \cite[Corollary~12.16]{Pisier2016}. It is claimed in \cite{PisierXu1987} that a similar result holds in continuous-time, but we are not aware of any publicly available sources for such a result. The corresponding result for continuous paths\footnote{Although the result in \cite{Koch2014} includes certain discontinuous paths, it does not appear to apply to general càdlàg paths.} is already contained in corresponding Besov-space embeddings; see, e.g., \cite[Proposition~2.3]{FrizSeegerKranich2022} or \cite[Lemma 5]{Koch2014}. Using these embeddings, the authors of \cite[Proposition~3.9]{ChevyrevFrizKorepanovMelbourneZhang2022} establish a Kolmogorov-type criterion for piecewise constant random rough paths with finitely many deterministic jumps. Our result allows one to extend this result to the general c\`adl\`ag setting, without requiring either finitely many deterministic jumps or a loss of regularity depending on the integrability assumptions.

We emphasize that our proof is much more constructive and akin to classical treatments of $p$-variation than the ones in the previously mentioned works. In particular, we do not resort to interpolation spaces or duality-type arguments, but instead work directly at the level of paths and controls. One may also interpret the result as a form of Minkowski's integral inequality for $p$-variation paths in Lebesgue spaces.

\begin{theorem}\label{theorem: a p-variation kolmogorov result}
Let $1 \leq p < q < \infty$, and let $Y$ be a stochastic process which is right-continuous in probability and such that $\|Y\|_{p,q,[0,T]} < \infty$. Then there exists a modification of $Y$ with sample paths which are almost surely c\`adl\`ag and such that, for every $\tp > p$, we have that
\begin{equation}\label{eq: Kolmogorov estimate}
\big\| \|Y\|_{\tp,[0,T]} \big\|_{L^q} \leq C \|Y\|_{p,q,[0,T]},
\end{equation}
where the constant $C$ depends only on $p$ and $\tp$.

Conversely, if $1 \leq q \leq p < \infty$, and if $Y$ is any process with almost surely c\`adl\`ag sample paths, then
\begin{equation*}
\|Y\|_{p,q,[0,T]} \leq \big\| \|Y\|_{p,[0,T]} \big\|_{L^q}.
\end{equation*}
\end{theorem}

The proof of Theorem~\ref{theorem: a p-variation kolmogorov result} is given in Appendix~\ref{section: Kolmogorov result}. We note that both of the conditions $p < q$ and $\tp > p$ in \eqref{eq: Kolmogorov estimate} are sharp. Indeed, in Example~\ref{example: p<q is sharp for kolmogorov} we provide a counterexample for when either of these conditions is not satisfied.

\section{Consistency between doubly stochastic DEs and rough SDEs}\label{sec: Consistency RSDEs and SDEs}

Throughout this section, we let $(\Omega, \cF, (\cF_t)_{t \in [0,T]}, \P)$ and $(\bar{\Omega}, \bar{\cF}, (\bar{\cF}_t)_{t \in [0,T]}, \bar{\P})$ be two filtered probability spaces, and denote by
\[(\hat{\Omega}, \hat{\cF}, (\hat{\cF}_t)_{t \in [0,T]}, \hat{\P}) := (\Omega \times \bar{\Omega}, \cF \otimes \bar{\cF}, (\cF_t \otimes \bar{\cF}_t)_{t \in [0,T]}, \P \otimes \bar{\P})\]
the corresponding filtered product probability space.\footnote{We implicitly assume that $\hat{\cF}_t$ has been suitably augmented, so that $(\hat{\cF}_t)_{t \in [0,T]}$ satisfies the usual conditions.} In particular, $\hat{\Omega}$ indicates that we are considering a product of two spaces (bringing them under one roof). In contrast to many probabilistic results, we actually do need this inner structure of the probability space, since on the one hand we are dealing with It\^{o} rough path lifts of semimartingales (whose randomness lives on $\Omega$), and on the other hand we have additional randomness (which lives on $\bar{\Omega}$), and these need to be separated and handled with care.

In the following, we will write, e.g., $\|\cdot\|_{q,r,s,\bar{\Omega}}$ for the norm $\|\cdot\|_{q,r,s}$ when we wish to emphasize that the norm is taken with respect to the probability space $(\bar{\Omega}, \bar{\cF}, (\bar{\cF}_t)_{t \in [0,T]}, \bar{\P})$. Similarly, we will write $\cV^{p,q,r,\bar{\Omega}}_X$ for the space of stochastic controlled paths (in the sense of Definition~\ref{defn: stochastic controlled path}) defined with respect to $(\bar{\Omega}, \bar{\cF}, (\bar{\cF}_t)_{t \in [0,T]}, \bar{\P})$.

Further, for any process $Y = (Y_t)_{t \in [0,T]}$ defined on $(\hat{\Omega}, \hat{\cF}, \hat{\P})$, and any two $\cF$-measurable random times $\tau_1, \tau_2$ on $(\Omega, \cF, \P)$ such that $\tau_1 \leq \tau_2$ holds $\P$-almost surely, we define $\|Y\|_{p,q,\infty, \llbracket \tau_1, \tau_2 \rrbracket, \bar{\Omega}}$ as the map from $\Omega \to [0,\infty]$ given by
\begin{equation}\label{eq: defn p,q,infty tau1 tau2 barOmega norm}
\Omega \ni \omega \mapsto \|Y(\omega, \cdot)\|_{p,q,\infty, [\tau_1(\omega), \tau_2(\omega)], \bar{\Omega}},
\end{equation}
and, with a slight abuse of notation, similarly define $\sup_{s \in \llbracket \tau_1, \tau_2 \rrbracket} \|Y_s\|_{L^r(\bar{\Omega})}$ as the map given by
\begin{equation}\label{eq: defn sup Ys L^r norm}
\Omega \ni \omega \mapsto \sup_{s \in [\tau_1(\omega),\tau_2(\omega)]} \|Y_s(\omega, \cdot)\|_{L^r(\bar{\Omega})}.
\end{equation}

Although it is not immediate that the maps defined in \eqref{eq: defn p,q,infty tau1 tau2 barOmega norm} and \eqref{eq: defn sup Ys L^r norm} are $\cF$-measurable, and thus well-defined random variables, this is indeed the case, as stated precisely in the following lemma. The proofs of this and the following lemma are straightforward consequences of standard measurability and Fubini-type arguments, and are therefore omitted.

\begin{lemma}\label{lemma: measurability on the product space of the p,q,r norms}
Let $p, q \in [1,\infty)$, $r \in [q,\infty]$, $\nu > 0$, and let $\tau_1, \tau_2 \colon \Omega \to [0,T]$ be $\cF$-measurable random times such that $\tau_1 \leq \tau_2$. The following then hold.
\begin{itemize}
\item[(i)] If $Z$ is an $\hat{\cF}$-measurable random variable such that $\|Z(\omega, \cdot)\|_{q,\infty,s,\bar{\Omega}} \leq \nu$ for $\P$-almost every $\omega \in \Omega$, then $\|Z\|_{q,r,s,\bar{\Omega}}$ is an $\cF$-measurable random variable.
\item[(ii)] If $Y = (Y_t)_{t \in [0,T]}$ is a $\hat{\P}$-almost surely c\`adl\`ag $\hat{\cF} \otimes \cB([0,T])$-measurable process, such that $\|Y(\omega, \cdot)\|_{p,q,\infty,[0,T],\bar{\Omega}} \leq \nu$ for $\P$-almost every $\omega \in \Omega$, then $\|Y\|_{p,q,r,\llbracket \tau_1, \tau_2 \rrbracket,\bar{\Omega}}$ is an $\cF$-measurable random variable.
\item[(iii)] Let $Y$ be as in case (ii) above. Then $\| 1 \wedge \sup_{t \in \llbracket \tau_1, \tau_2 \llbracket } |Y_t| \|_{L^q(\bar{\Omega})}$ is also $\cF$-measurable.
\item[(iv)] Let $Y$ be as in case (ii) above. If, in addition, $\sup_{t \in [0,T]} \|Y_t(\omega, \cdot)\|_{L^{\infty}(\bar{\Omega})} < \infty$ for $\P$-almost every $\omega \in \Omega$, then $\sup_{t \in \llbracket \tau_1, \tau_2 \llbracket} \|Y_t\|_{L^{\infty}(\bar{\Omega})}$ is also $\cF$-measurable.
\end{itemize}
\end{lemma}

Similarly, one can also show the following.

\begin{lemma}\label{lemma: measurability on the product space of the remainder}
For some $p \in [2,3)$, $q \in [2,\infty)$ and some constant $\nu > 0$, suppose that $X$ is an $\cF \otimes \cB([0,T])$-measurable process such that $X(\omega) \in V^p$ for $\P$-almost every $\omega \in \Omega$, and that $(Y,Y')$ is a pair of $\hat{\cF} \otimes \cB([0,T])$-measurable processes, such that $(Y(\omega, \cdot),Y'(\omega, \cdot)) \in \cV^{p,q,\infty,\bar{\Omega}}_{X(\omega)}$ with $\|(\bar{\E}_{\edot} R^Y)(\omega, \cdot)\|_{\frac{p}{2},\infty,[0,T],\bar{\Omega}} \leq \nu$ for $\P$-almost every $\omega \in \Omega$.

Then, for any two $\cF$-measurable random times $\tau_1 \leq \tau_2$, and any $r \in [q,\infty]$, we have that $\|\bar{\E}_{\edot} R^Y\|_{\frac{p}{2},r,\llbracket \tau_1, \tau_2 \rrbracket,\bar{\Omega}}$ is an $\cF$-measurable random variable.
\end{lemma}

\subsection{A suitable metric space}

For any two stopping times $\tau_1 \leq \tau_2$, and any process $X$, we let $X^{(\tau_1, \tau_2-)} := (X - X^{\tau_1})^{\tau_2-}$, so that $X^{(\tau_1,\tau_2-)} = 0$ on the event $\{\tau_1 = \tau_2\}$, and on $\{\tau_1 < \tau_2\}$ we have that
\begin{align*}
X^{(\tau_1, \tau_2-)}_t := \begin{cases}
0 & \text{ if~~} t < \tau_1,\\
\delta X_{\tau_1, t} & \text{ if~~} \tau_1 \leq t < \tau_2,\\
\delta X_{\tau_1, \tau_2-} & \text{ if~~} t \geq \tau_2.
\end{cases}
\end{align*}
We note that this construction is analogous to that used in $\alpha$-slicing for classical It\^o SDEs; cf.~\cite[Definition~16.3.8]{Cohen2015}. Further, if $X$ is a semimartingale, we let $\bX^{(\tau_1, \tau_2-)} = (X^{(\tau_1, \tau_2-)}, \X^{(\tau_1, \tau_2-)})$ be the It\^o rough path lift of $X^{(\tau_1, \tau_2-)}$, i.e.,
\begin{equation*}
\X^{(\tau_1, \tau_2-)}_{s,t} := \int_s^t \delta X^{(\tau_1, \tau_2-)}_{s,u-} \otimes \d X^{(\tau_1, \tau_2-)}_u
\end{equation*}
for $(s,t) \in \Delta_{[0,T]}$.

\begin{remark}\label{remark: consistency of slicing roughs stochastic integrals}
We note that, if $X$ is a semimartingale on $(\Omega, \cF, \P)$, and $\bX = (X,\X)$ is its It\^o rough path lift, and if $\tau_1 \leq \tau_2$ are $\cF_t$-stopping times, then for $\P$-almost any $\omega \in \Omega$ and any $(s,t) \in \Delta_{[0,T]}$, we have that
\begin{equation*}
\X_{s,t}^{(\tau_1, \tau_2-)}(\omega) = \begin{cases}
0 , & \text{ if~~} 0 \leq t < \tau_1(\omega),\\
\X_{\tau_1(\omega) \vee s, t}(\omega), & \text{ if~~} \tau_1(\omega) \leq t < \tau_2(\omega),\\
\X_{\tau_1(\omega) \vee s, \tau_2(\omega)-}(\omega) & \text{ if~~} s < \tau_2(\omega) \leq t \leq T,\\
0 & \text{ if~~} \tau_2(\omega) \leq s \leq T.
\end{cases}
\end{equation*}
Moreover, if $(Y,Y')$ is a pair of c\`adl\`ag $\hat{\cF}_t$-adapted processes, then we see that the condition
$$(Y(\omega, \cdot), Y'(\omega, \cdot)) \in \cV_{X^{(\tau_1, \tau_2-)}(\omega)}^{p,q,\infty,\bar{\Omega}}([0,T])$$
is equivalent to
$$(Y(\omega, \cdot), Y'(\omega, \cdot)) \in \cV^{p,q,\infty,\bar{\Omega}}_{X(\omega)}([\tau_1(\omega), \tau_2(\omega))).$$
Thus, if these conditions hold for $\P$-almost every $\omega \in \Omega$, then for almost every $\omega \in \Omega$ and every $t \in [\tau_1(\omega),\tau_2(\omega))$, the rough stochastic integrals
\begin{equation*}
\int_0^t Y_s(\omega, \cdot) \dd \bX^{(\tau_1, \tau_2-)}_s(\omega) = \int_{\tau_1(\omega)}^t Y_s(\omega, \cdot) \dd \bX_s(\omega)
\end{equation*}
coincide. Similarly, for any (semi)martingale $M$ and predictable process $Z$, we have that
\begin{equation*}
\int_0^t Z_s \dd M^{(\tau_1, \tau_2-)}_s = \bigg(\int_0^t Z_s \dd M_s\bigg)^{\hspace{-2pt}(\tau_1, \tau_2-)}
\end{equation*}
as well as $[M^{(\tau_1, \tau_2-)}]_t = [M]_t^{(\tau_1, \tau_2-)}$.
\end{remark}

Let $X$ be a process defined on $(\Omega, \cF, \P)$ such that $X(\omega) \in V^p$ for $\P$-almost every $\omega \in \Omega$. In the following we will consider pairs of processes $(Y,Y')$ defined on the product space $(\hat{\Omega},\hat{\cF},\hat{\P})$, such that $(Y(\omega, \cdot), Y'(\omega, \cdot)) \in \cV^{p,q,\infty,\bar{\Omega}}_{X^{(\tau_1, \tau_2-)}(\omega)}([\tau_1(\omega),\tau_2(\omega)])$ for $\P$-almost every $\omega \in \Omega$, for some $\cF_t$-stopping times $\tau_1 \leq \tau_2$. For any two such pairs $(Y,Y')$, $(\tY,\tY')$, and some $\eta > 1$, we define
\begin{equation}\label{eq: definition of metric on B_T for contraction}
\begin{split}
d_{\tau_1, \tau_2}^{\eta, \bar{\Omega}} \big( (Y,&Y'), (\tY,\tY') \big) := \|Y' - \tY'\|_{p,q,\llbracket \tau_1, \tau_2 \rrbracket,\bar{\Omega}} + \Big\| 1 \wedge \sup_{u \in \llbracket \tau_1, \tau_2 \rrbracket} |Y'_u - \tY'_u| \Big\|_{L^q(\bar{\Omega})}\\
&+ \eta \Big( \|Y - \tY\|_{p,q,\llbracket \tau_1, \tau_2 \rrbracket,\bar{\Omega}} + \|\bar{\E}_{\edot} (R^Y - R^{\tY})\|_{\frac{p}{2},q,\llbracket \tau_1, \tau_2 \rrbracket,\bar{\Omega}} + \Big\| 1 \wedge \sup_{u \in \llbracket \tau_1, \tau_2 \rrbracket} |Y_u - \tY_u| \Big\|_{L^q(\bar{\Omega})} \Big).
\end{split}
\end{equation}

\begin{remark}
We note that this almost coincides with the metric defined in the proof of \cite[Theorem~5.2]{AllanPieper2026}, except that here we additionally include the term $\| 1 \wedge \sup_{u \in \llbracket \tau_1, \tau_2 \rrbracket} |Y'_u - \tY'_u| \|_{L^q(\bar{\Omega})}$. This term is included to ensure that limits under this metric are almost surely c\`adl\`ag, which is a necessary condition in Lemma~\ref{lemma: measurability on the product space of the p,q,r norms} to ensure the measurability of $\sup_{u \in \llbracket \tau_1, \tau_2 \llbracket} \|Y'_u\|_{L^\infty(\bar{\Omega})}$.
\end{remark}

Let $\tau_1 \leq \tau_2$ be $\cF_t$-stopping times, and let $\xi, \xi'$ be $\hat{\cF}_{\tau_1}$-measurable random variables. Let $X = (X_t)_{t \in [0,T]}$ be a process defined on $(\Omega, \cF, \P)$ which is adapted to $(\cF_t)_{t \in [0,T]}$, and such that $X(\omega) \in V^p$ for $\P$-almost every $\omega \in \Omega$. For constants $\nu_1, \nu_2 > 0$, we write
$$\mathbf{B}_{\tau_1, \tau_2}(\nu_1, \nu_2, \xi, \xi')$$
for the set of all pairs of processes $(Y,Y')$ such that
\begin{itemize}
\item $Y$ and $Y'$ are both $\hat{\cF}_t$-adapted and $\hat{\P}$-almost surely c\`adl\`ag,
\item $\hat{\P}$-almost surely, we have that $(Y,Y') = (0,0)$ on $\llbracket 0, \tau_1 \llbracket$, with $(Y_{\tau_1},Y'_{\tau_1}) = (\xi,\xi')$, the processes $Y$ and $Y'$ are continuous at $\tau_2$, and they are constant on $\llbracket \tau_2, T \rrbracket$,
\item and we have that $(Y(\omega, \cdot), Y'(\omega, \cdot)) \in \cV^{p,q,\infty,\bar{\Omega}}_{X^{(\tau_1, \tau_2-)}(\omega)}([\tau_1(\omega),\tau_2(\omega)])$, and satisfy
\begin{equation*}
\|Y'(\omega, \cdot)\|_{p,q,\infty,[0,T],\bar{\Omega}} \vee \sup_{u \in [0,T]} \|Y'_u(\omega, \cdot)\|_{L^\infty(\bar{\Omega})} \leq \nu_1
\end{equation*}
and
\begin{equation*}
\|Y(\omega, \cdot)\|_{p,q,\infty,[0,T],\bar{\Omega}} \vee \|\bar{\E}_{\edot} R^Y(\omega, \cdot)\|_{\frac{p}{2},\infty,[0,T],\bar{\Omega}} \leq \nu_2
\end{equation*}
for $\P$-almost every $\omega \in \Omega$.
\end{itemize}
For any $\eta > 1$ and any $(Y,Y'), (\tY,\tY') \in \mathbf{B}_{\tau_1, \tau_2}(\nu_1, \nu_2, \xi, \xi')$, we define
\begin{equation*}
\mathbf{d}^\eta_{\tau_1, \tau_2} \big( (Y,Y'), (\tY,\tY') \big) := \big\| d^{\eta, \bar{\Omega}}_{\tau_1, \tau_2} \big( (Y, Y'), (\tY, \tY') \big) \big\|_{L^\infty(\Omega)}.
\end{equation*}

\begin{lemma}\label{lemma: B_tau is a closed space with respect to hat d}
Let $X = (X_t)_{t \in [0,T]}$ be a process defined on $(\Omega, \cF, \P)$ which is $\cF_t$-adapted, and such that $X(\omega) \in V^p$ for $\P$-almost every $\omega \in \Omega$. Let $\tau_1 \leq \tau_2$ be $\cF_t$-stopping times, let $\xi, \xi'$ be $\hat{\cF}_{\tau_1}$-measurable random variables, and let $\nu_1, \nu_2 > 0$ and $\eta > 1$.

Then $\mathbf{B}_{\tau_1, \tau_2}(\nu_1, \nu_2, \xi, \xi')$ and $\mathbf{d}^\eta_{\tau_1, \tau_2}$ are both well-defined, and $(\mathbf{B}_{\tau_1, \tau_2}(\nu_1, \nu_2, \xi, \xi'), \mathbf{d}^\eta_{\tau_1, \tau_2})$ is a complete metric space.
\end{lemma}

To see that $\mathbf{B}_{\tau_1, \tau_2}(\nu_1, \nu_2, \xi, \xi')$ and $\mathbf{d}^\eta_{\tau_1, \tau_2}$ are well-defined is simply a matter of noting that all the norms involved are $\cF$-measurable by Lemmas~\ref{lemma: measurability on the product space of the p,q,r norms} and \ref{lemma: measurability on the product space of the remainder}, so that it is then valid to take the $L^\infty(\Omega)$ norm. Verifying that $(\mathbf{B}_{\tau_1, \tau_2}(\nu_1, \nu_2, \xi, \xi'), \mathbf{d}^\eta_{\tau_1, \tau_2})$ is a complete metric space then follows standard arguments, similar to the proofs of \cite[Lemmas~2.9 and 2.10]{AllanPieper2026}. In particular, if a sequence is Cauchy with respect to $\mathbf{d}^\eta_{\tau_1, \tau_2}$, then it is also Cauchy in the u.c.p.~topology on $(\hat{\Omega}, \hat{\cF}, \hat{\P})$, so that the existence and measurability of the limit follow immediately. The full proof of Lemma~\ref{lemma: B_tau is a closed space with respect to hat d} is omitted for brevity.

\subsection{The consistency result}

The following result establishes consistency between It\^o stochastic integrals and rough stochastic integrals against It\^o rough path lifts of semimartingales. It will be unsurprising to readers familiar with rough analysis, but we nonetheless provide a proof in Appendix~\ref{appendix: rough paths}.

\begin{proposition}\label{proposition: consistency rough and stochastic integrals for semimartingales}
Let $p \in (2,3)$, $q \in [2,\infty)$ and $r \in [q,\infty]$, and let $(\Omega \times \bar{\Omega}, \cF \otimes \bar{\cF}, \P \otimes \bar{\P})$ be a product probability space which satisfies the usual conditions. Let $X$ be a c\`adl\`ag semimartingale on $(\Omega, \cF, \P)$, and let $\bX = (X,\X)$ be its It\^o rough path lift, so that $\X_{s,t} = \int_s^t \delta X_{s,u-} \otimes \d X_u$ for every $(s,t) \in \Delta_{[0,T]}$, and $\bX(\omega) \in \sV^p$ for $\P$-almost every $\omega \in \Omega$. We note that $X$ is also a semimartingale on the product space (by identifying $X_t(\omega, \bar{\omega}) = X_t(\omega)$ for all $(\omega,\bar{\omega}) \in \Omega \times \bar{\Omega}$ and $t \in [0,T]$). Let $Y$ and $Y'$ be adapted processes on $(\Omega \times \bar{\Omega}, \cF \otimes \bar{\cF}, \P \otimes \bar{\P})$ which are almost surely c\`adl\`ag, and suppose that $(Y(\omega, \cdot),Y'(\omega, \cdot)) \in \cV_{X(\omega)}^{p,q,r,\bar{\Omega}}$ for $\P$-almost every $\omega \in \Omega$.

Then, for $\P$-almost every $\omega \in \Omega$, the processes
\begin{equation*}
\bigg( \int_0^\cdot Y_{u-} \dd X_u \bigg)(\omega, \cdot) = \int_0^\cdot Y_u(\omega, \cdot) \dd \bX_u(\omega),
\end{equation*}
are $\bar{\P}$-indistinguishable, where on the left-hand side is the It\^o integral of $Y_-$ against $X$, and on the right-hand side is the rough stochastic integral of $(Y(\omega, \cdot),Y'(\omega, \cdot))$ against $\bX(\omega)$.
\end{proposition}

The following theorem is the main result of this section, and shows that the solution to a \emph{doubly stochastic DE} is also the solution to the corresponding (random) \emph{rough SDE}. The basic strategy is to construct a contraction mapping on the metric space in Lemma~\ref{lemma: B_tau is a closed space with respect to hat d}.

\begin{remark}\label{remark: consistency also works with random measures}
For simplicity of notation, here we present the result without including an integral against a random measure. However, using Assumption~\ref{assumption: Regularity of measure integral for RSDE} and Lemma~\ref{lemma: application of conditional BDG for jump measures}, it is straightforward to generalize the proof of Theorem~\ref{theorem: consistency between doubly stochastic SDEs and RSDEs} to show that the result also holds when an integral against a compensated random measure is included in the equation, as in \eqref{eq: RSDE with measure}.
\end{remark}

\begin{theorem} \label{theorem: consistency between doubly stochastic SDEs and RSDEs}
Let $p \in (2,3)$, $q \in [2,\infty)$, $b \in C^1_b$, $\sigma \in C^1_b$ and $f \in C^3_b$, and let
\begin{equation*}
(\hat{\Omega}, \hat{\cF}, (\hat{\cF}_t)_{t \in [0,T]}, \hat{\P}) = (\Omega \times \bar{\Omega}, \cF \otimes \bar{\cF}, (\cF_t \otimes \bar{\cF}_t)_{t \in [0,T]}, \P \otimes \bar{\P})
\end{equation*}
be a product probability space which satisfies the usual conditions.

Let $y_0$ be an $\hat{\cF}_0$-measurable random variable such that $\|y_0(\omega, \cdot)\|_{L^q(\bar{\Omega})} < \infty$ for $\P$-almost every $\omega \in \Omega$. Let $X$ be an $\cF_t$-adapted c\`adl\`ag semimartingale on $(\Omega, \cF, \P)$, and let $M \in V^p L^{q,\infty}(\bar{\Omega})$ be an $\bar{\cF}_t$-adapted c\`adl\`ag martingale on $(\bar{\Omega}, \bar{\cF}, \bar{\P})$.

Let $Y$ be the unique strong solution to the SDE
\begin{equation}\label{eq: the doubly stochastic SDE}
Y_t = y_0 + \int_0^t b(Y_s) \dd s + \int_0^t \sigma(Y_{s-}) \dd M_s + \int_0^t f(Y_{s-}) \dd X_s, \qquad t \in [0,T],
\end{equation}
on $(\hat{\Omega}, \hat{\cF}, \hat{\P})$, which in particular is $\hat{\cF}_t$-adapted and has $\hat{\P}$-almost surely c\`adl\`ag sample paths.

Then, for $\P$-almost every $\omega \in \Omega$, the pair $(Y(\omega, \cdot),f(Y(\omega, \cdot))) \in \cV^{p,q,\infty,\bar{\Omega}}_{X(\omega)}$ is a stochastic controlled path relative to $X(\omega)$, and $Y(\omega, \cdot)$ is the solution to the rough SDE
\begin{equation}\label{eq: the corresponding RSDE}
Y_t(\omega, \cdot) = y_0(\omega, \cdot) + \int_0^t b(Y_s(\omega, \cdot)) \dd s + \int_0^t \sigma(Y_{s-}(\omega, \cdot)) \dd M_s + \int_0^t f(Y_s(\omega, \cdot)) \dd \bX_s(\omega)
\end{equation}
for $t \in [0,T]$, where $\bX$ is the It\^o rough path lift of $X$.
\end{theorem}

\begin{proof}
We let $(\tau_i)_{i \in \N}$ be the sequence of $\cF_t$-stopping times defined recursively by $\tau_0 = 0$, and
\begin{equation*}
\tau_{i+1} = T \wedge \inf \big\{ t > \tau_i \, \big| \, (t - \tau_i) + \|M\|_{p,q,\infty,\llbracket \tau_i, t \rrbracket,\bar{\Omega}} + \|\bX\|_{p,\llbracket \tau_i, t \rrbracket} \geq \alpha \big\}
\end{equation*}
for each $i \in \N$, where $\alpha \in (0,1]$ is a constant which we will specify later.

Let us fix an $i \in \N$. In particular, we then have that
\begin{equation}\label{eq: small interval bound}
(\tau_{i+1} - \tau_i) + \|M^{(\tau_i, \tau_{i+1}-)}\|_{p,q,\infty,\llbracket \tau_i, \tau_{i+1} \rrbracket,\bar{\Omega}} + \|\bX^{(\tau_i, \tau_{i+1}-)}\|_{p,\llbracket \tau_i, \tau_{i+1} \rrbracket} \leq \alpha
\end{equation}
$\P$-almost surely. We also let $\xi_i$ be an $\hat{\cF}_{\tau_i}$-measurable random variable such that $\xi_i(\omega, \cdot) \in L^q(\bar{\Omega})$ for $\P$-almost every $\omega \in \Omega$.

For $(Y,Y') \in \mathbf{B}_{\tau_i,\tau_{i+1}}(\|f\|_{C^3_b}, 1, \xi_i, f(\xi_i))$, we let
\begin{equation*}
\begin{split}
&\Psi(Y,Y') := \big(\psi(Y),\psi(Y)'\big)\\
&:= \bigg( \xi_i \1_{\llbracket \tau_i, T \rrbracket} + \int_{\tau_i \wedge \cdot}^{\tau_{i+1} \wedge \cdot} b(Y_u) \dd u + \int_0^\cdot \sigma(Y_{u-}) \dd M^{(\tau_i, \tau_{i+1}-)}_u + \int_0^\cdot f(Y_{u-}) \dd X^{(\tau_i, \tau_{i+1}-)}_u, f(Y) \1_{\llbracket \tau_i, T \rrbracket} \bigg)
\end{split}
\end{equation*}
where the latter two integrals are both defined as It\^o integrals on the product space.

We will show that the space $\mathbf{B}_{\tau_i,\tau_{i+1}}(\|f\|_{C^3_b}, 1, \xi_i, f(\xi_i))$ is non-empty and invariant under the map $\Psi$, and then that $\Psi$ is a contraction on $\mathbf{B}_{\tau_i,\tau_{i+1}}(\|f\|_{C^3_b}, 1, \xi_i, f(\xi_i))$ with respect to the metric $\mathbf{d}^\eta_{\tau_i,\tau_{i+1}}$, for suitable choices of $\alpha \in (0,1]$ and $\eta > 1$.

We first note that, if we set
\[(Y_t,Y'_t) = \big( \big(\xi_i + f(\xi_i) X^{(\tau_i, \tau_{i+1}-)}_t\big) \1_{\llbracket \tau_i, T \rrbracket}, f(\xi_i) \1_{\llbracket \tau_i, T \rrbracket} \big)\]
for all $t \in [0,T]$, then, provided that $\alpha \leq \|f\|_{C^3_b}^{-1}$, we have that $(Y,Y') \in \mathbf{B}_{\tau_i,\tau_{i+1}}(\|f\|_{C^3_b}, 1, \xi_i, f(\xi_i))$, so that the space $\mathbf{B}_{\tau_i,\tau_{i+1}}(\|f\|_{C^3_b}, 1, \xi_i, f(\xi_i))$ is non-empty.

In the following we will recall some estimates from the proof of \cite[Theorem~5.2]{AllanPieper2026}. These are written in terms of a function $\Phi$, which, given a time interval $[0,T]$, an initial value $\xi$, a rough path $\bZ = (Z,\Z)$ and stochastic controlled path $(Y,Y') \in \cV^{p,q,\infty}_Z$, is defined by
\begin{equation*}
\Phi_{\bZ}(Y,Y') := \big( \phi_{\bZ}(Y), \phi_{\bZ}(Y)' \big) := \bigg( \xi + \int_0^\cdot b(Y_u) \dd u + \int_0^\cdot \sigma(Y_{u-}) \dd M_u + \int_0^\cdot f(Y_u) \dd \bZ_u, f(Y) \bigg).
\end{equation*}

By Proposition~\ref{proposition: consistency rough and stochastic integrals for semimartingales}, we have that, for $\P$-almost every $\omega \in \Omega$,
\begin{equation*}
\bigg( \int_0^\cdot f(Y_{u-}) \dd X^{(\tau_i, \tau_{i+1}-)}_u \bigg) (\omega, \cdot) = \int_0^\cdot f(Y_u(\omega, \cdot)) \dd \bX^{(\tau_i, \tau_{i+1}-)}_u(\omega)
\end{equation*}
$\bar{\P}$-almost surely, where on the right-hand side is the rough stochastic integral of the stochastic controlled path $(f(Y(\omega, \cdot)), \D f(Y(\omega, \cdot)) Y'(\omega, \cdot))$ with respect to the rough path $\bX^{(\tau_i, \tau_{i+1}-)}(\omega)$.

By considering convergence of the Riemann sums $\lim_{|\cP| \to 0} \sum_{[u,v] \in \cP} \sigma(Y_u) \delta M^{(\tau_i, \tau_{i+1}-)}_{u,v}$, it is also straightforward to see that, for $\P$-almost every $\omega \in \Omega$,
\begin{equation*}
\bigg( \int_0^\cdot \sigma(Y_{u-}) \dd M^{(\tau_i, \tau_{i+1}-)}_u \bigg) (\omega, \cdot) = \int_0^\cdot \sigma(Y_{u-}(\omega, \cdot)) \dd M^{(\tau_i(\omega), \tau_{i+1}(\omega)-)}_u
\end{equation*}
$\bar{\P}$-almost surely, noting that $M^{(\tau_i(\omega), \tau_{i+1}(\omega)-)}$ is a martingale with respect to $(\bar{\cF}_t)_{t \in [0,T]}$. This means that, for $\P$-almost every $\omega \in \Omega$,
\begin{align*}
\psi(Y)(\omega, \cdot) &= \xi_i \1_{\llbracket \tau_i, T \rrbracket} + \int_{\tau_i(\omega) \wedge \cdot}^{\tau_{i+1}(\omega) \wedge \cdot} b(Y_u(\omega, \cdot)) \dd u + \int_0^\cdot \sigma(Y_{u-}(\omega, \cdot)) \dd M^{(\tau_i(\omega), \tau_{i+1}(\omega)-)}_u\\
&\quad \, + \int_0^\cdot f(Y_u(\omega, \cdot)) \dd \bX^{(\tau_i, \tau_{i+1}-)}_u(\omega)\\
&= \phi_{\bX(\omega)}(Y(\omega, \cdot))
\end{align*}
$\bar{\P}$-almost surely, and it is also clear that $\psi(Y)'(\omega, \cdot) = f(Y(\omega, \cdot)) = \phi_{\bX(\omega)}(Y(\omega, \cdot))'$.

\emph{Invariance:}
To see that $\mathbf{B}_{\tau_i,\tau_{i+1}}(\|f\|_{C^3_b}, 1, \xi_i, f(\xi_i))$ is invariant under $\Psi$, we can now simply apply the estimates derived in an analogous context in Step~1 of the proof of \cite[Theorem~5.2]{AllanPieper2026}. Specifically, for $\P$-almost every $\omega \in \Omega$, we have that
\begin{align*}
&\|\psi(Y)'(\omega, \cdot)\|_{p,q,\infty,[\tau_i(\omega),\tau_{i+1}(\omega)],\bar{\Omega}} \vee \sup_{s \in [\tau_i(\omega),\tau_{i+1}(\omega)]}\|\psi(Y)'_s(\omega, \cdot)\|_{L^\infty(\bar{\Omega})}\\
&= \|\phi_{\bX(\omega)}(Y(\omega, \cdot))'\|_{p,q,\infty,[\tau_i(\omega),\tau_{i+1}(\omega)],\bar{\Omega}} \vee \sup_{s \in [\tau_i(\omega),\tau_{i+1}(\omega)]}\|\phi_{\bX(\omega)}(Y(\omega, \cdot))'_s\|_{L^\infty(\bar{\Omega})} \leq \|f\|_{C^3_b},
\end{align*}
and
\begin{align*}
&\|\psi(Y)(\omega, \cdot)\|_{p,q,\infty,[\tau_i(\omega),\tau_{i+1}(\omega)],\bar{\Omega}} \vee \|\bar{\E}_{\edot} R^{\psi(Y)(\omega, \cdot)}\|_{\frac{p}{2},\infty,[\tau_i(\omega),\tau_{i+1}(\omega)],\bar{\Omega}}\\
&= \|\phi_{\bX(\omega)}(Y(\omega, \cdot))\|_{p,q,\infty,[\tau_i(\omega),\tau_{i+1}(\omega)],\bar{\Omega}} \vee \|\bar{\E}_{\edot} R^{\phi_{\bX(\omega)}(Y(\omega, \cdot))}\|_{\frac{p}{2},\infty,[\tau_i(\omega),\tau_{i+1}(\omega)],\bar{\Omega}}\\
&\leq C_1 \Big((\tau_{i+1}(\omega) - \tau_i(\omega)) + \|M^{(\tau_i(\omega),\tau_{i+1}(\omega)-)}\|_{p,q,\infty,[\tau_i(\omega),\tau_{i+1}(\omega)],\bar{\Omega}}\\
&\qquad \quad + \|\bX^{(\tau_i(\omega),\tau_{i+1}(\omega)-)}(\omega)\|_{p,[\tau_i(\omega),\tau_{i+1}(\omega)]}\Big),
\end{align*}
where the constant $C_1$ depends only on $p, q, \|b\|_{C^1_b}, \|\sigma\|_{C^1_b}$ and $\|f\|_{C^3_b}$.\footnote{We may ignore the dependence of the constant $C_1$ (as well as the constant $C_2$ below) on $\|X(\omega)\|_{p,[\tau_i(\omega),\tau_{i+1}(\omega)]}$ here, thanks to \eqref{eq: small interval bound} and the fact that $\alpha \leq 1$.} Provided that we choose $\alpha \leq \frac{1}{C_1}$, it then follows from \eqref{eq: small interval bound} that
\begin{equation*}
\|\psi(Y)(\omega, \cdot)\|_{p,q,\infty,[\tau_i(\omega),\tau_{i+1}(\omega)],\bar{\Omega}} \vee \|\bar{\E}_{\edot} R^{\psi(Y)(\omega, \cdot)}\|_{\frac{p}{2},\infty,[\tau_i(\omega),\tau_{i+1}(\omega)],\bar{\Omega}} \leq 1.
\end{equation*}
Since these estimates hold for $\P$-almost every $\omega \in \Omega$, it is then clear that
\begin{equation*}
\big\| \|\psi(Y)'\|_{p,q,\infty,\llbracket \tau_i, \tau_{i+1} \rrbracket,\bar{\Omega}} \big\|_{L^\infty(\Omega)} \vee \Big\| \sup_{s \in \llbracket \tau_i, \tau_{i+1} \rrbracket} \|\psi(Y)'_s\|_{L^\infty(\bar{\Omega})} \Big\|_{L^\infty(\Omega)} \leq \|f\|_{C^3_b}
\end{equation*}
and
\begin{equation*}
\big\| \|\psi(Y)\|_{p,q,\infty,\llbracket \tau_i, \tau_{i+1} \rrbracket,\bar{\Omega}} \big\|_{L^\infty(\Omega)} \vee \big\| \|\bar{\E}_{\edot} R^{\psi(Y)}\|_{\frac{p}{2},\infty,\llbracket \tau_i, \tau_{i+1} \rrbracket,\bar{\Omega}} \big\|_{L^\infty(\Omega)} \leq 1.
\end{equation*}
It follows that $\Psi(Y,Y') \in \mathbf{B}_{\tau_i,\tau_{i+1}}(\|f\|_{C^3_b}, 1, \xi_i, f(\xi_i))$, and hence that $\mathbf{B}_{\tau_i,\tau_{i+1}}(\|f\|_{C^3_b}, 1, \xi_i, f(\xi_i))$ is invariant under the map $\Psi$.

\emph{Contraction:}
Similarly, to see that $\Psi$ is a contraction on $\mathbf{B}_{\tau_i,\tau_{i+1}}(\|f\|_{C^3_b}, 1, \xi_i, f(\xi_i))$ with respect to $\mathbf{d}^\eta_{\tau_i,\tau_{i+1}}$, we apply the estimates derived in an analogous setting in Step~2 of the proof of \cite[Theorem~5.2]{AllanPieper2026} (which rely on the invariance established above). Specifically, for $(Y,Y'), (\tY, \tY') \in \mathbf{B}_{\tau_i,\tau_{i+1}}(\|f\|_{C^3_b}, 1, \xi_i, f(\xi_i))$, and $\P$-almost any $\omega \in \Omega$, we have that
\begin{align*}
&d^{\eta,\bar{\Omega}}_{\tau_i,\tau_{i+1}} \big(\Psi(Y,Y')(\omega, \cdot), \Psi(\tY,\tY')(\omega, \cdot)\big)\\
&= d^{\eta,\bar{\Omega}}_{\tau_i,\tau_{i+1}} \big( \Phi_{\bX(\omega)}\big(Y(\omega, \cdot), Y'(\omega, \cdot)\big), \Phi_{\bX(\omega)}\big(\tY(\omega, \cdot), \tY'(\omega, \cdot)\big) \big)\\
&\leq C_2 \Big( \Big( \|Y(\omega, \cdot) - \tY(\omega, \cdot)\|_{p,q,[\tau_i(\omega),\tau_{i+1}(\omega)],\bar{\Omega}} + \Big\| 1 \wedge \sup_{u \in [\tau_i(\omega),\tau_{i+1}(\omega)]} |Y_u(\omega, \cdot) - \tY_u(\omega, \cdot)| \Big\|_{L^q(\bar{\Omega})} \Big)\\
&\quad + \eta \Big( \|Y(\omega, \cdot) - \tY(\omega, \cdot)\|_{p,q,[\tau_i(\omega),\tau_{i+1}(\omega)],\bar{\Omega}} + \|Y'(\omega, \cdot) - \tY'(\omega, \cdot)\|_{p,q,[\tau_i(\omega),\tau_{i+1}(\omega)],\bar{\Omega}}\\
&\qquad + \|\bar{\E}_{\edot} (R^{Y(\omega, \cdot)} - R^{\tY(\omega, \cdot)})\|_{\frac{p}{2},q,[\tau_i(\omega),\tau_{i+1}(\omega)],\bar{\Omega}} + \Big\| 1 \wedge \sup_{u \in [\tau_i(\omega),\tau_{i+1}(\omega)]} |Y_u(\omega, \cdot) - \tY_u(\omega, \cdot)| \Big\|_{L^q(\bar{\Omega})} \Big)\\
&\qquad \ \times \Big( (\tau_{i+1}(\omega) - \tau_i(\omega)) + \|M^{(\tau_i(\omega),\tau_{i+1}(\omega)-)}\|_{p,q,\infty,[\tau_i(\omega),\tau_{i+1}(\omega)],\bar{\Omega}}\\
&\qquad \qquad + \|\bX^{(\tau_i(\omega),\tau_{i+1}(\omega)-)}(\omega)\|_{p,[\tau_i(\omega),\tau_{i+1}(\omega)]} \Big) \Big),
\end{align*}
where the constant $C_2 > \frac{1}{2}$ depends only on $p, q, \|b\|_{C^1_b}, \|\sigma\|_{C^1_b}$ and $\|f\|_{C^3_b}$.\footnote{The extra term in the metric  in \eqref{eq: definition of metric on B_T for contraction}, compared to the metric used in the proof of \cite[Theorem~5.2]{AllanPieper2026}, results in a slightly different constant $C_2$ here, but this is immaterial.}

By choosing $\eta = 2 C_2 > 1$ and $\alpha \leq \frac{1}{4 C_2^2}$, it then follows from \eqref{eq: small interval bound} that
\begin{equation*}
d^{\eta,\bar{\Omega}}_{\tau_i,\tau_{i+1}} \big(\Psi(Y,Y')(\omega, \cdot), \Psi(\tY,\tY')(\omega, \cdot)\big) \leq \frac{\eta + 1}{2\eta} d^{\eta,\bar{\Omega}}_{\tau_i,\tau_{i+1}} \big( (Y,Y')(\omega, \cdot), (\tY,\tY')(\omega, \cdot)\big).
\end{equation*}
Since this inequality holds for $\P$-almost every $\omega \in \Omega$, we then have that
\begin{equation*}
\mathbf{d}^\eta_{\tau_i,\tau_{i+1}} \big(\Psi(Y,Y'), \Psi(\tY,\tY')\big) \leq \frac{\eta + 1}{2\eta} \mathbf{d}^\eta_{\tau_i,\tau_{i+1}} \big( (Y,Y'), (\tY,\tY')\big),
\end{equation*}
so that $\Psi$ is indeed a contraction on $\mathbf{B}_{\tau_i,\tau_{i+1}}(\|f\|_{C^3_b}, 1, \xi_i, f(\xi_i))$ with respect to $\mathbf{d}^\eta_{\tau_i,\tau_{i+1}}$. Since, by Lemma~\ref{lemma: B_tau is a closed space with respect to hat d}, this is a complete metric space, it follows from the Banach fixed-point theorem that there exists a unique fixed point $(Y,Y') \in \mathbf{B}_{\tau_i,\tau_{i+1}}(\|f\|_{C^3_b}, 1, \xi_i, f(\xi_i))$ of $\Psi$.

In particular, $\hat{\P}$-almost surely, whenever $t \in \llbracket \tau_i, \tau_{i+1} \llbracket$, we have that
\begin{equation*}
Y_t = \xi_i + \int_{\tau_i}^t b(Y_u) \dd u + \int_{\tau_i}^t \sigma(Y_{u-}) \dd M_u + \int_{\tau_i}^t f(Y_{u-}) \dd X_u,
\end{equation*}
and, recalling Remark~\ref{remark: consistency of slicing roughs stochastic integrals}, for $\P$-almost every $\omega \in \Omega$ and any $t \in [\tau_i(\omega), \tau_{i+1}(\omega))$,
\begin{align*}
Y_t(\omega, \cdot) = \xi_i(\omega, \cdot) + \int_{\tau_i}^t b(Y_u(\omega, \cdot)) \dd u + \int_{\tau_i}^t \sigma(Y_{u-}(\omega, \cdot)) \dd M_u + \int_{\tau_i}^t f(Y_u(\omega, \cdot)) \dd \bX_u(\omega)
\end{align*}
holds $\bar{\P}$-almost surely, so that $Y$ is the solution to both the doubly SDE \eqref{eq: the doubly stochastic SDE} and the rough SDE \eqref{eq: the corresponding RSDE} on the interval $\llbracket \tau_i, \tau_{i+1} \llbracket$, with the initial condition $Y_{\tau_i} = \xi_i$.

\emph{Global solution:}
Now let $(\xi_i)_{i \in \N}$ be a sequence of random variables such that, for each $i \in \N$, $\xi_i$ is $\hat{\cF}_{\tau_i}$-measurable and $\xi(\omega, \cdot) \in L^q(\bar{\Omega})$ for $\P$-almost every $\omega \in \Omega$. Of course, the argument above is valid on each of the intervals $\llbracket \tau_i, \tau_{i+1} \llbracket$ for $i \in \N$. Thus, for each $i \in \N$, we obtain a solution $Y$ on the interval $\llbracket \tau_i, \tau_{i+1} \llbracket$, with the initial condition $Y_{\tau_i} = \xi_i$.

We then extend this to a solution on $\llbracket \tau_i, \tau_{i+1} \rrbracket$ by introducing the jump at time $\tau_{i+1}$. That is, given the solution $Y$ on $\llbracket \tau_i, \tau_{i+1} \llbracket$, we let
\begin{equation*}
Y_{\tau_{i+1}} = Y_{\tau_{i+1}-} + \sigma(Y_{\tau_{i+1}-}) \Delta M_{\tau_{i+1}} + f(Y_{\tau_{i+1}-}) \Delta X_{\tau_{i+1}},
\end{equation*}
which is consistent with the canonical jump structure of both SDEs and rough SDEs, noting in particular that $\Delta \X_{\tau_{i+1}} = 0$ for any It\^o rough path lift.

We then obtain a global solution $Y$ by simply matching each initial value $\xi_i$ with the terminal value from the previous subinterval. That is, we let $\xi_0 = y_0$, and for each $i \geq 1$, we let $\xi_i = Y_{\tau_i}$, where $Y$ is the solution obtained above on the interval $\llbracket \tau_{i-1}, \tau_i \rrbracket$. Since, for $\P$-almost every $\omega \in \Omega$, we have that $\tau_i(\omega) = T$ for sufficiently large $i$, we see that this defines a process $Y$ on $[0,T]$, which is the unique solution to both the SDE in \eqref{eq: the doubly stochastic SDE} and the rough SDE in \eqref{eq: the corresponding RSDE}.
\end{proof}

\subsection{A stability estimate for doubly stochastic DEs}

In light of Theorem~\ref{theorem: consistency between doubly stochastic SDEs and RSDEs}, one may ask whether the stability of rough SDE solutions with respect to their driving noise carry over to the corresponding doubly stochastic setting. Indeed, the following result establishes a novel continuity estimate for solutions to doubly stochastic DEs.

\begin{proposition}\label{proposition: Lipschitz continuity of doubly stochastic DEs}
Let $p \in (2,3)$, $q \in [2,\infty)$, and let $m, n, r \in [1,\infty]$ such that $\frac{1}{n} + \frac{1}{r} = \frac{1}{m}$. Let $b \in C^1_b$, $\sigma \in C^1_b$ and $f \in C^3_b$, and let
\begin{equation*}
(\hat{\Omega}, \hat{\cF}, (\hat{\cF}_t)_{t \in [0,T]}, \hat{\P}) = (\Omega \times \bar{\Omega}, \cF \otimes \bar{\cF}, (\cF_t \otimes \bar{\cF}_t)_{t \in [0,T]}, \P \otimes \bar{\P})
\end{equation*}
be a product probability space which satisfies the usual conditions.

Let $y_0, \ty_0$ be $\hat{\cF}_0$-measurable random variables such that $\|y_0\|_{L^q(\bar{\Omega})}, \|\tilde{y}_0\|_{L^q(\bar{\Omega})} \in L^r(\Omega)$. Let $X$ and $\tX$ be $\cF_t$-adapted c\`adl\`ag semimartingales on $(\Omega, \cF, \P)$, such that their It\^o rough path lifts $\bX, \tbX$ satisfy $\|\bX\|_{p,[0,T]}, \|\tbX\|_{p,[0,T]} \in L^{2np}(\Omega) \cap L^r(\Omega)$, and let $M, \tM \in V^p L^{q,\infty}(\bar{\Omega})$ be $\bar{\cF}_t$-adapted c\`adl\`ag martingales on $(\bar{\Omega}, \bar{\cF}, \bar{\P})$, and suppose that the norms $\|\|\bX\|_{p,[0,T]}\|_{L^{2np}(\Omega)}$, $\|\|\tbX\|_{p,[0,T]}\|_{L^{2np}(\Omega)}$, $\|M\|_{p,q,\infty,[0,T],\bar{\Omega}}$ and $\|\tM\|_{p,q,\infty,[0,T],\bar{\Omega}}$ are all bounded by a constant $L > 0$.

Let $Y$ and $\tY$ be the unique strong solutions to the SDE \eqref{eq: the doubly stochastic SDE} with data $(y_0,M,X)$ and $(\ty_0,\tM,\tX)$ respectively. We then have that
\begin{equation*}
\begin{split}
\big\|&\|Y - \tY\|_{p,q,[0,T], \bar{\Omega}}\big\|_{L^m(\Omega)} + \big\|\|Y' - \tY'\|_{p,q,[0,T],\bar{\Omega}}\big\|_{L^m(\Omega)} + \big\|\|\E_{\edot} (R^Y - R^{\tY})\|_{\frac{p}{2},q,[0,T],\bar{\Omega}}\big\|_{L^m(\Omega)}\\
&\leq C \Big(\big\|\|y_0 - \ty_0\|_{L^q(\bar{\Omega})}\big\|_{L^{r}(\Omega)} + \|M - \tM\|_{p,q,[0,T],\bar{\Omega}} + \big\|\|\bX - \tbX\|_{p,[0,T]}\big\|_{L^{r}(\Omega)}\Big),
\end{split}
\end{equation*}
where the constant $C$ depends only on $p, q, \|b\|_{C^1_b}, \|\sigma\|_{C^1_b}, \|f\|_{C^3_b}, T$ and $L$.
\end{proposition}

\begin{proof}
We let $w$ be the control given by
\begin{equation*}
w(s,t) := (t - s) + \|M\|_{p,q,\infty,[s,t],\bar{\Omega}}^p + \|\tM\|_{p,q,\infty,[s,t],\bar{\Omega}}^p + \|\bX\|_{p,[s,t]}^p + \|\tbX\|_{p,[s,t]}^p
\end{equation*}
for $(s,t) \in \Delta_{[0,T]}$. In the following, we will consider intervals $[s,t]$ such that $w(s,t) \leq 1$, meaning in particular that $\|\bX\|_{p,[s,t]}, \|\tbX\|_{p,[s,t]} \leq 1$.

By Theorem~\ref{theorem: consistency between doubly stochastic SDEs and RSDEs}, we know that $Y$ and $\tY$ are the solutions to the rough SDE in \eqref{eq: the corresponding RSDE} with data $(y_0,M,\bX)$ and $(\ty_0,\tM,\tbX)$ respectively. It was established in Step~3 of the proof of \cite[Theorem~5.2]{AllanPieper2026} that there exists a constant $\epsilon \in (0,1]$, which depends only on $p, q, \|b\|_{C^1_b}, \|\sigma\|_{C^1_b}$ and $\|f\|_{C^3_b}$, such that, for every $(s,t) \in \Delta_{[0,T]}$ with $w(s,t) \leq \epsilon$, we have the local estimate
\begin{equation}\label{eq: local robustness result}
\begin{split}
\|&Y - \tY\|_{p,q,[s,t],\bar{\Omega}} + \|Y' - \tY'\|_{p,q,[s,t],\bar{\Omega}} + \|\E_{\edot} (R^Y - R^{\tY})\|_{\frac{p}{2},q,[s,t],\bar{\Omega}}\\
&\lesssim \big( \|Y_s - \tY_s\|_{L^q(\bar{\Omega})} + \|M - \tM\|_{p,q,[s,t],\bar{\Omega}} + \|\bX - \tbX\|_{p,[s,t]} \big),
\end{split}
\end{equation}
where the implicit multiplicative constant also depends only on $p, q, \|b\|_{C^1_b}, \|\sigma\|_{C^1_b}$ and $\|f\|_{C^3_b}$. In particular, by taking $\epsilon \leq 1$, we ensure that $\|\bX\|_{p,[s,t]}, \|\tbX\|_{p,[s,t]} \leq 1$, so that these constants may be chosen to not depend on $\bX$ or $\tbX$.

To then extend this to a global estimate, we take a partition $\{t_i\}_{i=0}^K$ of the interval $[0,T]$ such that $w(t_i,t_{i+1}-) \leq \epsilon$, so that \eqref{eq: local robustness result} holds on $[t_i,t_{i+1})$ for each $i$. By the superadditivity of the control $w$, we may choose this partition such that
\begin{equation}\label{eq: bound on n without X and tX}
K \lesssim 1 + \|\bX\|_{p,[0,T]}^p + \|\tbX\|_{p,[0,T]}^p,
\end{equation}
where the multiplicative constant depends only on $p, q, \|b\|_{C^1_b}, \|\sigma\|_{C^1_b}, \|f\|_{C^3_b}$ and on the norms $\|M\|_{p,q,\infty,[0,T],\bar{\Omega}}$ and $\|\tM\|_{p,q,\infty,[0,T],\bar{\Omega}}$. Using the canonical jump structure of $Y$ and $\tY$, we have that
\begin{equation*}
\|Y - \tY\|_{p,q,[t_{i-1},t_i],\bar{\Omega}} \lesssim \|Y_{t_{i-1}} - \tY_{t_{i-1}}\|_{L^q(\bar{\Omega})} + \|M - \tM\|_{p,q,[t_{i-1},t_i],\bar{\Omega}} + \|\bX - \tbX\|_{p,[t_{i-1},t_i]}.
\end{equation*}
Since $\|Y_{t_{i-1}} - \tY_{t_{i-1}}\|_{L^q(\bar{\Omega})} \leq \|Y_{t_{i-2}} - \tY_{t_{i-2}}\|_{L^q(\bar{\Omega})} + \|Y - \tY\|_{p,q,[t_{i-2},t_{i-1}],\bar{\Omega}}$ for each $i$, we then deduce that
\begin{equation*}
\|Y - \tY\|_{p,q,[t_{i-1},t_i],\bar{\Omega}} \lesssim i \big(\|y_0 - \ty_0\|_{L^q(\bar{\Omega})} + \|M - \tM\|_{p,q,[0,t_i],\bar{\Omega}} + \|\bX - \tbX\|_{p,[0,t_i]}\big).
\end{equation*}
We then have
\begin{align*}
\|Y - \tY\|_{p,q,[0,T],\bar{\Omega}} &\leq K^{\frac{p-1}{p}} \bigg( \sum_{i=0}^{K-1} \|Y - \tY\|_{p,q,[t_i,t_{i+1}],\bar{\Omega}}^p \bigg)^{\frac{1}{p}}\\
&\lesssim K^2 \big( \|y_0 - \ty_0\|_{L^q(\bar{\Omega})} + \|M - \tM\|_{p,q,[0,T],\bar{\Omega}} + \|\bX - \tbX\|_{p,[0,T]} \big),
\end{align*}
where again the multiplicative constant is independent of $\bX$ and $\tbX$. Thus, applying H\"older's inequality in the case when $n, r < \infty$, we have that
\begin{align*}
&\big\|\|Y - \tY\|_{p,q,[0,T],\bar{\Omega}}\big\|_{L^m(\Omega)}\\
&\lesssim \|K^2\|_{L^n(\Omega)} \Big( \big\|\|y_0 - \ty_0\|_{L^q(\bar{\Omega})}\big\|_{L^r(\Omega)} + \|M - \tM\|_{p,q,[0,T],\bar{\Omega}} + \big\|\|\bX - \tbX\|_{p,[0,T]}\big\|_{L^r(\Omega)} \Big),
\end{align*}
where, by \eqref{eq: bound on n without X and tX}, we see that $\|K^2\|_{L^n(\Omega)} \lesssim 1 + \|\|\bX\|_{p,[0,T]}\|_{L^{2np}(\Omega)}^{2p} + \|\|\tbX\|_{p,[0,T]}\|_{L^{2np}(\Omega)}^{2p}$. An identical argument gives the same bound for $\|Y' - \tY'\|_{p,q,[0,T],\bar{\Omega}}$ and for $\|\E_{\edot} (R^Y - R^{\tY})\|_{\frac{p}{2},q,[0,T],\bar{\Omega}}$.
\end{proof}

\section{A John--Nirenberg inequality for $\textup{BMO}^{p\textup{-var}}$ processes}\label{sec: John-Nirenberg}

Integrability properties of random rough paths and associated RDEs have drawn significant interest; see, e.g., \cite{CassLittererLyons2013}, \cite{CassOgrodnik2017}, \cite{FrizGessGulisashviliRiedel2016}, \cite{FrizOberhauser2010} and \cite{FrizRiedel2013}. In particular, the existence of exponential moments for solutions to rough SDEs has been crucial for applications to robust stochastic filtering, as seen in \cite{CrisanDiehlFrizOberhauser2013} and \cite{DiehlOberhauserRiedel2015}. A particularly useful object in the study of exponential moments of solutions to RDEs turned out to be an object, typically denoted by $N_{\alpha, [s,t]}(w)$ for a continuous control $w$, first introduced in \cite{CassLittererLyons2013}, which counts how many times the control exceeds a given level $\alpha > 0$ on an interval $[s,t]$. This has been shown to often exhibit better tail behaviour than $\|\bX\|_{p,[0,T]}^p$ for the rough path lift of a semimartingale; see, e.g., \cite{FrizGessGulisashviliRiedel2016}.

In \cite[Section~3]{Le2022a}, the author discusses a relevant class of processes with vanishing mean oscillation, denoted by $\textup{VMO}^{p\textup{-var}}$, and their connection to rough SDEs. In particular, a John--Nirenberg-type inequality provides an explicit bound on the exponential moments of such processes (see \cite[Theorem~3.4]{Le2022a}), and the existence of exponential moments is established for the solutions to rough SDEs with continuous driving noise. In this section, we extend these results to discontinuous processes, and establish exponential moments of rough stochastic integrals and of solutions to rough SDEs with jumps, as required for our subsequent application to robust filtering. To this end, we first extend the notion of $\textup{VMO}^{p\textup{-var}}$ processes to more general $\textup{BMO}^{p\textup{-var}}$ ones, and relate the corresponding bounds to a generalization of $N_{\alpha,[s,t]}(w)$ to discontinuous controls. This both establishes the desired exponential moments for rough stochastic integrals, and will also allow us to establish corresponding moments in a doubly stochastic setting in Section~\ref{sec: application to filtering}.

\subsection{Definitions and preliminary results}

For convenience, we recall the notion of a BMO process, along with some relevant notation as introduced in \cite{Le2022a}.

\begin{definition}
Given a c\`adl\`ag adapted process $V = (V_t)_{t \in [0,T]}$, we define its \textit{modulus of mean variation} over an interval $[s,t] \subseteq [0,T]$ by
\begin{equation*}
\rho_{s,t}(V) := \sup_{s \leq \sigma \leq \tau \leq t} \big\|\E_{\sigma} [|\delta V_{\sigma-,\tau}|]\big\|_{L^\infty},
\end{equation*}
where the supremum is taken over all stopping times $\sigma, \tau$ such that $s \leq \sigma \leq \tau \leq t$. The process $V$ is said to be of \emph{bounded mean oscillation} (BMO) if $\rho_{0,T}(V) < \infty$.
\end{definition}

\begin{definition}
For $p \in [1,\infty)$, we write $\textup{BMO}^{p\textup{-var}}$ for the class of c\`adl\`ag adapted processes $V = (V_t)_{t \in [0,T]}$ such that
\begin{equation*}
\|V\|_{\textup{BMO}^{p\text{-var}},[0,T]} := \bigg(\sup_{\cP \subset [0,T]} \sum_{[s,t] \in \cP} \rho_{s,t}(V)^p\bigg)^{\hspace{-2pt}\frac{1}{p}} < \infty.
\end{equation*}
\end{definition}

The following result is a generalization of \cite[Proposition~4.2]{Le2022a} to the c\`adl\`ag setting.

\begin{proposition}\label{proposition: BMO-p-var is p,q,infty}
Let $p, q \in [1,\infty)$. Let $V$ be a c\`adl\`ag adapted process such that $V_0 \in L^q$, and define $\Gamma = (\Gamma_{s,t})_{(s,t) \in \Delta_{[0,T]}}$ by
\begin{equation}\label{eq: defn Gamma_st}
\Gamma_{s,t} := \sup_{u \in [s,t]} |\Delta V_u|
\end{equation}
for each $(s,t) \in \Delta_{[0,T]}$. Then there exists a constant $C > 0$, which depends only on $q$, such that
\begin{equation*}
\frac{1}{5} \|V\|_{\textup{BMO}^{p\textup{-var}},[0,T]} \leq \|V\|_{p,q,\infty,[0,T]} + \|\Gamma\|_{p,\infty,[0,T]} \leq C \|V\|_{\textup{BMO}^{p\textup{-var}},[0,T]}.
\end{equation*}
In particular, $V \in \textup{BMO}^{p\textup{-var}}$ if and only if $V \in V^p L^{q,\infty}$ and $\|\Gamma\|_{p,\infty,[0,T]} < \infty$.
\end{proposition}

\begin{proof}
For some $(s,t) \in \Delta_{[0,T]}$, we let $A_u := \sup_{r \in [s,u]} |\delta V_{s,r}|$ for $u \in [s,t]$. Then, by \cite[Proposition~2.8]{Le2022a}, for any stopping time $s \leq \sigma \leq t$, we have that
\begin{equation*}
\big\|\E_\sigma [|\delta A_{\sigma-,t}|]\big\|_{L^\infty} \leq \rho_{s,t}(A) \leq 11 \rho_{s,t}(V).
\end{equation*}
Hence, by \cite[Lemma~A.2]{Le2022a}, we have that
\begin{equation*}
\|\delta V_{s,t}\|_{q,\infty,s}^{\lceil q \rceil} \leq \big\|\E_s \big[|A_t|^{\lceil q \rceil}\big]\big\|_{L^\infty} \lesssim \rho_{s,t}(V)^{\lceil q \rceil},
\end{equation*}
which implies that $\|V\|_{p,q,\infty,[0,T]} \lesssim \|V\|_{\textup{BMO}^{p\text{-var}},[0,T]}$. By \cite[Proposition~2.1]{Le2022a}, we also have that $\|\Gamma_{s,t}\|_{L^\infty} \leq \rho_{s,t}(V)$, so that $\|\Gamma\|_{p,\infty, [0,T]} \leq \|V\|_{\textup{BMO}^{p\text{-var}},[0,T]}$.

For the reverse inequality, we apply \cite[Proposition~2.2]{Le2022a} to find that
\begin{equation}\label{eq: bound on BMO-p-var}
\rho_{s,t}(V) \leq 2 \|V\|_{p,q,\infty,[s,t]} + 5 \|\Gamma_{s,t}\|_{L^\infty},
\end{equation}
so that $\|V\|_{\textup{BMO}^{p\text{-var}},[0,T]} \leq 5 (\|V\|_{p,q,\infty,[0,T]} + \|\Gamma\|_{p,\infty,[0,T]})$.
\end{proof}

The following corollary is an immediate consequence of Proposition~\ref{proposition: BMO-p-var is p,q,infty} and Theorem~\ref{theorem: a p-variation kolmogorov result}.

\begin{corollary}\label{corollary: V in VpLqinfty for all q}
Let $p \in [1,\infty)$. If $V \in \textup{BMO}^{p\textup{-var}}$ and $V_0 \in L^\infty$, then $V \in V^p L^{q,\infty}$ for every $q \in [1,\infty)$. In particular, for any $\tp > p$, we have that
\begin{equation*}
\big\| \|V\|_{\tp,[0,T]} \big\|_{L^q} < \infty
\end{equation*}
for every $q \in [1,\infty)$.
\end{corollary}

\begin{definition}\label{definition: regular from the inside}
We call a control $w$ on $[0,T]$ \emph{regular from the inside} if it is continuous from the inside, in the sense that $w(s+,t) = w(s,t) = w(s,t-)$ for all $(s,t) \in \Delta_{[0,T]}$, and also satisfies
\begin{equation*}
\lim_{t \searrow s} w(s,t) = 0
\end{equation*}
for every $s \in [0,T)$.
\end{definition}

The following definition and subsequent lemma may be found in \cite{FrizRiedel2013}, which itself builds on the results of \cite{CassLittererLyons2013}. Although here we consider slightly more general controls, the proof of Lemma~\ref{lemma: N_alpha for w_1 leq w_2} follows the corresponding proof in \cite{FrizRiedel2013} verbatim.

\begin{definition}
Let $w$ be a control which is regular from the inside. For a given $\alpha > 0$ and $(s,t) \in \Delta_{[0,T]}$, we write $t_0(\alpha) := s$, and
\begin{equation}\label{eq: definition tau_i(alpha)}
t_i(\alpha) := \inf \{u > t_{i-1}(\alpha) \ | \ w(t_{i-1}(\alpha),u) \geq \alpha\} \wedge t
\end{equation}
for $i \geq 1$, and then define
\begin{equation*}
N_{\alpha, [s,t]}(w) := \sup \{n \in \N \cup \{0\} \ | \ t_n(\alpha) < t\}.
\end{equation*}
\end{definition}

\begin{lemma}[Lemma~2 in \cite{FrizRiedel2013}]\label{lemma: N_alpha for w_1 leq w_2}
Let $w_1, w_2$ be controls which are regular from the inside, and let $\alpha > 0$ and $(s,t) \in \Delta_{[0,T]}$. Suppose that $w_1(u,v) \leq C w_2(u,v)$ holds for all $(u,v) \in \Delta_{[0,T]}$ such that $w_2(u,v) \leq \alpha$, for some constant $C > 0$. Then
\[N_{C\alpha, [s,t]}(w_1) \leq N_{\alpha, [s,t]}(w_2).\]
\end{lemma}

The following result is a generalisation of \cite[Lemma~3]{FrizRiedel2013}, where here we need to take care of the large jumps separately.

\begin{lemma}\label{lemma: N_alpha < (2 N_beta+1)}
Let $w$ be a control which is regular from the inside, and let $0 < \alpha \leq \beta$. Then
\begin{equation*}
N_{\alpha, [s,t]}(w) \leq C \big(N_{\beta, [s,t]}(w) + 1\big)
\end{equation*}
holds for any $(s,t) \in \Delta_{[0,T]}$, where the constant $C > 0$ depends only on $\alpha$ and $\beta$.
\end{lemma}

\begin{proof}
For notational simplicity, we will write $\{t_i\}_{i=0}^{N_{\alpha, [s,t]}(w)}$ for the sequence of times in $[s,t)$ defined in \eqref{eq: definition tau_i(alpha)}. For $(u,v) \in \Delta_{[0,T]}$, we write
\begin{equation*}
\Delta_{u,v} w := w(u,v+) - w(u,v),
\end{equation*}
and we then set
\begin{equation*}
N^{\Delta}_{\frac{\alpha}{2}, [s,t]}(w) := \# \Big\{i \in \{0, \ldots, N_{\alpha, [s,t]}(w) - 1\} \ :  \Delta_{t_i,t_{i+1}} w > \frac{\alpha}{2} \Big\}.
\end{equation*}
(Here it is crucial that we consider the times $t_i = t_i(\alpha)$, rather than $t_i(\frac{\alpha}{2})$.) We also define
\begin{equation*}
w_{\alpha}(s,t) := \sup_{\substack{\cP = (u_i)_{i=1}^k \subset [s,t]\\
w(u_i,u_{i+1}) \leq \alpha \, \forall i}} \, \sum_{i=1}^{k-1} w(u_i,u_{i+1}).
\end{equation*}

We note that, for each $i$, either $w(t_i,t_{i+1}) = \alpha$ or $\Delta_{t_i,t_{i+1}} w > 0$. Thus,
\begin{align*}
\alpha N_{\alpha,[s,t]}(w) &\leq \sum_{i=0}^{N_{\alpha,[s,t]}(w)-1} w(t_i,t_{i+1}) + \alpha N^{\Delta}_{\frac{\alpha}{2},[s,t]}(w) + \sum_{i \, : \, \Delta_{t_i,t_{i+1}} w \leq \frac{\alpha}{2}} \Delta_{t_i,t_{i+1}} w\\
&\leq w_{\alpha}(s,t) + \alpha N^{\Delta}_{\frac{\alpha}{2},[s,t]}(w) + \frac{\alpha}{2} N_{\alpha, [s,t]}(w),
\end{align*}
so that, rearranging, we obtain
\begin{equation*}
N_{\alpha,[s,t]}(w) \leq \frac{2}{\alpha} \big(w_{\alpha}(s,t) + \alpha N^{\Delta}_{\frac{\alpha}{2}, [s,t]}(w)\big) \leq \frac{2}{\alpha} \big(w_{\beta}(s,t) + \alpha N^{\Delta}_{\frac{\alpha}{2}, [s,t]}(w)\big).
\end{equation*}
By the same proof as that of \cite[Proposition~4.11]{CassLittererLyons2013}, one can show that
\begin{equation*}
w_{\beta}(s,t) \leq \beta \big(2 N_{\beta, [s,t]}(w) + 1\big),
\end{equation*}
and it thus suffices to show that there exists a constant $\tilde{C} > 0$, depending only on $\alpha$ and $\beta$, such that
\begin{equation}\label{eq: N Delta alpha/2 bounded by N beta}
N_{\frac{\alpha}{2}, [s,t]}^{\Delta}(w) \leq \tilde{C} N_{\beta, [s,t]}(w).
\end{equation}
We will show that this holds with $\tilde{C} = \lceil \frac{2\beta}{\alpha} \rceil + 2$. To this end, we write
\begin{equation*}
(t^{\Delta}_j(\alpha))_{j=0}^{N^{\Delta}_{\frac{\alpha}{2},[s,t]}(w)} := \Big\{ t_j(\alpha) \ : \ \Delta_{t_j(\alpha),t_{j+1}(\alpha)} w > \frac{\alpha}{2} \Big\}
\end{equation*}
for the times when $w$ exceeds $\alpha$ with a jump larger than $\frac{\alpha}{2}$. For any $0 \leq i \leq N_{\beta, [s,t]}(w) - 1$, we then have that
\begin{equation*}
\# \Big\{j \ : \ t^{\Delta}_j(\alpha) \in [t_i(\beta),t_{i+1}(\beta)] \Big\} \leq \bigg\lceil \frac{2\beta}{\alpha} \bigg\rceil + 2,
\end{equation*}
since, otherwise, by the superadditivity of $w(\cdot, \cdot+)$, we would have that
\begin{equation*}
\beta < \bigg(\bigg\lceil \frac{2\beta}{\alpha} \bigg\rceil + 1\bigg) \frac{\alpha}{2} \leq \sum_{t_i(\beta) \leq t^{\Delta}_j(\alpha) < t^{\Delta}_{j+1}(\alpha) < t_{i+1}(\beta)} \Delta_{t^{\Delta}_j(\alpha), t^{\Delta}_{j+1}(\alpha)} w \leq w(t_i(\beta), t_{i+1}(\beta)) \leq \beta.
\end{equation*}
Since there are $N_{\beta, [s,t]}(w)$ intervals of the form $[t_i(\beta),t_{i+1}(\beta)]$, the bound in \eqref{eq: N Delta alpha/2 bounded by N beta} follows.
\end{proof}

\begin{lemma}\label{lemma: bound on N_alpha(w_1+ w_2)}
Let $w_1, w_2$ be controls which are regular from the inside, and let $\alpha > 0$ and $(s,t) \in \Delta_{[0,T]}$. Then
\begin{equation*}
N_{\alpha, [s,t]}(w_1 + w_2) \leq N_{\frac{\alpha}{2}, [s,t]}(w_1) + N_{\frac{\alpha}{2}, [s,t]}(w_2).
\end{equation*}
\end{lemma}

\begin{proof}
This is simply a consequence of the definition of $N_{\alpha, [s,t]}$, and the fact that, if $w_1(t_i,u) + w_2(t_i,u) \geq \alpha$, then either $w_1(t_i,u) \geq \frac{\alpha}{2}$ or $w_2(t_i,u) \geq \frac{\alpha}{2}$.
\end{proof}

\subsection{The John--Nirenberg inequality}

\begin{theorem}[John--Nirenberg inequality]\label{theorem: John-Nirenberg inequality}
Let $V = (V_t)_{t \in [0,T]}$ be a $\textup{BMO}^{p\textup{-var}}$ process for some $p \in [1,\infty)$. Then, for any $\alpha > 0$, any $\lambda > 0$ and any $r \in [0,T]$, we have that
\begin{equation*}
\Big\| \E_r \Big[ \exp \Big( \lambda \sup_{t \in [r,T]} |\delta V_{r,t}| \Big) \Big] \Big\|_{L^\infty} \leq \exp \Big( C \Big(N_{\alpha,[r,T]} \big( \|V\|_{\textup{BMO}^{p\textup{-var}},(\cdot,\cdot)}^p \big) + 1 \Big) \Big(1 + \sup_{t \in (r,T]} \|\Delta V_t\|_{L^\infty}\Big) \Big).
\end{equation*}
where the constant $C$ depends only on $p, \alpha$ and $\lambda$.
\end{theorem}

Theorem~\ref{theorem: John-Nirenberg inequality} follows from the following theorem, combined with Proposition~\ref{proposition: BMO-p-var is p,q,infty} applied with $q = 1$ to the process $(\delta V_{0,t})_{t \in [0,T]}$.

\begin{theorem}\label{theorem: John-Nirenberg inequality for p,q,infty}
Let $p, q \in [1,\infty)$, and let $V \in V^p L^{q,\infty}$ such that $\|\Gamma\|_{p,\infty,[0,T]} < \infty$, where $\Gamma$ is the two-parameter process defined in \eqref{eq: defn Gamma_st}. Let $\alpha > 0$ and $\lambda > 0$. Then there exists a constant $C$, which depends only on $p, \alpha$ and $\lambda$, such that, for every $r \in [0,T]$, we have that
\begin{equation}\label{eq: JN bound for VpLqinfty}
\begin{split}
\Big\| &\E_r \Big[ \exp \Big(\lambda \sup_{t \in [r,T]} |\delta V_{r,t}|\Big) \Big] \Big\|_{L^\infty}\\
&\leq \exp \Big( C \Big(N_{\alpha,[r,T]} \big(\|V\|_{p,1,\infty, [\cdot,\cdot)}^p + \|\Gamma\|_{p,\infty,(\cdot,\cdot)}^p\big) + 1 \Big) \Big(1 + \sup_{t \in (r,T]} \|\Delta V_t\|_{L^\infty}\Big) \Big).
\end{split}
\end{equation}
\end{theorem}

\begin{proof}
By \cite[Theorem~1.3]{Le2022b}, whenever $\lambda \rho_{s,t}(V) \leq e^{-3}$, we have that
\begin{equation*}
\Big\| \E_s \Big[ \exp \Big( \lambda \sup_{u \in [s,t]} |\delta V_{s,u}| \Big) \Big] \Big\|_{L^\infty} \leq M,
\end{equation*}
where $M := 1 + \sum_{m=1}^{\infty} \frac{(c_m m)^m}{m!} e^{-3m} < \infty$ and $c_m^m := m(1+\frac{1}{m})^{(m+1)^2}$.

With the convention that $\|\Gamma\|_{p,\infty,(s,s)} := 0$ for all $s \in [0,T]$, we have that
\begin{equation*}
\bar{w}(s,t) := \|V\|^p_{p,1,\infty,[s,t)} + \|\Gamma\|_{p,\infty,(s,t)}^p
\end{equation*}
defines a control $\bar{w}$ which is regular from the inside. We let $\gamma = (10 \lambda e^3)^{-p}$, and consider the partition $\{t_i(\gamma)\}_{i=0}^{N_{\gamma,[r,t]}(\bar{w})+1}$ of the interval $[r,t]$ as defined in \eqref{eq: definition tau_i(alpha)}. For notational simplicity, we will write $t_i = t_i(\gamma)$ and $N = N_{\gamma,[r,T]}(\bar{w})$.

We define a process $\tV^i = (\tV^i_t)_{t \in [t_i,t_{i+1}]}$ such that $\tV^i_t = V_t$ for $t \in [t_i,t_{i+1})$, and $\tV^i_{t_{i+1}} = V_{t_{i+1}-}$. Then $\|\tV^i\|_{p,1,\infty,[t_i,t_{i+1}]} = \|V\|_{p,1,\infty,[t_i,t_{i+1})}$ and, setting $\Delta \tV^i_{t_i} := 0$,
\[\Big\|\sup_{u \in [t_i,t_{i+1}]} |\Delta \tV^i_u| \Big\|_{L^\infty} = \Big\|\sup_{u \in (t_i,t_{i+1})} |\Delta V_u| \Big\|_{L^\infty} \leq \|\Gamma\|_{p,\infty,(t_i,t_{i+1})}.\]
Using \eqref{eq: bound on BMO-p-var}, and the fact that $\bar{w}(t_i,t_{t+1})^{\frac{1}{p}} \leq \gamma^{\frac{1}{p}} = (10 \lambda e^3)^{-1}$, we then have that
\[\lambda \rho_{t_i,t_{i+1}}(\tV^{i}) \leq 5 \lambda \big(\|V\|_{p,1,\infty,[t_i,t_{i+1})} + \|\Gamma\|_{p,\infty,(t_i,t_{i+1})}\big) \leq 10 \lambda \bar{w}(t_i,t_{t+1})^{\frac{1}{p}} \leq e^{-3}.\]
For any $r \in [0,T)$, there exists an $i_0 \in \{0, 1, \ldots, N\}$ such that $r \in [t_{i_0},t_{i_{0}+1})$. Then
\begin{equation*}
\sup_{t \in [r,T]} |\delta V_{r,t}| \leq \sup_{t \in [r,t_{i_0+1}]} |\delta \tV^{i_0}_{r,t}| + \sum_{i=i_0+1}^N \sup_{t \in [t_i,t_{i+1}]} |\delta \tV^i_{t_i,t}| + \sum_{i=i_{0}+1}^{N+1} |\Delta V_{t_i}|,
\end{equation*}
so that
\begin{align*}
&\Big\| \E_r \Big[ \exp \Big( \lambda \sup_{t \in [r,T]} |\delta V_{r,t}| \Big) \Big] \Big\|_{L^\infty}\\
&\leq \bigg\| \E_r \bigg[ \exp \bigg( \lambda \bigg( \sup_{t \in [r,t_{i_0+1}]} |\delta \tV^{i_0}_{r,t}| + \sum_{i=i_0+1}^N \sup_{t \in [t_i,t_{i+1}]} |\delta \tV^i_{t_i,t}| + \sum_{i=i_{0}+1}^{N+1} |\Delta V_{t_i}| \bigg) \bigg) \bigg] \bigg\|_{L^\infty}\\
&\leq \Big\| \E_r \Big[ \exp \Big( \lambda \sup_{t \in [r,t_{i_0+1}]} |\delta \tV^{i_0}_{r,t}| \Big) \Big] \Big\|_{L^\infty} \prod_{i=i_0+1}^N \Big\| \E_{t_i} \Big[ \exp \Big( \lambda \sup_{t \in [t_i,t_{i+1}]} |\delta \tV^i_{t_i,t}| \Big) \Big] \Big\|_{L^\infty}\\
&\quad \times \exp \bigg( \lambda \sum_{i=i_0+1}^{N+1} \|\Delta V_{t_i}\|_{L^\infty} \bigg)\\
&\leq M^{N+1} \exp \Big(\lambda (N + 1) \sup_{t \in (r,T]} \|\Delta V_t\|_{L^\infty}\Big),
\end{align*}
which implies that \eqref{eq: JN bound for VpLqinfty} holds with $\alpha = \gamma$, and hence also for general $\alpha$ by Lemma~\ref{lemma: N_alpha < (2 N_beta+1)}.
\end{proof}

In particular, the John--Nirenberg inequality implies the following result, the proof of which is given in Appendix~\ref{sec: proof of JN corollary}.

\begin{proposition}\label{proposition: exp(V) is Holder continuous}
Let $p \in [1,\infty)$. If $V \in \textup{BMO}^{p\textup{-var}}$ is a one-dimensional process with $V_0 \in L^\infty$, then $\exp(V) \in V^p L^{q,r}$ for any $q \in [1,\infty)$ and $r \in [q,\infty)$.

Moreover, if $1 \leq \tilde{q} \leq r < q < \infty$, and if $\tV \in \textup{BMO}^{p\textup{-var}}$ is another one-dimensional process with $\tV_0 \in L^\infty$, then we have the estimate
\begin{equation}\label{eq: exp(V) - exp(tV) estimate}
\big\|\exp(V) - \exp(\tV)\big\|_{p,\tilde{q},r,[0,T]} \leq C \big(\|V_0 - \tV_0\|_{L^q} + \|V - \tV\|_{p,q,\frac{qr}{\tilde{q}},[0,T]}\big),
\end{equation}
where the constant $C > 0$ depends only on $p, q, \tilde{q}, r, \|V_0\|_{L^\infty}, \|\tV_0\|_{L^\infty}, \|V\|_{\textup{BMO}^{p\textup{-var}},[0,T]}$ and on $\|\tV\|_{\textup{BMO}^{p\textup{-var}},[0,T]}$.
\end{proposition}

\begin{proposition}\label{proposition: exponential integrability of a rough stochastic integral}
Let $p \in [2,3)$, $q \in [2,\infty)$, $\bX \in \sV^p$, and let $(Y,Y') \in \cV^{p,q,\infty}_X$ such that $\|\bX\|_{p,[0,T]}$, $\sup_{s \in [0,T]} \|Y_s\|_{L^\infty}$, $\sup_{s \in [0,T]} \|Y'_s\|_{L^\infty}$, $\|Y\|_{p,q,\infty,[0,T]}$, $\|Y'\|_{p,q,\infty,[0,T]}$ and $\|\E_{\edot} R^Y\|_{\frac{p}{2},\infty,[0,T]}$ are all bounded by a constant $L > 0$. Then the rough stochastic integral $\int_0^\cdot Y_s \dd \bX_s \in \textup{BMO}^{p\textup{-var}}$. Moreover, for any $\alpha > 0$ and $\lambda > 0$, there exists a constant $C$, which depends only on $p, q, \alpha, \lambda$ and $L$, such that, for every $r \in [0,T]$, we have that
\begin{equation}\label{eq: exponential bound for rough stochastic integral}
\bigg\| \E_r \bigg[ \exp \bigg( \lambda \sup_{t \in [r,T]} \bigg| \int_r^t Y_s \dd \bX_s \bigg| \bigg) \bigg] \bigg\|_{L^\infty} \leq \exp \Big( C \Big(N_{\alpha,[r,T]}\big(\|\bX\|_{p,[\cdot,\cdot)}^p\big) + 1 \Big) \Big(1 + \sup_{t \in (r,T]} |\Delta \bX_t| \Big) \Big).
\end{equation}
\end{proposition}

\begin{proof}
For each $(s,t) \in \Delta_{[0,T]}$, we let
\begin{equation*}
\Gamma_{s,t} := \sup_{r \in [s,t]} \bigg|\Delta \bigg(\int_0^\cdot Y_u \dd \bX_u\bigg)_{\hspace{-2pt}r}\bigg| = \sup_{r \in [s,t]} \big|Y_{r-} \Delta X_r + Y'_{r-} \Delta \X_r\big|,
\end{equation*}
where the second equality follows from the jump structure of rough stochastic integrals; see \cite[Lemma~4.7]{AllanPieper2026}. Let us write $J_\bX \subset (0,T]$ for the (countable) set of jump times of $\bX = (X,\X)$. Then, for any partition $\cP$, we have that
\begin{align*}
\sum_{[u,v] \in \cP} \|\Gamma_{u,v}\|_{L^\infty}^p &= \sum_{[u,v] \in \cP} \sup_{r \in [u,v] \cap J_\bX} \big\|Y_{r-} \Delta X_r + Y'_{r-} \Delta \X_r\big\|_{L^\infty}^p\\
&\leq 2 \sum_{r \in J_\bX} \big(\|Y_{r-}\|_{L^\infty} + \|Y'_{r-}\|_{L^\infty}\big)^p |\Delta \bX_r|^p\\
&\lesssim \sup_{u \in [0,T]} \big(\|Y_u\|_{L^\infty} + \|Y'_u\|_{L^\infty}\big)^p \|\bX\|_{p,[0,T]}^p \leq 2^p L^p \|\bX\|_{p,[0,T]}^p.
\end{align*}
It follows that $\|\Gamma\|_{p,\infty,(s,t)} \lesssim \|\bX\|_{p,[s,t)}$ for any $(s,t) \in \Delta_{[0,T]}$, and it is also straightforward to see that
\begin{equation*}
\sup_{t \in (0,T]} \bigg\| \Delta \bigg(\int_0^\cdot Y_u \dd \bX_u\bigg)_{\hspace{-2pt}t} \bigg\|_{L^\infty} \lesssim \sup_{t \in (0,T]} |\Delta \bX_t|.
\end{equation*}
By \cite[Lemma~4.8]{AllanPieper2026}, we also have that
\begin{equation*}
\bigg\| \int_0^\cdot Y_u \dd \bX_u \bigg\|_{p,q,\infty,[s,t)} \lesssim \|\bX\|_{p,[s,t)}
\end{equation*}
for any $(s,t) \in \Delta_{[0,T]}$. It thus follows from Proposition~\ref{proposition: BMO-p-var is p,q,infty} that $\int_0^\cdot Y_u \dd \bX_u \in \textup{BMO}^{p\textup{-var}}$, and the bound in \eqref{eq: exponential bound for rough stochastic integral} follows from Theorem~\ref{theorem: John-Nirenberg inequality for p,q,infty}, combined with Lemma~\ref{lemma: N_alpha for w_1 leq w_2} and Lemma~\ref{lemma: N_alpha < (2 N_beta+1)}.
\end{proof}


\subsection{Locally affine maps}

We generalize the notion of locally linear maps in \cite[Section~4]{FrizRiedel2013} to the setting of rough stochastic analysis. Taking the It\^o--Lyons map as prime example, the presence of additional noise terms, not present in classical RDEs, necessitates a slightly different definition.

\begin{definition}
Let $p \in [2,3)$ and $q \in [2, \infty)$. We call a map $\Psi \colon \sV^p \to V^p L^{q,\infty}$ \emph{locally affine} if there exists an $R \in (0,\infty]$, and a control $\bar{w}$ which is regular from the inside (in the sense of Definition~\ref{definition: regular from the inside}), such that $\|\Psi\|_{R,\bar{w}} < \infty$, where
\begin{align*}
\|\Psi\|_{R,\bar{w}} := \inf \big\{ C > 0 \, : \ &\|\Psi(\bX)\|_{p,q,\infty,[s,t)}^p \leq C \bar{w}_{\bX}(s,t)\\
&\text{for all } \bX \in \sV^p \text{ and } (s,t) \in \Delta_{[0,T]} \text{ such that } \bar{w}_{\bX}(s,t) \leq R \big\}
\end{align*}
and $\bar{w}_{\bX}(s,t) := \|\bX\|_{p,[s,t)}^p + \bar{w}(s,t)$.
\end{definition}

With this definition, we easily conclude an analogue of \cite[Proposition~2]{FrizRiedel2013} in our context. The following result follows as a direct consequence of Lemmas~\ref{lemma: N_alpha for w_1 leq w_2} and \ref{lemma: bound on N_alpha(w_1+ w_2)}.

\begin{lemma}
Let $\Psi \colon \sV^p \to V^p L^{q,\infty}$ be a locally affine map, for some $R \in (0,\infty]$ and control $\bar{w}$. Then
\begin{equation*}
N_{\alpha \|\Psi\|_{R,\bar{w}},[s,t]}\big(\|\Psi(\bX)\|_{p,q,\infty,[\cdot,\cdot)}^p\big) \leq N_{\frac{\alpha}{2},[s,t]}\big(\|\bX\|_{p,[\cdot,\cdot)}^p\big) + N_{\frac{\alpha}{2},[s,t]}(\bar{w})
\end{equation*}
for every $\alpha \in (0,R]$.
\end{lemma}

\begin{proposition}\label{proposition: The Ito-Lyons map is locally affine}
Under the assumptions of Theorem~\ref{theorem: existence and estimates for solutions to RSDEs with measures}, the It\^o--Lyons map
\begin{equation*}
\sV^p \ni \bX \mapsto Y \in V^p L^{q,\infty},
\end{equation*}
where $Y$ is the solution to the rough SDE \eqref{eq: RSDE with measure} driven by $\bX$, is locally affine, with any $R \in (0,\infty)$ and the control $\bar{w}$ given by
\[\bar{w}(s,t) := (t - s) + \|M\|_{p,q,\infty,[s,t)}^p + \|A\|_{\frac{p}{2},\frac{q}{2},\infty,[s,t)}^{\frac{p}{2}} + \|A\|_{\frac{p}{q},1,\infty,[s,t)}^{\frac{p}{q}}.\]

Moreover, for any $h \in C^2_b$, the map $\bX \mapsto \int_0^{\cdot} h(Y_s) \dd \bX_s$ is also locally affine, with any $R \in (0,\infty)$ and the same control $\bar{w}$. In particular, if $\bX \in \sV^p$ with $\|\bX\|_{p,[0,T]}, \bar{w}(0,T) \leq L$ for some $L > 0$, then for any $\alpha > 0$, $\lambda > 0$ and $r \in [0,T]$,
\begin{equation}\label{eq: exponential bound for rough stochastic integral of RSDE solutions}
\bigg\| \E_r \bigg[ \exp \bigg( \lambda \sup_{t \in [r,T]} \bigg| \int_r^t h(Y_s) \dd \bX_s \bigg| \bigg) \bigg] \bigg\|_{L^\infty} \leq \exp \Big( C \Big(N_{\alpha,[r,T]}\big(\|\bX\|_{p,[\cdot,\cdot)}^p\big) + 1 \Big) \Big(1 + \sup_{t \in (r,T]} |\Delta \bX_t| \Big) \Big),
\end{equation}
where the constant $C$ depends only on $p, q, \alpha, \lambda$ and $L$.
\end{proposition}

\begin{proof}
By Step~1 in the proof of \cite[Theorem~5.2]{AllanPieper2026} (suitably adapted to include the integral against a random measure), there exists an $\epsilon \in (0,1]$, which does not depend on $\bX$, such that, for every $(s,t) \in \Delta_{[0,T]}$ with $\bar{w}_\bX(s,t) \leq \epsilon$, we have that
\begin{equation}\label{eq: bound for Y and ER^Y in loc affine proof}
\|Y\|_{p,q,\infty,[s,t)}^p \vee \|\E_{\edot} R^Y\|_{\frac{p}{2},\infty,[s,t)}^{\frac{p}{2}} \lesssim \bar{w}_\bX(s,t).
\end{equation}
Since $Y' = f(Y)$ and $f$ is Lipschitz, we also have that $\|Y'\|_{p,q,\infty,[s,t)}^p \lesssim \|Y\|_{p,q,\infty,[s,t)}^p \lesssim \bar{w}_\bX(s,t)$.

We now fix any $\bX \in \sV^p$ and $(s,t) \in \Delta_{[0,T]}$ such that $w_{\bX}(s,t) \leq R$. By \cite[Lemma~1.5]{FrizZhang2018}, there exists a partition $\{t_i\}_{i=0}^M$ of the interval $[s,t]$ such that $w_{\bX}(t_i,t_{i+1}) \leq \epsilon$ for each $0 \leq i < M$. By the superadditivity of $\bar{w}_\bX$, we may also choose this partition such that $M \lesssim \bar{w}_\bX(s,t) \leq R$.

For $0 \leq i \leq M-2$, using \cite[Lemma~4.8]{AllanPieper2026}, and the fact that
\[\|f(Y)\|_{p,q,\infty,[v,t_{i+1})} + \|f(Y)'\|_{p,q,\infty,[v,t_{i+1})} + \|\E_{\edot} R^{f(Y)}\|_{\frac{p}{2},\infty,[v,t_{i+1})} \to 0\]
as $v \nearrow t_{i+1}$, we see that
\begin{equation*}
\lim_{v \nearrow t_{i+1}} \bigg\| \int_v^{t_{i+1}} f(Y_r) \dd \bX_r \bigg\|_{q,\infty,v}^p \lesssim \|\bX\|_{p,[t_i,t_{i+1}]}^p \leq \|\bX\|_{p,[s,t)}^p.
\end{equation*}
Using the bound in \eqref{eq: bound for Y and ER^Y in loc affine proof}, the conditional BDG inequality and Lemma~\ref{lemma: application of conditional BDG for jump measures}, we then have that
\begin{align*}
\|&Y\|_{p,q,\infty,[t_i,t_{i+1}]}^p \lesssim \|Y\|_{p,q,\infty,[t_i,t_{i+1})}^p + \lim_{v \nearrow t_{i+1}} \|\delta Y_{v,t_{i+1}}\|_{q,\infty,v}^p\\
&\lesssim \bar{w}_\bX(t_i,t_{i+1}) + \lim_{v \nearrow t_{i+1}} \bigg\| \int_v^{t_{i+1}} \sigma(Y_{r-}) \dd M_r \bigg\|_{q,\infty,v}^p + \lim_{v \nearrow t_{i+1}} \bigg\| \int_v^{t_{i+1}} f(Y_r) \dd \bX_r \bigg\|_{q,\infty,v}^p\\
&\quad + \lim_{v \nearrow t_{i+1}} \bigg\| \int_v^{t_{i+1}} \int_{\U} g(r,Y_{r-},u) \, \tN(\d u,\d r) \bigg\|_{q,\infty,v}^p\\
&\lesssim \bar{w}_\bX(t_i,t_{i+1}) + \|M\|_{p,q,\infty,[t_i,t_{i+1}]}^p + \|A\|_{\frac{p}{2},\frac{q}{2},\infty,[t_i,t_{i+1}]}^{\frac{p}{2}} + \|A\|_{\frac{p}{q},1,\infty,[t_i,t_{i+1}]}^{\frac{p}{q}} + \|\bX\|_{p,[s,t)}^p\\
&\lesssim \bar{w}_\bX(s,t).
\end{align*}
It is also clear that $\|Y\|_{p,q,\infty,[t_{M-1},t_M)}^p \lesssim \bar{w}_\bX(t_{M-1},t_M) \leq \bar{w}_\bX(s,t)$, and we can thus bound
\begin{equation*}
\|Y\|_{p,q,\infty,[s,t)}^p \leq M^{p-1} \bigg( \sum_{i=0}^{M-2} \|Y\|_{p,q,\infty,[t_i,t_{i+1}]}^p + \|Y\|_{p,q,\infty,[t_{M-1},t_M)}^p \bigg) \lesssim \bar{w}_\bX(s,t),
\end{equation*}
which implies that the map $\bX \mapsto Y$ is locally affine.

With the same line of argument, one can conclude similar bounds for $\|Y'\|_{p,q,\infty,[s,t)}^p$ and $\|\E_{\edot} R^Y\|_{\frac{p}{2},\infty,[s,t)}^{\frac{p}{2}}$, and the final claim then follows by combining these bounds with those in \cite[Lemma~4.8]{AllanPieper2026}, and arguing exactly as in the proof of Proposition~\ref{proposition: exponential integrability of a rough stochastic integral}.
\end{proof}

\section{Robust stochastic filtering with jumps}\label{sec: application to filtering}

On a filtered probability space $(\Omega,\cF,(\cF_t)_{t \in [0,T]},\P)$, we let $B$ and $W$ be independent Brownian motions. Moreover, we let $N_1$ and $N_2$ be integer-valued random measures, defined on Blackwell spaces $(\U_1,\cU_1)$, $(\U_2,\cU_2)$, and let $\tN_1$, $\tN_2$ denote the corresponding compensated random measures.

We consider a \emph{signal} process $X$, and an \emph{observation} process $Y$, governed by the SDEs
\begin{equation}\label{eq: SDEofX and Y}
\begin{split}
\d X_t &= b_1(t,X_t,Y_t) \dd t + \sigma_0(t,X_t,Y_t) \dd B_t + \sigma_1(t,X_t,Y_t) \dd W_t\\
&\quad + \int_{\U_1} f_1(t,X_{t-},Y_{t-},u) \, \tN_1(\d t,\d u) + \int_{\U_2} f_2(t,X_{t-},Y_{t-},u) \, \tN_2(\d t,\d u),\\
\d Y_t &= b_2(t,X_t,Y_t) \dd t + \sigma_2(t,Y_t) \dd W_t + \int_{\U_2} f_3(t,Y_{t-},u) \, \tN_2(\d t,\d u).
\end{split}
\end{equation}
We suppose that compensators of $N_1$ and $N_2$ are given by $\nu_1(\d u) \dd t$ and $\lambda(t,X_{t-},u) \, \nu_2(\d u) \dd t$ respectively, where $\nu_1$, $\nu_2$ are $\sigma$-finite measures on the respective Blackwell spaces, and $\lambda \colon [0,T] \times \R^{d_X} \times \U_2 \to (0,\infty)$ is a Borel measurable function. In particular, we note that $N_1$ is a Poisson random measure, but $N_2$ is not in general.

Here, the coefficients $b_1 \colon [0,T] \times \R^{d_X + d_Y} \to \R^{d_X}$, $b_2 \colon [0,T] \times \R^{d_X + d_Y} \to \R^{d_Y}$, $\sigma_0 \colon [0,T] \times \R^{d_X + d_Y} \to \R^{d_X \times d_B}$, $\sigma_1 \colon [0,T] \times \R^{d_X + d_Y} \to \R^{d_X \times d_Y}$, $\sigma_2 \colon [0,T] \times \R^{d_Y} \to \R^{d_Y \times d_Y}$, $f_1 \colon [0,T] \times \R^{d_X + d_Y} \times \U_1 \to \R^{d_X}$, $f_2 \colon [0,T] \times \R^{d_X + d_Y} \times \U_2 \to \R^{d_X}$ and $f_3 \colon [0,T] \times \R^{d_Y} \times \U_2 \to \R^{d_Y}$ are all Borel measurable, and $d_X, d_Y, d_B \in \N$ are the dimensions of the respective processes. We note that this setting includes, for instance, those in \cite{GermGyongy2025PartI} and \cite{Qiao2021}.

\emph{Stochastic filtering} is concerned with calculating the conditional law of $X_t$ given the filtration generated by $Y$ up to time $t$. All the characteristics of $X$ and $Y$ are assumed to be known by the observer; see also Remark~\ref{rem:reconstruction}. \emph{Robust stochastic filtering} establishes continuity properties of the conditional law with respect to the observation $Y$, which is of course crucial in applications.

\subsection{Preliminaries on stochastic filtering}

We require the following assumptions.

\begin{assumption}\label{assumption: on the coefficients for the SDE, linear growth}
The coefficients have at most linear growth. That is, there exists a constant $K > 0$ such that
\begin{align*}
|b_1(t,x,y)| + |b_2(t,x,y)| + |\sigma_0(t,x,y)| + |\sigma_1(t,x,y)| + |\sigma_2(t,y)| &\leq K (1 + |x| + |y|),\\
\|f_1(t,x,y,\cdot)\|_{L^2(\nu_1)} + \|f_2(t,x,y,\cdot)\|_{L^2(\nu_2)} + \|f_3(t,y,\cdot)\|_{L^2(\nu_2)} &\leq K(1 + |x| + |y|)
\end{align*}
for all $t \in [0,T]$, $x \in \R^{d_X}$ and $y \in \R^{d_Y}$.
\end{assumption}

\begin{assumption}\label{assumption: boundedness of Z_0}
The initial values $X_0$, $Y_0$ are $\cF_0$-measurable and square integrable.
\end{assumption}

\begin{assumption}
Almost surely, the random measures $N_1$ and $N_2$ do not jump at the same time, in the sense that $\P(\{\omega \in \Omega : (\omega,t) \in D_1 \cap D_2 \ \textup{ for any } \ t \in [0,T]\}) = 0$, where $D_i$ denotes the set of jump times of $N_i$, as defined in \eqref{eq: definition D_i and beta^i}.
\end{assumption}

\begin{assumption}\label{assumption: non-degeneracy of sigma and boundedness of h}
The map $\sigma_2 \colon [0,T] \times \R^{d_Y} \to \R^{d_Y \times d_Y}$ takes values in the space of invertible matrices, and
\begin{equation*}
\sup_{t \in [0,T], \, y \in \R^{d_Y}} |\sigma_2(t,y)^{-1}| < \infty.
\end{equation*}
Further the map $h \colon [0,T] \times \R^{d_X + d_Y} \to \R^{d_Y}$ given by
\begin{equation}\label{eq: defn function h}
h(t,x,y) := \sigma_2(t,y)^{-1} \bigg( b_2(t,x,y) + \int_{\U_2} f_3(t,y,u) (1 - \lambda(t,x,u)) \, \nu_2(\d u) \bigg)
\end{equation}
satisfies
\begin{equation*}
\sup_{(x,y) \in \R^{d_X + d_Y}} \int_0^T |h(s,x,y)|^2 \dd s < \infty.
\end{equation*}
\end{assumption}

\begin{assumption}\label{assumption: on lambda for the measure change}
The function $\lambda$ is uniformly bounded and uniformly bounded away from $0$. Further,
\begin{equation}\label{eq: assumption on lambda}
\sup_{t \in [0,T], \, x \in \R^{d_X}} \int_{\U_2} \frac{(1 - \lambda(t,x,u))^2}{\lambda(t,x,u)} \, \nu_2(\d u) < \infty.
\end{equation}
\end{assumption}

Conditions of the form in \eqref{eq: assumption on lambda} are standard for filtering models with L\'evy noise; see, e.g., \cite[Assumption~(6)]{Gertner1978}, \cite[Assumption~($\boldsymbol{\mathrm{H}}_\lambda$)]{Qiao2021}, \cite[Assumption~5]{Qiao2023} or \cite[Assumption~3.2]{QiaoDuan2015}.

\smallskip

The following lemma provides some required integrability. Since this result is rather standard (see, e.g., \cite[Lemma~16.1.4]{Cohen2015} or \cite[Remark~2.1]{GermGyongy2025PartI}), we omit its proof.

\begin{lemma}
Suppose that $X$ and $Y$ satisfy the SDEs in \eqref{eq: SDEofX and Y}. Then, under Assumptions~\ref{assumption: on the coefficients for the SDE, linear growth} and \ref{assumption: boundedness of Z_0}, we have that
\begin{equation*}
\E \Big[ \sup_{t \in [0,T]} \big(|X_t|^2 + |Y_t|^2\big) \Big] \leq C \big( 1 + \E \big[|X_0|^2\big] + \E \big[|Y_0|^2\big] \big) < \infty,
\end{equation*}
where the constant $C$ depends in particular on the constant $K$ in Assumption~\ref{assumption: on the coefficients for the SDE, linear growth}, which implies in particular that
\begin{equation*}
\E \bigg[ \int_0^T \int_{\U_i} |f_i(s,X_{s-},Y_{s-},u)|^2 \, \nu_i(\d u) \dd s \bigg] < \infty
\end{equation*}
for each $i = 1, 2, 3$ (with $\U_3 := \U_2$, $\nu_3 := \nu_2$ and without the dependence on $X$ when $i = 3$).
\end{lemma}

A standard approach in stochastic filtering, which we will utilize here, is the so-called \emph{reference measure method}. In the following, we will define a probability measure $\tP$, which is equivalent to $\P$ on $\cF_T$, under which the noises driving the observation process $Y$ become independent of the additional noises driving the signal process $X$. In our setting, this is achieved by setting
\begin{equation*}
\frac{\d \tP}{\d \P}\bigg|_{\cF_T} = \Lambda_T^{-1},
\end{equation*}
where
\begin{equation*}
\Lambda_t^{-1} := \exp(-I_t)
\end{equation*}
for $t \in [0,T]$, and the process $I$ is defined by
\begin{equation}\label{eq: defn I_t}
\begin{split}
I_t &= \int_0^t h(s,X_s,Y_s)^\top \dd W_s + \frac{1}{2} \int_0^t |h(s,X_s,Y_s)|^2 \dd s + \int_0^t \int_{\U_2} \log(\lambda(s,X_{s-},u)) \, \tN_2(\d s,\d u)\\
&\quad + \int_0^t \int_{\U_2} \big( 1 - \lambda(s,X_{s-},u) + \lambda(s,X_{s-},u) \log(\lambda(s,X_{s-},u)) \big) \, \nu_2(\d u) \dd s
\end{split}
\end{equation}
for $t \in [0,T]$, and the function $h$ was defined in \eqref{eq: defn function h}, so that
\begin{equation*}
h(s,X_s,Y_s) = \sigma_2(s,Y_s)^{-1} \bigg( b_2(s,X_s,Y_s) + \int_{\U_2} f_3(s,Y_{s-},u) (1 - \lambda(s,X_{s-},u)) \, \nu_2(\d u) \bigg).
\end{equation*}

Since $(v - 1)/v \leq \log(v) \leq v - 1$ for all $v > 0$, and $\lambda$ is uniformly bounded, we note that
\begin{equation*}
\big| 1 - \lambda(s,x,u) + \lambda(s,x,u) \log(\lambda(s,x,u)) \big| \leq (1 - \lambda(s,x,u))^2 \lesssim \frac{(1 - \lambda(s,x,u))^2}{\lambda(s,x,u)},
\end{equation*}
which, combined with Assumption~\ref{assumption: on lambda for the measure change}, ensures the existence of the final integral in \eqref{eq: defn I_t}.

\smallskip

The following few lemmas are also standard, and their proofs are rather technical, and are therefore omitted. See, for instance, \cite{Gertner1978}, \cite{Qiao2021}, \cite[Ch.~3]{BainCrisan2009} or \cite[Ch.~6.2]{MandrekarRudiger2014} for similar results in analogous settings.

\begin{lemma}
Under Assumptions~\ref{assumption: on the coefficients for the SDE, linear growth}, \ref{assumption: boundedness of Z_0}, \ref{assumption: non-degeneracy of sigma and boundedness of h} and \ref{assumption: on lambda for the measure change}, we have that $\Lambda^{-1} = (\Lambda^{-1}_t)_{t \in [0,T]}$ is the Dol\'eans--Dade exponential of the process given by
\[t \mapsto -\int_0^t h(s,X_s,Y_s)^\top \dd W_s + \int_0^t \int_{\U_2} \frac{1 - \lambda(s,X_{s-},u)}{\lambda(s,X_{s-},u)} \, \tN_2(\d s,\d u),\]
and thus in particular is an exponential martingale.
\end{lemma}

We let $\tW$ be the process given by $\tW_t = W_t + \int_0^t h(s,X_s,Y_s) \dd s$ for $t \in [0,T]$, and we also let $\tilde{b}_1(t,x,y) := b_1(t,x,y) - \sigma_1(t,x,y) h(t,x,y) - \int_{\U_2} f_2(t,x,y,u) (1 - \lambda(t,x,u)) \, \nu_2(\d u)$.

\begin{lemma}\label{lemma: application of Girsanov for measure change}
Under Assumptions~\ref{assumption: on the coefficients for the SDE, linear growth}--\ref{assumption: on lambda for the measure change}, we have, under the measure $\tP$, that $B$ and $\tW$ are Brownian motions, $N_1$ and $N_2$ are Poisson random measures with compensators $\nu_1(\d u) \dd t$ and $\nu_2(\d u) \dd t$ respectively, and that $B, \tW, N_1$ and $N_2$ are all independent.

Moreover, writing $\tN$ for the compensated random measure associated with $N_2$ under $\tP$, i.e., $\tN(\d t,\d u) = N_2(\d t,\d u) - \nu_2(\d u) \dd t$, the signal $X$ and observation $Y$ satisfy
\begin{equation}\label{eq: SDEFilteringUpdated}
\begin{split}
\d X_t &= \tilde{b}_1(t,X_t,Y_t) \dd t + \sigma_0(t,X_t,Y_t) \dd B_t + \sigma_1(t,X_t,Y_t) \dd \tW_t\\
&\quad + \int_{\U_1} f_1(t,X_{t-},Y_{t-},u) \, \tN_1(\d t,\d u) + \int_{\U_2} f_2(t,X_{t-},Y_{t-},u) \, \tN(\d t,\d u),\\
\d Y_t &= \sigma_2(t,Y_t) \dd \tW_t + \int_{\U_2} f_3(t,Y_{t-},u) \, \tN(\d t,\d u).
\end{split}
\end{equation}
\end{lemma}

Under standard assumptions (see, e.g., \cite{DavieGermGyongy2024} or \cite[Ch.~16]{Cohen2015}), there exists a unique solution $X, Y$ to the SDEs in \eqref{eq: SDEFilteringUpdated}, which is then also the solution to the SDEs in \eqref{eq: SDEofX and Y}.

\smallskip

The conditional distribution for our filtering model is given by the following result, known as the \emph{Kallianpur--Striebel formula}. In the following, we will write $\widetilde{\E}$ for the expectation under $\tP$, and write $(\cF^Y_t)_{t \in [0,T]}$ for the observation filtration, i.e., for each $t \in [0,T]$, $\cF^Y_t$ is the $\P$-completion of $\sigma(Y_s, s \in [0,t])$.

\begin{lemma}\label{lemmaKallianpurStriebel}
For any bounded measurable function $f \colon \R^{d_X + d_Y} \to \R$, we have that
\begin{equation*}
\pi_t(f) := \E[f(X_t,Y_t) \, | \, \cF^Y_t] = \frac{\widetilde{\E}[f(X_t,Y_t) \Lambda_t \, | \, \cF^Y_t]}{\widetilde{\E}[\Lambda_t \, | \, \cF^Y_t]} =: \frac{\rho_t(f)}{\rho_t(1)}
\end{equation*}
for $t \in [0,T]$, where $\Lambda_t = \exp(I_t)$. Moreover, $\Lambda = (\Lambda_t)_{t \in [0,T]}$ is the Dol\'eans--Dade exponential of the process
\begin{equation*}
t \mapsto \int_0^t h(s,X_s,Y_s)^\top \dd \tW_s - \int_0^t \int_{\U_2} (1 - \lambda(s,X_{s-},u)) \, \tN(\d s,\d u),
\end{equation*}
and thus is itself an exponential martingale.
\end{lemma}

\subsection{Reformulation as a rough SDE}\label{subsection: reformulation as a rough SDE}

Similarly to the approach adopted in \cite{CrisanDiehlFrizOberhauser2013}, rather than working with the SDEs \eqref{eq: SDEFilteringUpdated} on a common probability space, we consider analogous SDEs on a product probability space, such that the additional noise terms included in the signal are defined on an independent part of the space to the observation. By the independence of these terms, as provided in Lemma~\ref{lemma: application of Girsanov for measure change}, the solution to the SDEs on the product space has the same distribution as the solution to the original SDEs in \eqref{eq: SDEFilteringUpdated}. Moreover, we can then lift all the $\cF^Y_t$-adapted processes to rough paths, and, given the consistency result of Section~\ref{sec: Consistency RSDEs and SDEs}, we can consider the resulting equation as a rough SDE, and exploit the stability properties of such equations to obtain robustness of the conditional distribution. To this end, we need the following assumption.

\begin{assumption}\label{assumption: FY= FW v FN}
We assume that
\begin{equation*}
\cF^Y_t = \cF^Y_0 \vee \cF^{\tW}_t \vee \cF^{\tN}_t
\end{equation*}
for every $t \in [0,T]$. Here, $(\cF^{\tW}_t)_{t \in [0,T]}$ denotes the natural filtration of $\tW$, and $\cF^{\tN}_t$ is the $\tP$-completion of the $\sigma$-algebra generated by the random variables $N_2((0,s] \times U)$ for all $s \in (0,t]$ and all $U \in \cU_2$ such that $\nu_2(U) < \infty$. We also assume that $\cF^Y_0$ is independent of $\cF^B_T \vee \cF^{\tN_1}_T$, where $(\cF^B_t)_{t \in [0,T]}$ is the natural filtration of $B$, and $\cF^{\tN_1}_t$ is defined analogously to $\cF^{\tN}_t$.
\end{assumption}

\begin{remark}\label{remark: some cases when FY= FW v FN}
Assumption~\ref{assumption: FY= FW v FN} is essentially an additional condition on the coefficient $f_3$.

Suppose, for instance, that $f_3(s,y,u) = h_3(s,y) g_3(s,u)$, where $h_3$ takes values in the space of invertible $d_Y \times d_Y$-matrices with $\sup_{s \in [0,T], \, y \in \R^{d_Y}} |h_3(s,y)^{-1}| < \infty$, and $g_3 \colon [0,T] \times \U_2 \to \R^{d_Y} \setminus \{0\}$ is such that $\sigma(\iota|_{[0,t] \times \U_2}) = \cB([0,t]) \otimes \cU_2$ for every $t \in [0,T]$, where the function $\iota \colon [0,T] \times \U_2 \to [0,T] \times (\R^{d_Y} \setminus \{0\})$ is given by $\iota(s,u) = (s,g_3(s,u))$. An extension of \cite[Lemma~4.2]{GermGyongy2025PartI} then shows that Assumption~\ref{assumption: FY= FW v FN} is satisfied.
\end{remark}

As indicated above, our intention is to fix a realization of certain noise terms in \eqref{eq: SDEFilteringUpdated}, and then lift these to a rough path. While this procedure is classical for It\^o integrals against Brownian motion, the Poisson random measures cannot immediately be lifted to a rough path; indeed, a priori there is no path to lift. We circumvent this issue by splitting the integrand into a product, to isolate the dependence on $X$ and $Y$ from the dependence on $u$.

\begin{assumption}\label{assumption: f= hg}
We suppose that $\tilde{b}_1, \sigma_0 \in C^1_b$, $h \in C^2_b$ and $\sigma_1, \sigma_2 \in C^3_b$. Further, we assume that $f_2(t,x,y,u) = h_2(t,x,y) g_2(t,u)$ and $f_3(t,y,u) = h_3(t,y) g_3(t,u)$, where $h_2, h_3 \in C^3_b$, and, for each $i = 2, 3$, the function $g_i$ is Borel measurable and satisfies
\begin{equation*}
\int_0^T \int_{\U_2} |g_i(s,u)|^2 \, \nu_2(\d u) \dd s < \infty.
\end{equation*}
Moreover, we assume that $\log(\lambda(t,x,u)) = \kappa(t,x) \gamma(t,u)$, where $\kappa \in C^2_b$ and $\gamma$ is Borel measurable, bounded, and satisfies
\begin{equation*}
\int_0^T \int_{\U_2} |\gamma(s,u)| \, \nu_2(\d u) \dd s < \infty.
\end{equation*}
\end{assumption}

\begin{remark}
The conditions on $f_2, f_3$ and $\lambda$ in Assumption~\ref{assumption: f= hg} may seem somewhat restrictive at first glance. However, we recall that finite sums of products of smooth functions, of the form $\sum_{j=1}^m h_{2,j}(t,x,y) g_{2,j}(t,u)$, are dense in the space of continuous compactly supported functions of $(t,x,y,u)$, by a straightforward application of the Stone--Weierstrass theorem. Indeed, all of our results below hold when the functions $f_2, f_3$ and $\lambda$ are given by finite sums of products, i.e., when $f_2(t,x,y,u) = \sum_{j=1}^m h_{2,j}(t,x,y) g_{2,j}(t,u)$ (and similarly for $f_3$), and $\log(\lambda(t,x,u)) = \sum_{j=1}^m \kappa_j(t,x) \gamma_j(t,u)$, where the functions $h_{2,j}, h_{3,j}, g_{2,j}, g_{3,j}, \kappa_j$ and $\gamma_j$ satisfy the relevant conditions in Assumption~\ref{assumption: f= hg}. It what follows we take singular products purely for notational simplicity.
\end{remark}

Under Assumption~\ref{assumption: f= hg}, and using the associativity of integrals against random measures (e.g., \cite[Ch.~II, Proposition~1.30]{JacodShiryaev2003}), we may rewrite \eqref{eq: SDEFilteringUpdated} as
\begin{equation}\label{eq: dynamics of X Y driven by G}
\begin{split}
\d X_t &= \tilde{b}_1(t,X_t,Y_t) \dd t + \sigma_0(t,X_t,Y_t) \dd B_t + \int_{\U_1} f_1(t,X_{t-},Y_{t-},u) \, \tN_1(\d t,\d u)\\
&\quad + (\sigma_1(t,X_{t-},Y_{t-}), h_2(t,X_{t-},Y_{t-}), 0, 0) \dd G_t,\\
\d Y_t &= (\sigma_2(t,Y_{t-}), 0, h_3(t,Y_{t-}), 0) \dd G_t,
\end{split}
\end{equation}
where we set
\begin{equation}\label{eq: definition G}
G_t = \bigg( \tW_t, \int_0^t \int_{\U_2} g_2(s,u) \, \tN(\d s,\d u), \int_0^t \int_{\U_2} g_3(s,u) \, \tN(\d s,\d u), \int_0^t \int_{\U_2} \gamma(s,u) \, \tN(\d s,\d u) \bigg)^{\hspace{-2pt}\top}
\end{equation}
for $t \in [0,T]$, which, by Assumption~\ref{assumption: FY= FW v FN}, defines an $\cF^Y_t$-adapted local martingale $G = (G_t)_{t \in [0,T]}$. By Lemma~\ref{lemma: application of Girsanov for measure change}, we also have that $G$ is independent of $B$ and $N_1$.

\smallskip

In the following, we will continue to consider $N_2, \tW$ and $G$ as being defined on $(\Omega,\cF,\tP)$. However, we now introduce a second probability space, which we denote by $(\bar{\Omega},\bar{\cF},\bar{\P})$. On this new space, we let $B$ be a Brownian motion, and let $N_1$ be a Poisson random measure with compensator $\nu_1(\d u) \dd t$. We also define the product space
\[(\hat{\Omega},\hat{\cF},\hat{\P}) := (\Omega \times \bar{\Omega},\cF \otimes \bar{\cF},\tP \otimes \bar{\P}).\]

Of course, we can consider all stochastic objects as also living on the product space, by simply letting, e.g., $B(\omega,\bar{\omega}) = B(\bar{\omega})$ for all $(\omega,\bar{\omega}) \in \hat{\Omega}$. Moreover, by the independence provided in Lemma~\ref{lemma: application of Girsanov for measure change}, it is clear that, after this change of framework, the law of these stochastic objects (under the measure $\hat{\P}$) is unchanged from how they were defined originally (under $\tP$).

Let $X$ and $Y$ be the solutions to the SDEs in \eqref{eq: dynamics of X Y driven by G}, now defined on $(\hat{\Omega},\hat{\cF},\hat{\P})$. We also define
\begin{equation}\label{eq: defn I on product space}
\begin{split}
I_t &:= \int_0^t h(s,X_s,Y_s)^\top \dd \tW_s - \frac{1}{2} \int_0^t |h(s,X_s,Y_s)|^2 \dd s + \int_0^t \int_{\U_2} \kappa(s,X_{s-}) \gamma(s,u) \, \tN(\d s,\d u)\\
&\quad + \int_0^t \big(1 - \lambda(s,X_{s-},u) + \log(\lambda(s,X_{s-},u))\big) \, \nu_2(\d u) \dd s\\
&= \int_0^t H(s,X_{s-},Y_{s-}) \dd G_s - \frac{1}{2} \int_0^t |h(s,X_s,Y_s)|^2 \dd s\\
&\quad + \int_0^t \big(1 - \lambda(s,X_{s-},u) + \log(\lambda(s,X_{s-},u))\big) \, \nu_2(\d u) \dd s,
\end{split}
\end{equation}
where
\begin{equation*}
H(t,x,y) := \big(h(t,x,y)^\top, 0, 0, \kappa(t,x)\big).
\end{equation*}
Again, it is clear that $(X,Y,I)$ has the same law under $\hat{\P}$ as it originally had under $\tP$.

\smallskip

For any (deterministic) c\`adl\`ag rough path $\boldsymbol{\eta} \in \sV^p$, we denote by $(X^{\boldsymbol{\eta}}, Y^{\boldsymbol{\eta}}) \in V^p L^{2,\infty}(\bar{\Omega})$ the solution to the rough SDE
\begin{equation}\label{eq: filtering RSDE}
\begin{split}
\d X^{\boldsymbol{\eta}}_t &= \tilde{b}_1(t,X^{\boldsymbol{\eta}}_t,Y^{\boldsymbol{\eta}}_t) \dd t + \sigma_0(t,X^{\boldsymbol{\eta}}_t,Y^{\boldsymbol{\eta}}_t) \dd B_t + \int_{\U_1} f_1(t,X^{\boldsymbol{\eta}}_{t-},Y^{\boldsymbol{\eta}}_{t-},u) \, \widetilde{N}_1(\d t,\d u)\\
&\quad + (\sigma_1(t,X^{\boldsymbol{\eta}}_t,Y^{\boldsymbol{\eta}}_t), h_2(t,X^{\boldsymbol{\eta}}_t,Y^{\boldsymbol{\eta}}_t), 0, 0) \dd \boldsymbol{\eta}_t,\\
\d Y^{\boldsymbol{\eta}}_t &= (\sigma_2(t,Y^{\boldsymbol{\eta}}_t), 0, h_3(t,Y^{\boldsymbol{\eta}}_t),0) \dd \boldsymbol{\eta}_t,
\end{split}
\end{equation}
defined on $(\bar{\Omega},\bar{\cF},\bar{\P})$, and we let $I^{\boldsymbol{\eta}} = I^{1,\boldsymbol{\eta}} + I^{2,\boldsymbol{\eta}}$, where
\begin{equation}\label{eq: definitio I^eta}
\begin{split}
I^{1,\boldsymbol{\eta}}_t &:= \int_0^t H(s,X^{\boldsymbol{\eta}}_s,Y^{\boldsymbol{\eta}}_s) \dd \boldsymbol{\eta}_s - \frac{1}{2} \int_0^t |h(s,X^{\boldsymbol{\eta}}_s,Y^{\boldsymbol{\eta}}_s)|^2 \dd s,\\
I^{2,\boldsymbol{\eta}}_t &:= \int_0^t \int_{\U_2} \big(1 - \lambda(s,X^{\boldsymbol{\eta}}_{s-},u) + \log(\lambda(s,X^{\boldsymbol{\eta}}_{s-},u))\big) \, \nu_2(\d u) \dd s
\end{split}
\end{equation}
for $t \in [0,T]$, noting that these equations are well-defined by Assumptions~\ref{assumption: f= hg} and \ref{assumption: on lambda for the measure change}, combined with the fact that, since $(v - 1)/v \leq \log(v) \leq v - 1$ for all $v > 0$, we have
\begin{equation}\label{eq: bound on 1-lambda+ log(lambda)}
\big| 1 - \lambda(s,x,u) + \log(\lambda(s,x,u)) \big| = \lambda(s,x,u) - 1 - \log(\lambda(s,x,u)) \leq \frac{(1 - \lambda(s,x,u))^2}{\lambda(s,x,u)}.
\end{equation}

\subsection{Robustness}

We now proceed to consider continuity of the conditional distribution, viewed as a function of the rough path driving the system \eqref{eq: filtering RSDE}. To this end, given a bounded measurable function $F \colon \R^{d_X + d_Y} \to \R$, we define functions $g^F$ and $\Theta^F$ on the space of rough paths $\sV^p$, such that
\begin{equation}\label{eq: defn g^F Theta^F}
g^F_t(\boldsymbol{\eta}) := \bar{\E} \big[ F(X^{\boldsymbol{\eta}}_t,Y^{\boldsymbol{\eta}}_t) \exp (I^{\boldsymbol{\eta}}_t) \big] 
\qquad \text{and} \qquad \Theta^F_t(\boldsymbol{\eta}) := \frac{g^F_t(\boldsymbol{\eta})}{g^1_t(\boldsymbol{\eta})}
\end{equation}
for each $\boldsymbol{\eta} \in \sV^p$ and $t \in [0,T]$.

We are now ready to present the first main result of this section.

\begin{theorem}\label{theorem: robustness result}
Suppose that $f_1$ satisfies Assumption~\ref{assumption: Regularity of measure integral for RSDE} for $q = 2$ and some $p \in [2,3)$, and that Assumptions~\ref{assumption: on lambda for the measure change} and \ref{assumption: f= hg} also hold. For any $F \in C^1_b$ and $\boldsymbol{\eta} \in \sV^p$, we have that $g^F(\boldsymbol{\eta}), \Theta^F(\boldsymbol{\eta}) \in D([0,T];\R)$ (the space of real-valued c\`adl\`ag paths). Moreover, the following hold.
\begin{itemize}
\item[(i)] If $F \in C^1_b$, and if $\sV^p$ is endowed with (rough path) $p$-variation topology, and $D([0,T];\R)$ with the uniform topology, then $g^F$ and $\Theta^F$ are both locally Lipschitz continuous.
\item[(ii)] If $F \in C^2_b$, then $g^F(\boldsymbol{\eta}), \Theta^F(\boldsymbol{\eta}) \in V^p([0,T];\R)$, and if $\sV^p$ and $V^p([0,T];\R)$ are both endowed with $p$-variation topology, then $g^F$ and $\Theta^F$ are both locally Lipschitz continuous.
\item[(iii)] For any $F \in C^1_b$ and any fixed $t \in (0,T]$, $g^F_t$ and $\Theta^F_t$ are continuous when $\sV^p$ is endowed with the $p$-variation J1-Skorokhod distance $\sigma_{p,[0,t]}$, as defined in \eqref{eq: defn skorokhod metric}.
\end{itemize}
\end{theorem}

\begin{proof}
We first note that $I^{\boldsymbol{\eta}} \in \textup{BMO}^{p\textup{-var}}$ by Proposition~\ref{proposition: exponential integrability of a rough stochastic integral}, and hence that $\exp(I^{\boldsymbol{\eta}})$ has finite moments of all orders and is also c\`adl\`ag in $L^1$ by Proposition~\ref{proposition: exp(V) is Holder continuous}. Since $F$ is Lipschitz, it is also clear that $F(X^{\boldsymbol{\eta}}, Y^{\boldsymbol{\eta}})$ is c\`adl\`ag in $L^2$. Since, for any $(s,t) \in \Delta_{[0,T]}$,
\begin{align*}
|g^F_t(\boldsymbol{\eta}) - g^F_s(\boldsymbol{\eta})| &\leq \bar{\E} \big[ \big|F(X^{\boldsymbol{\eta}}_t, Y^{\boldsymbol{\eta}}_t) \exp(I^{\boldsymbol{\eta}}_t) - F(X^{\boldsymbol{\eta}}_s, Y^{\boldsymbol{\eta}}_s) \exp(I^{\boldsymbol{\eta}}_s)\big| \big]\\
&\leq \bar{\E} \big[ \big| \delta F(X^{\boldsymbol{\eta}}, Y^{\boldsymbol{\eta}})_{s,t} \exp(I^{\boldsymbol{\eta}}_t) \big| \big] + \bar{\E} \big[ \big| F(X^{\boldsymbol{\eta}}_s, Y^{\boldsymbol{\eta}}_s) \delta \exp(I^{\boldsymbol{\eta}})_{s,t} \big| \big]\\
&\lesssim \|\delta F(X^{\boldsymbol{\eta}}, Y^{\boldsymbol{\eta}})_{s,t}\|_{L^2(\bar{\Omega})} + \|\delta \exp(I^{\boldsymbol{\eta}})_{s,t}\|_{L^1(\bar{\Omega})},
\end{align*}
it follows that $g^F(\boldsymbol{\eta})$, and hence also $\Theta^F(\boldsymbol{\eta})$, are right-continuous. Similarly, we deduce that, whenever $s_n \nearrow s$, the sequence $(g^F_{s_n}(\boldsymbol{\eta}))_{n \in \N}$ is Cauchy, so that we also have the existence of left-limits.

\emph{(i):}
Let $\boldsymbol{\eta}, \tilde{\boldsymbol{\eta}} \in \sV^p$ with $\|\boldsymbol{\eta}\|_{p,[0,T]}, \| \tilde{\boldsymbol{\eta}}\|_{p,[0,T]} \leq L$ for some constant $L > 0$. Then, by H\"older's inequality and the inequality $|\exp(x) - \exp(y)| \leq |x-y| (\exp(x) \vee \exp(y))$, we have, for any $t \in [0,T]$,
\begin{equation}\label{eq: first Lipschitz bound on g^F}
\begin{split}
|g^F_t(\boldsymbol{\eta}) - g^F_t(\tilde{\boldsymbol{\eta}})| &\leq \bar{\E} \big[ \big| F(X^{\boldsymbol{\eta}}_t, Y^{\boldsymbol{\eta}}_t) \exp(I^{\boldsymbol{\eta}}_t) - F(X^{\tilde{\boldsymbol{\eta}}}_t, Y^{\tilde{\boldsymbol{\eta}}}_t) \exp(I^{\tilde{\boldsymbol{\eta}}}_t) \big| \big]\\
&\leq \big\| F(X^{\boldsymbol{\eta}}_t, Y^{\boldsymbol{\eta}}_t) - F(X^{\tilde{\boldsymbol{\eta}}}_t, Y^{\tilde{\boldsymbol{\eta}}}_t) \big\|_{L^2(\bar{\Omega})} \| \exp(I^{\tilde{\boldsymbol{\eta}}}_t) \|_{L^2(\bar{\Omega})}\\
&\quad + \|F\|_\infty \big\| \exp(I^{\boldsymbol{\eta}}_t) \vee \exp(I^{\tilde{\boldsymbol{\eta}}}_t) \big\|_{L^2(\bar{\Omega})} \| I^{\boldsymbol{\eta}}_t - I^{\tilde{\boldsymbol{\eta}}}_t \|_{L^2(\bar{\Omega})}.
\end{split}
\end{equation}
Using the Lipschitz continuity of $F$, and the stability of solutions to rough SDEs \eqref{eq: Lipschitz continuity of solution map}, we have that
\begin{equation*}
\big\| F(X^{\boldsymbol{\eta}}_t, Y^{\boldsymbol{\eta}}_t) - F(X^{\tilde{\boldsymbol{\eta}}}_t, Y^{\tilde{\boldsymbol{\eta}}}_t) \big\|_{L^2(\bar{\Omega})} \lesssim \| X^{\boldsymbol{\eta}}_t - X^{\tilde{\boldsymbol{\eta}}}_t \|_{L^2(\bar{\Omega})} + \| Y^{\boldsymbol{\eta}}_t - Y^{\tilde{\boldsymbol{\eta}}}_t \|_{L^2(\bar{\Omega})} \lesssim \|\boldsymbol{\eta} - \tilde{\boldsymbol{\eta}}\|_{p,[0,T]}.
\end{equation*}
The difference $\|I^{1,\boldsymbol{\eta}}_t - I^{1,\tilde{\boldsymbol{\eta}}}_t\|_{L^2(\bar{\Omega})}$ may be similarly bounded using the stability of rough stochastic integration (\cite[Lemma~4.9]{AllanPieper2026}) and of solutions to rough SDEs again. Finally, by Assumption~\ref{assumption: f= hg}, we have that
\begin{equation}\label{eq: Lipschitz continuity of lambda}
\begin{split}
|&\lambda(s,X^{\boldsymbol{\eta}}_{s-},u) - \lambda(s,X^{\tilde{\boldsymbol{\eta}}}_{s-},u)|\\
&\leq |\lambda(s,X^{\boldsymbol{\eta}}_{s-},u) \vee \lambda(s,X^{\tilde{\boldsymbol{\eta}}}_{s-},u)| |\log(\lambda(s,X^{\boldsymbol{\eta}}_{s-},u)) - \log(\lambda(s,X^{\tilde{\boldsymbol{\eta}}}_{s-},u))|\\
&= |\lambda(s,X^{\boldsymbol{\eta}}_{s-},u) \vee \lambda(s,X^{\tilde{\boldsymbol{\eta}}}_{s-},u)| |\kappa(s,X^{\boldsymbol{\eta}}_{s-}) - \kappa(s,X^{\tilde{\boldsymbol{\eta}}}_{s-})| |\gamma(s,u)|\\
&\lesssim \|\kappa\|_{C^1_b} |X^{\boldsymbol{\eta}}_{s-} - X^{\tilde{\boldsymbol{\eta}}}_{s-}| |\gamma(s,u)|,
\end{split}
\end{equation}
so that
\begin{equation}\label{eq: I^2 eta bound}
\|I^{2,\boldsymbol{\eta}}_t - I^{2,\tilde{\boldsymbol{\eta}}}_t\|_{L^2(\bar{\Omega})} \lesssim \bigg(\int_0^T \int_{\U_2} |\gamma(s,u)| \, \nu_2(\d u) \dd s\bigg) \|X^{\boldsymbol{\eta}} - X^{\tilde{\boldsymbol{\eta}}}\|_{p,2,[0,T],\bar{\Omega}} \lesssim \|X^{\boldsymbol{\eta}} - X^{\tilde{\boldsymbol{\eta}}}\|_{p,2,[0,T],\bar{\Omega}},
\end{equation}
and it follows that $\sup_{t \in [0,T]} |g^F_t(\boldsymbol{\eta}) - g^F_t(\tilde{\boldsymbol{\eta}})| \lesssim \|\boldsymbol{\eta} - \tilde{\boldsymbol{\eta}}\|_{p,[0,T]}$. By Jensen's inequality, we have that
\begin{equation}\label{eq: bound on g^1(eta)}
g^1_t(\boldsymbol{\eta}) = \bar{\E} [\exp(I^{\boldsymbol{\eta}}_t)] \geq \exp(-\bar{\E} [|I^{\boldsymbol{\eta}}_t|]),
\end{equation}
from which we infer by Proposition~\ref{proposition: exponential integrability of a rough stochastic integral} that $g^1_t(\boldsymbol{\eta})$ is bounded from both above and below (with bounds which depend on $L$). It follows that $g^F$ and $\Theta^F$ are both locally Lipschitz.

\emph{(ii):}
Let $\boldsymbol{\eta}, \tilde{\boldsymbol{\eta}} \in \sV^p$ such that $\|\boldsymbol{\eta}\|_{p,[0,T]}, \|\tilde{\boldsymbol{\eta}}\|_{p,[0,T]} \leq L$ for some constant $L > 0$. For $(s,t) \in \Delta_{[0,T]}$, we have that
\begin{align*}
|\delta g^F_{s,t}(\boldsymbol{\eta}) - \delta g^F_{s,t}(\tilde{\boldsymbol{\eta}})| &\leq \bar{\E} \big[ \big| \delta F(X^{ \boldsymbol{\eta}},Y^{ \boldsymbol{\eta}})_{s,t} \exp(I^{\boldsymbol{\eta}}_t) - \delta F(X^{\tilde{\boldsymbol{\eta}}},Y^{\tilde{\boldsymbol{\eta}}})_{s,t} \exp(I^{\tilde{\boldsymbol{\eta}}}_t) \big| \big]\\
&\quad + \bar{\E} \big[ \big| F(X^{\boldsymbol{\eta}},Y^{\boldsymbol{\eta}})_s \delta \exp(I^{\boldsymbol{\eta}})_{s,t} - F(X^{\tilde{\boldsymbol{\eta}}},Y^{\tilde{\boldsymbol{\eta}}})_s \delta \exp(I^{\tilde{\boldsymbol{\eta}}})_{s,t} \big| \big]\\
&\leq \bar{\E} \big[ \big| \delta F(X^{\boldsymbol{\eta}},Y^{\boldsymbol{\eta}})_{s,t} - \delta F(X^{\tilde{\boldsymbol{\eta}}},Y^{\tilde{\boldsymbol{\eta}}})_{s,t} \big| \big| \exp(I^{\boldsymbol{\eta}}_t) \big| \big]\\
&\quad + \bar{\E} \big[ \big| \delta F(X^{\tilde{\boldsymbol{\eta}}},Y^{\tilde{\boldsymbol{\eta}}})_{s,t} \big| \big| \delta \exp(I^{\boldsymbol{\eta}})_{s,t} - \delta \exp(I^{\tilde{\boldsymbol{\eta}}})_{s,t} \big| \big]\\
&\quad + \bar{\E} \big[ \big| \delta F(X^{\tilde{\boldsymbol{\eta}}},Y^{\tilde{\boldsymbol{\eta}}})_{s,t} \big| \big| \exp(I^{\boldsymbol{\eta}}_s) - \exp(I^{\tilde{\boldsymbol{\eta}}}_s) \big| \big]\\
&\quad + \bar{\E} \big[ \big| F(X^{\boldsymbol{\eta}},Y^{\boldsymbol{\eta}})_s - F(X^{\tilde{\boldsymbol{\eta}}},Y^{\tilde{\boldsymbol{\eta}}})_s \big| \big| \delta \exp(I^{\boldsymbol{\eta}})_{s,t} \big| \big]\\
&\quad + \bar{\E} \big[ \big| F(X^{\tilde{\boldsymbol{\eta}}},Y^{\tilde{\boldsymbol{\eta}}})_s \big| \big| \delta \exp(I^{\boldsymbol{\eta}})_{s,t} - \delta \exp(I^{\tilde{\boldsymbol{\eta}}})_{s,t} \big| \big].
\end{align*}
By H\"older's inequality, Proposition~\ref{proposition: The Ito-Lyons map is locally affine}, and an application of Proposition~\ref{proposition: exp(V) is Holder continuous} to $I^{\boldsymbol{\eta}}, I^{\tilde{\boldsymbol{\eta}}}$ (with $\tilde{q} = r = 1$ and $q = 2$), we then have that
\begin{align*}
|\delta g^F_{s,t}(\boldsymbol{\eta}) - \delta g^F_{s,t}(\tilde{\boldsymbol{\eta}})| &\lesssim \|X^{\boldsymbol{\eta}} - X^{\tilde{\boldsymbol{\eta}}}\|_{p,2,[s,t],\bar{\Omega}} + \|Y^{\boldsymbol{\eta}} -  Y^{\tilde{\boldsymbol{\eta}}}\|_{p,2,[s,t],\bar{\Omega}}\\
&\quad + \|F\|_\infty \|I^{\boldsymbol{\eta}} - I^{\tilde{\boldsymbol{\eta}}}\|_{p,2,[s,t],\bar{\Omega}}\\
&\quad +\big(\|X^{\boldsymbol{\eta}}\|_{p,2,\infty,[s,t],\bar{\Omega}} + \|Y^{\boldsymbol{\eta}}\|_{p,2,\infty,[s,t],\bar{\Omega}}\big) \|I^{\boldsymbol{\eta}} - I^{\tilde{\boldsymbol{\eta}}}\|_{p,2,[0,T],\bar{\Omega}}\\
&\quad + \big( \|X^{\boldsymbol{\eta}} - X^{\tilde{\boldsymbol{\eta}}}\|_{p,2,[0,T],\bar{\Omega}} + \|Y^{\boldsymbol{\eta}} - Y^{\tilde{\boldsymbol{\eta}}}\|_{p,2,[0,T],\bar{\Omega}} \big) \|\exp(I^{\boldsymbol{\eta}})\|_{p,2,[s,t],\bar{\Omega}}\\
&\quad + \|F\|_\infty \|I^{\boldsymbol{\eta}} - I^{\tilde{\boldsymbol{\eta}}}\|_{p,2,[s,t],\bar{\Omega}},
\end{align*}
where we also used the first bound in \cite[Lemma~4.13]{AllanPieper2026} (which requires $F \in C^2_b$) to bound $\|F(X^{\boldsymbol{\eta}}, Y^{\boldsymbol{\eta}}) - F(X^{\tilde{\boldsymbol{\eta}}}, Y^{\tilde{\boldsymbol{\eta}}})\|_{p,2,[s,t],\bar{\Omega}} \lesssim \|X^{\boldsymbol{\eta}} - X^{\tilde{\boldsymbol{\eta}}}\|_{p,2,[s,t],\bar{\Omega}} + \|Y^{\boldsymbol{\eta}} - Y^{\tilde{\boldsymbol{\eta}}}\|_{p,2,[s,t],\bar{\Omega}}$.

We can bound $\|I^{1,\boldsymbol{\eta}} - I^{1,\tilde{\boldsymbol{\eta}}}\|_{p,2,[0,T],\bar{\Omega}} \lesssim \|\boldsymbol{\eta} - \tilde{\boldsymbol{\eta}}\|_{p,[0,T]}$ using the stability of rough integration (\cite[Lemma~4.9]{AllanPieper2026}) and of solutions to rough SDEs \eqref{eq: Lipschitz continuity of solution map}, and we can use \eqref{eq: Lipschitz continuity of lambda} again to see that
\begin{equation*}
\|I^{2,\boldsymbol{\eta}} - I^{2,\tilde{\boldsymbol{\eta}}}\|_{p,2,[0,T],\bar{\Omega}} \lesssim \bigg( \int_0^T \int_{\U_2} |\gamma(r,u)| \, \nu_2(\d u) \dd r \bigg) \|X^{\boldsymbol{\eta}} - X^{\tilde{\boldsymbol{\eta}}}\|_{p,2,[0,T],\bar{\Omega}} \lesssim \|\boldsymbol{\eta} - \tilde{\boldsymbol{\eta}}\|_{p,[0,T]}.
\end{equation*}
It follows that $\|g^F(\boldsymbol{\eta}) - g^F(\tilde{\boldsymbol{\eta}})\|_{p,[0,T]} \lesssim \|\boldsymbol{\eta} - \tilde{\boldsymbol{\eta}}\|_{p,[0,T]}$.

A straightforward calculation, using the fact that $g^1_t(\boldsymbol{\eta})$ is bounded below by \eqref{eq: bound on g^1(eta)}, shows that
\begin{align*}
|\delta \Theta^F_{s,t}(\boldsymbol{\eta}) - \delta \Theta^F_{s,t}(\tilde{\boldsymbol{\eta}})| &\lesssim |\delta g^F_{s,t}(\boldsymbol{\eta}) - \delta g^F_{s,t}(\tilde{\boldsymbol{\eta}})| + |g^1_t(\boldsymbol{\eta}) - g^1_t(\tilde{\boldsymbol{\eta}})| |\delta g^F_{s,t}(\tilde{\boldsymbol{\eta}})|\\
&\quad + |g^1_t(\tilde{\boldsymbol{\eta}}) g^1_s(\tilde{\boldsymbol{\eta}}) g^F_s(\boldsymbol{\eta}) \delta g^1_{s,t}(\boldsymbol{\eta}) - g^1_t(\boldsymbol{\eta}) g^1_s(\boldsymbol{\eta}) g^F_s(\tilde{\boldsymbol{\eta}}) \delta g^1_{s,t}(\tilde{\boldsymbol{\eta}})|.
\end{align*}
We have already shown how to bound the first term on the right-hand side. For the second term, we may use the result of part~(i), and the fact that $\|g^F(\tilde{\boldsymbol{\eta}})\|_{p,[0,T]}$ is uniformly bounded (by a constant which depends in particular on $F$ and $L$), and we can similarly bound the final term by further splitting it up by adding zeros in the obvious way. Putting this all together, we deduce that $\|\Theta^F(\boldsymbol{\eta}) - \Theta^F(\tilde{\boldsymbol{\eta}})\|_{p,[0,T]} \lesssim \|\boldsymbol{\eta} - \tilde{\boldsymbol{\eta}}\|_{p,[0,T]}$.

\emph{(iii):}
Let $\boldsymbol{\eta} \in \sV^p$ and $t \in (0,T]$, and let $(\boldsymbol{\eta}^n)_{n \in \N} \subset \sV^p$ be a sequence of rough paths such that $\sigma_{p,[0,t]}(\boldsymbol{\eta}^n,\boldsymbol{\eta}) \to 0$ as $n \to \infty$. In the proof of part~(i) above, we showed in particular that
\begin{equation*}
|g^F_t(\boldsymbol{\eta}^n) - g^F_t(\boldsymbol{\eta})| \lesssim \| X^{\boldsymbol{\eta}^n}_t - X^{\boldsymbol{\eta}}_t \|_{L^2(\bar{\Omega})} + \| Y^{\boldsymbol{\eta}^n}_t - Y^{\boldsymbol{\eta}}_t \|_{L^2(\bar{\Omega})} + \| I^{\boldsymbol{\eta}^n}_t - I^{\boldsymbol{\eta}}_t \|_{L^2(\bar{\Omega})}.
\end{equation*}
It follows from Proposition~\ref{prop: Skorokhod continuity for RSDEs} and Remark~\ref{remark: Skorokhod continuity for rough stochastic integrals} that the right-hand side above tends to zero as $n \to \infty$, and we deduce that $g^F_t$, and hence also $\Theta^F_t$, are continuous.
\end{proof}

As highlighted by, e.g., Crisan et al.~\cite{CrisanDiehlFrizOberhauser2013}, a benefit of obtaining a robust representation of the conditional distribution is that in real-world applications the model chosen for the observation process may be an imperfect one, but a robust representation ensures that, provided the chosen model is close in some weak sense to the real one, our estimate of the conditional distribution should still be close to the true distribution. The following result can be interpreted as a novel and concrete formulation of this idea.

\begin{theorem}\label{theorem: Lipschitz continuity of Theta in L^m}
Let $m \in [1,\infty)$, $\epsilon > 0$, $\alpha > 0$, $L > 0$ and $F \in C^1_b$, and let us adopt the assumptions of Theorem~\ref{theorem: robustness result}. There exist constants $\beta > 0$ and $C > 0$, which depend on $p, T, m, \epsilon, \alpha$ and $\|F\|_{C^1_b}$, as well as the constants in Assumption~\ref{assumption: Regularity of measure integral for RSDE} for $f_1$, and on the other coefficients and constants specified in Assumption~\ref{assumption: f= hg} (and $C$ also depends on $L$), such that, whenever $\bV$ and $\tbV$ are two random c\`adl\`ag rough paths on $(\Omega,\cF,\tP)$ which satisfy
\begin{equation}\label{eq: exp needs to be L1}
\Big\| \exp \Big( \beta \Big( N_{\alpha,[0,T]}\big(\|\bV\|_{p,[\cdot,\cdot)}^p\big) + 1 \Big) \Big( 1 + \sup_{t \in (0,T]} |\Delta \bV_t| \Big) \Big) \Big\|_{L^1(\Omega)} \leq L,
\end{equation}
and the same inequality with $\bV$ replaced by $\tbV$, we then have that
\begin{equation}\label{eq: local Lipschitz continuity in L^m for Theta}
\Big\| \sup_{t \in [0,T]} \big|\Theta^F_t(\bV) - \Theta^F_t(\tbV)\big| \Big\|_{L^m(\Omega)} \leq C \big\| \|\bV - \tbV\|_{p,[0,T]} \big\|_{L^{m+\epsilon}(\Omega)}.
\end{equation}
\end{theorem}

\begin{proof}
Let $n > 1$ such that $\frac{1}{n} + \frac{1}{m+\epsilon} = \frac{1}{m+\frac{\epsilon}{2}}$. Recalling the bound in \eqref{eq: first Lipschitz bound on g^F}, we have that
\begin{equation*}
\begin{split}
&\Big\| \sup_{t \in [0,T]} \big|g^F_t(\bV) - g^F_t(\tbV)\big| \Big\|_{L^{m+\frac{\epsilon}{2}}(\Omega)}\\
&\lesssim \Big\| \sup_{t \in [0,T]} \big\| F(X^{\bV}_t,Y^{\bV}_t) - F(X^{\tbV}_t,Y^{\tbV}_t) \big\|_{L^2(\bar{\Omega})} \Big\|_{L^{m+\epsilon}(\Omega)} \Big\| \sup_{t \in [0,T]} \|\exp(I^{\tbV}_t)\|_{L^2(\bar{\Omega})} \Big\|_{L^n(\Omega)}\\
&\quad + \Big\| \sup_{t \in [0,T]} \|\exp(I^{\bV}_t) \vee \exp(I^{\tbV}_t)\|_{L^2(\bar{\Omega})} \Big\|_{L^n(\Omega)} \Big\| \sup_{t \in [0,T]} \|I^{\bV}_t - I^{\tbV}_t\|_{L^2(\bar{\Omega})} \Big\|_{L^{m+\epsilon}(\Omega)}.
\end{split}
\end{equation*}
Let $I^{\bV} = I^{1,\bV} + I^{2,\bV}$, as in \eqref{eq: definitio I^eta}. By the estimate in \eqref{eq: exponential bound for rough stochastic integral of RSDE solutions} and the assumption in \eqref{eq: exp needs to be L1}, for any $q \in [1,\infty)$, there exists a constant $\zeta > 0$ such that
\begin{equation*}
\Big\| \exp \Big( q \sup_{t \in [0,T]} |I^{1,\bV}_t| \Big) \Big\|_{L^1(\bar{\Omega})} \leq \exp \Big( \zeta \Big( N_{\alpha,[0,T]}\big(\|\bV\|_{p,[\cdot,\cdot)}^p\big) + 1 \Big) \Big( 1 + \sup_{t \in (0,T]} |\Delta \bV_t| \Big) \Big),
\end{equation*}
and so by choosing $\beta = \zeta$ in \eqref{eq: exp needs to be L1}, we have that $\| \| \exp ( q \sup_{t \in [0,T]} |I^{1,\bV}_t| ) \|_{L^1(\bar{\Omega})} \|_{L^1(\Omega)} \leq L$. Recalling \eqref{eq: bound on 1-lambda+ log(lambda)} and Assumption~\ref{assumption: on lambda for the measure change}, we also have that
\begin{equation*}
\Big\| \exp \Big( q \sup_{t \in [0,T]} |I^{2,\bV}_t| \Big) \Big\|_{L^\infty(\bar{\Omega})} \leq \exp \bigg( q T \sup_{t \in [0,T], \, x \in \R^{d_X}} \int_{\U_2} \frac{(1 - \lambda(t,x,u))^2}{\lambda(t,x,u)} \, \nu_2(\d u) \bigg) < \infty,
\end{equation*}
and the above also holds for $I^{\tbV} = I^{1,\tbV} + I^{2,\tbV}$. We then have that
\begin{align*}
\Big\| \sup_{t \in [0,T]} \big|g^F_t(\bV) - g^F_t(\tbV)\big| \Big\|_{L^{m+\frac{\epsilon}{2}}(\Omega)} &\lesssim \Big\| \sup_{t \in [0,T]} \big\|F(X^{\bV}_t,Y^{\bV}_t) - F(X^{\tbV}_t,Y^{\tbV}_t)\big\|_{L^2(\bar{\Omega})} \Big\|_{L^{m+\epsilon}(\Omega)}\\
&\quad + \Big\| \sup_{t \in [0,T]} \|I^{\bV}_t - I^{\tbV}_t\|_{L^2(\bar{\Omega})} \Big\|_{L^{m+\epsilon}(\Omega)}.
\end{align*}
The first term above can be bounded by the right-hand side of \eqref{eq: local Lipschitz continuity in L^m for Theta} by the Lipschitz continuity of $F$, and an application of Proposition~\ref{proposition: Lipschitz continuity of doubly stochastic DEs}, which is applicable since the combination of Lemma~\ref{lemma: |bX|_p,[0,T] leq N_alpha + jumps} with the assumption in \eqref{eq: exp needs to be L1} ensures that $\|\bV\|_{p,[0,T]}$ and $\|\tbV\|_{p,[0,T]}$ have finite moments of all orders. For the second term, we again split $I^{\bV} - I^{\tbV}$ into the sum of $I^{1,\bV} - I^{1,\tbV}$ and $I^{2,\bV} - I^{2,\tbV}$, where the former may be treated using \cite[Lemma~4.9]{AllanPieper2026} and Proposition~\ref{proposition: Lipschitz continuity of doubly stochastic DEs}, while the latter may be treated by recalling \eqref{eq: I^2 eta bound} and again using Proposition \ref{proposition: Lipschitz continuity of doubly stochastic DEs}. We thus obtain
\begin{equation*}
\Big\| \sup_{t \in [0,T]} \big|g^F_t(\bV) - g^F_t(\tbV)\big| \Big\|_{L^{m+\frac{\epsilon}{2}}(\Omega)} \lesssim \big\| \|\bV - \tbV\|_{p,[0,T]} \big\|_{L^{m+\epsilon}(\Omega)}.
\end{equation*}

Let $\tilde{n} > 1$ such that $\frac{1}{\tilde{n}} + \frac{1}{m+\frac{\epsilon}{2}} = \frac{1}{m}$. By H\"older's inequality, it is straightforward to see that
\begin{equation}\label{eq: Theta V - Theta tV bound}
\begin{split}
\bigg\| &\sup_{t \in [0,T]} \big|\Theta^F_t(\bV) - \Theta^F_t(\tbV)\big| \bigg\|_{L^m(\Omega)}\\
&\leq \bigg\| \sup_{t \in [0,T]} \big|g^F_t(\bV) - g^F_t(\tbV)\big| \bigg\|_{L^{m+\frac{\epsilon}{2}}(\Omega)} \bigg\| \sup_{t \in [0,T]} \frac{1}{g^1_t(\bV)} \bigg\|_{L^{\tilde{n}}(\Omega)}\\
&\quad + \bigg\| \sup_{t \in [0,T]} \big|g^1_t(\bV) - g^1_t(\tbV)\big| \bigg\|_{L^{m+\frac{\epsilon}{2}}(\Omega)} \bigg\| \sup_{t \in [0,T]} \frac{|g^F_t(\tbV)|}{g^1_t(\bV) g^1_t(\tbV)} \bigg\|_{L^{\tilde{n}}(\Omega)}.
\end{split}
\end{equation}
By Jensen's inequality, we have that
\begin{equation*}
\sup_{t \in [0,T]} \frac{1}{g^1_t(\bV)} \leq \frac{1}{\bar{\E}[\exp(-\sup_{t \in [0,T]} |I^{\bV}_t|)]} \leq \frac{1}{\exp(\bar{\E}[-\sup_{t \in [0,T]} |I^{\bV}_t|])} \leq \bar{\E} \Big[ \exp \Big( \sup_{t \in [0,T]} |I^{\bV}_t| \Big) \Big],
\end{equation*}
and since $F$ is bounded we also have that $g^F_t(\bV) \leq \bar{\E}[\exp(\sup_{t \in [0,T]} |I^{\bV}_t|)]$, and of course these inequalities also hold for $\tbV$. By the bounds on $I^{1,\bV}$ and $I^{2,\bV}$ established above, it follows that the norms on the right-hand side of \eqref{eq: Theta V - Theta tV bound} are finite, and we thus obtain the estimate in \eqref{eq: local Lipschitz continuity in L^m for Theta}.
\end{proof}

\begin{example}
If $\bV$ is continuous and satisfies
\begin{equation*}
\big\| \exp \big( \gamma N_{\alpha,[0,T]}\big(\|\bV\|_{p,[\cdot,\cdot)}^p\big)^q \big) \big\|_{L^1(\Omega)} < \infty
\end{equation*}
for some $\alpha > 0$, $\gamma > 0$ and $q > 1$, then, for any $\beta > 0$, one can find an $L > 0$ for which \eqref{eq: exp needs to be L1} holds. To see this we simply note that
\begin{align*}
\exp \big( \beta N_{\alpha,[0,T]}\big(\|\bV\|_{p,[\cdot,\cdot)}^p\big) \big) &\lesssim \exp \big( \beta N_{\alpha,[0,T]}\big(\|\bV\|_{p,[\cdot,\cdot)}^p\big) \1_{\{N_{\alpha,[0,T]}(\|\bV\|_{p,[\cdot,\cdot)}^p)^{q-1} > \frac{\beta}{\gamma}\}} \big)\\
&\leq \exp \big( \gamma N_{\alpha,[0,T]}\big(\|\bV\|_{p,[\cdot,\cdot)}^p\big)^q \big),
\end{align*}
where the implicit multiplicative constant depends only on $\gamma, q$ and $\beta$.

By \cite[Corollary~21]{DiehlOberhauserRiedel2015}, we thus conclude that, if $V$ is a centered Gaussian process with covariance of finite mixed $(1,\rho)$-variation for some $\rho < 2$, then the enhanced Gaussian process $\bV$ (as defined in \cite[Section~15.3.3]{FrizVictoir2010}) satisfies \eqref{eq: exp needs to be L1}. In particular, this includes the case of fractional Brownian motion with Hurst parameter $H > \frac{1}{4}$; see \cite[Section~2.3]{FrizGessGulisashviliRiedel2016}. In \cite[Proposition~4.10]{CassOgrodnik2017}, the authors provide another example of a random rough path fulfilling \eqref{eq: exp needs to be L1}, namely certain Markovian rough paths emerging from specific Dirichlet forms.
\end{example}

\begin{remark}
Suppose that $(\bV^n)_{n \in \N}$ is a sequence of approximations of the random rough path $\bV$, such that convergence rates for $\bV^n \to \bV$ are known in a suitable Lebesgue space $L^{m+\epsilon}$. By Theorem~\ref{theorem: Lipschitz continuity of Theta in L^m}, those convergence rates then carry over to the $L^m$ approximation of $\Theta^F_t(\bV)$ by $\Theta^F_t(\bV^n)$. For example, if $\bV$ is (Stratonovich) Brownian rough path, and $\bV^n$ is the piecewise linear approximation thereof on the $n$\textsuperscript{th} dyadic partition, then it follows by \cite[Corollary~13.21]{FrizVictoir2010} that, for any $\eta \in (0,\frac{1}{2})$, we have the convergence rate
\begin{equation*}
\Big\| \sup_{t \in [0,T]} \big|\Theta^F_t(\bV^n) - \Theta^F_t(\bV)\big| \Big\|_{L^m(\Omega)} \lesssim \big\| \|\bV^n - \bV\|_{p,[0,T]} \big\|_{L^{m+\epsilon}(\Omega)} \lesssim 2^{-\frac{\eta n}{2}}.
\end{equation*}
\end{remark}

\subsection{Consistency}

We recall the local martingale $G$ defined in \eqref{eq: definition G}. Let us now denote by $\bG = (G,\G)$ the It\^o rough path lift of $G$, so that $\G_{s,t} = \int_s^t \delta G_{s,u} \otimes \d G_u$, defined as an It\^o integral on $(\Omega,\cF,\tP)$ for each $(s,t) \in \Delta_{[0,T]}$. Of course, $\bG$ is then a random c\`adl\`ag rough path, such that $\bG(\omega) \in \sV^p$ for $\tP$-almost every $\omega \in \Omega$ and any $p \in (2,3)$.

As before, we let $(X,Y)$ be the solution to the (doubly) SDE \eqref{eq: dynamics of X Y driven by G}, and let $I$ be as defined in \eqref{eq: defn I on product space}, both defined on the product probability space $(\hat{\Omega},\hat{\cF},\hat{\P})$. By Theorem~\ref{theorem: consistency between doubly stochastic SDEs and RSDEs} (and Remark~\ref{remark: consistency also works with random measures}), we have that $(X,Y)$ is also the solution to the corresponding random rough SDE. That is, for $\tP$-almost every $\omega \in \Omega$, we have that $(X(\omega,\cdot), Y(\omega,\cdot))$ is the solution to the rough SDE in \eqref{eq: filtering RSDE} driven by $\boldsymbol{\eta} = \bG(\omega)$, i.e.,
\begin{equation}\label{eq: consistency of X and Y}
(X(\omega,\cdot), Y(\omega,\cdot)) = (X^{\bG(\omega)}, Y^{\bG(\omega)}).
\end{equation}
By Proposition~\ref{proposition: consistency rough and stochastic integrals for semimartingales}, we then also have that $I$ is the randomised version of the process $I^{\boldsymbol{\eta}}$ in \eqref{eq: definitio I^eta}. That is, for $\tP$-almost every $\omega \in \Omega$, we have that $I(\omega,\cdot) = I^{\bG(\omega)}$.

\smallskip

We recall from Section~\ref{subsection: reformulation as a rough SDE} (with a slight abuse of notation) that $(X,Y,I)$ has the same law on $(\Omega,\cF)$ under $\tP$ as it did on $(\hat{\Omega},\hat{\cF})$ under $\hat{\P}$. Hence, for any $t \in [0,T]$ and $A \in \cF^Y_t$, using Fubini's theorem, and the fact that $\cF^Y_t$ and $\cF^B_t \vee \cF^{\tN_1}_t$ are independent by Assumption~\ref{assumption: FY= FW v FN}, we have that
\begin{align*}
\widetilde{\E} \big[F(X_t,Y_t) \exp(I_t) \1_A\big] = \hat{\E} \big[F(X_t,Y_t) \exp(I_t) \1_A\big] = \widetilde{\E} \big[\bar{\E} \big[F(X_t,Y_t) \exp(I_t)\big] \1_A\big].
\end{align*}
By another use of Fubini's theorem, we also have that $\bar{\E}[F(X_t,Y_t) \exp(I_t)]$ is $\cF^Y_t$-measurable, and thus coincides with $\widetilde{\E}[F(X_t,Y_t) \exp(I_t) \,|\, \cF^Y_t]$. By the Kallianpur--Striebel formula (Lemma~\ref{lemmaKallianpurStriebel}), we then have that the conditional distribution on the product space can be represented as
\begin{equation*}
\pi_t(F) = \frac{\rho_t(F)}{\rho_t(1)} = \frac{\bar{\E}[F(X_t,Y_t) \exp(I_t)]}{\bar{\E}[\exp(I_t)]}
\end{equation*}
for all bounded measurable functions $F$.

\begin{proposition}\label{prop: consistency for the filter}
Suppose that $f_1$ satisfies Assumption~\ref{assumption: Regularity of measure integral for RSDE} for $q = 2$ and some $p \in (2,3)$, and that Assumptions~\ref{assumption: on the coefficients for the SDE, linear growth}--\ref{assumption: on lambda for the measure change}, \ref{assumption: FY= FW v FN} and \ref{assumption: f= hg} also hold. Let us recall the functions $g^F$ and $\Theta^F$ defined in \eqref{eq: defn g^F Theta^F}. For any bounded measurable function $F$, we have that $g^F(\bG)$ and $\Theta^F(\bG)$ are $\cF^Y_t$-adapted, and that, for any $t \in [0,T]$,
\begin{equation}\label{eq: consistency result}
g^F_t(\bG) = \rho_t(F) \quad \text{and} \quad \Theta^F_t(\bG) = \pi_t(F)
\end{equation}
hold $\tP$-almost surely.
\end{proposition}

\begin{proof}
For $\tP$-almost any $\omega \in \Omega$, using \eqref{eq: consistency of X and Y}, we have that
\begin{align*}
g^F_t(\bG(\omega)) &= \bar{\E} \big[ F\big(X^{\bG(\omega)}_t,Y^{\bG(\omega)}_t\big) \exp \big(I^{\bG(\omega)}_t\big) \big]\\
&= \bar{\E} \big[ F(X_t(\omega,\cdot), Y_t(\omega,\cdot)) \exp (I_t(\omega,\cdot)) \big] = \rho_t(F)(\omega).
\end{align*}
It follows that the equalities in \eqref{eq: consistency result} hold. In particular, we infer that $g^F_t(\bG)$ and $\Theta^F_t(\bG)$ are $\cF^Y_t$-measurable random variables.
\end{proof}

\begin{remark}
For any $F \in C^1_b$, it follows from Theorem~\ref{theorem: robustness result} and Proposition~\ref{prop: consistency for the filter} that $\Theta^F(\bG)$ is a c\`adl\`ag $\cF^Y_t$-adapted (and hence $\cF^Y_t$-optional) process. Thus, although we only defined $\pi_t(F)$ for fixed times $t \in [0,T]$, we immediately deduce that $\pi(F)$ has an $\cF^Y_t$-optional version (without needing to take an optional projection, as in classical filtering theory).
\end{remark}

\subsection{Robust filtering with additive noise}

To express the conditional distribution of the signal as a function of the observable quantities, it was necessary to isolate these terms via a change of measure. As a result, the corresponding Radon--Nikodym derivative involves stochastic integrals against both the continuous and jump components of the observation noise, each with a distinct integrand.

In the continuous additive noise setting of \cite{CrisanDiehlFrizOberhauser2013}, the role of the process $G$ (as defined in \eqref{eq: definition G}) is played by the observation process $Y$, and the robust representation is written as a function of its rough path lift. However, as we have seen, contrary to the continuous setting, obtaining a robust representation in the presence of additional jump terms requires us to observe \emph{all} the components of $G$, and to express the filter as a function of the joint lift of these components. In particular, in general it is necessary to be able to distinguish the continuous and jump components of the observation process. This requirement corresponds to Assumption~\ref{assumption: FY= FW v FN}, which is commonly adopted in the study of stochastic filtering for jump-diffusion models (see, e.g., \cite{GermGyongy2025PartI}, \cite{Qiao2021} or \cite{QiaoDuan2015}).

While in general one needs to consider the separate jump components of the process $G$, in the special case when the observation is perturbed by additive noise, it is sufficient to observe only the continuous and jump parts of $Y$. This is made precise in the following corollary of Theorem~\ref{theorem: robustness result} and Proposition~\ref{prop: consistency for the filter}.

\begin{corollary}\label{corollary: additive noise robustness requires only continuous and jump part}
Suppose that $X$ and $Y$ satisfy
\begin{equation*}
\begin{split}
\d X_t &= b_1(t,X_t,Y_t) \dd t + \sigma_0(t,X_t,Y_t) \dd B_t + \sigma_1(t,X_t,Y_t) \dd W_t\\
&\quad + \int_{\U_1} f_1(t,X_{t-},Y_{t-},u) \, \tN_1(\d t,\d u) + \int_{\U_2} f_2(t,X_{t-},Y_{t-},u) \, \tN_2(\d t,\d u),\\
\d Y_t &= b_2(t,X_t,Y_t) \dd t + \d W_t + \int_{\U_2} g(t,u) \, \tN_2(\d t,\d u)
\end{split}
\end{equation*}
where, under the measure $\P$, $B$ and $W$ are independent Brownian motions, $N_1$ and $N_2$ are random measures with compensators $\nu_1(\d u) \dd t$ and $\lambda(t,X_{t-},u) \, \nu_2(\d u) \dd t$ respectively, and we suppose that $\log(\lambda(t,x,u)) = \kappa(t,x) g(t,u)$. Let us write $Y^c$ and $Y^d$ for the continuous and purely discontinuous local martingale parts of $Y$ under $\tP$, and let us adopt the assumptions of Theorem~\ref{theorem: robustness result} with $g = g_2 = g_3 = \gamma$, except with Assumption~\ref{assumption: FY= FW v FN} replaced by the assumption that $\cF^Y_t = \cF^Y_0 \vee \cF^{Y^c}_t \vee \cF^{Y^d}_t$ for every $t \in [0,T]$.

Then, for every bounded measurable function $F$ and every $t \in [0,T]$, there exists a function $\Theta^F_t$ on the space of c\`adl\`ag rough paths $\sV^p$, such that
\begin{equation*}
\pi_t(F) = \Theta^F_t\big(\bY^{c,d}\big)
\end{equation*}
holds $\P$-almost surely, where we write $\bY^{c,d}$ for the It\^o rough path lift of the pair $(Y^c, Y^d)^\top$. Moreover, the function $\Theta^F$ satisfies the same continuity properties as established for the corresponding function in Theorem~\ref{theorem: robustness result}.
\end{corollary}

\begin{remark}
We recall from Remark~\ref{remark: some cases when FY= FW v FN} that, when the function $g$ is suitably regular, the condition $\cF^Y_t = \cF^Y_0 \vee \cF^{Y^c}_t \vee \cF^{Y^d}_t$ is indeed satisfied. This includes, for instance, the case of additive L\'evy noise, corresponding to the choice $g(s,u) = u$ with $\U_2 = \R^{d_Y} \setminus \{0\}$.
\end{remark}

\begin{remark}
In light of Corollary~\ref{corollary: additive noise robustness requires only continuous and jump part}, it may be tempting to ask whether it would be sufficient to lift each of $Y^c$ and $Y^d$ to rough paths $\bY^c$ and $\bY^d$ individually, and seek to represent the filter as a function of the pair $(\bY^c, \bY^d)$. However, even in the additive noise case, the Radon--Nikodym derivative includes the exponential of a stochastic integral of the form
\begin{equation*}
\int_0^t h(s,X_s,Y^c_s + Y^d_s)^\top \dd Y^c_s,
\end{equation*}
which in general is not continuous with respect to $(\bY^c, \bY^d)$, but only with respect to $\bY^{c,d}$.
\end{remark}

\begin{remark}
Let us recall the setting of Corollary~\ref{corollary: additive noise robustness requires only continuous and jump part}, and let us further suppose that the function $b_2$ is independent of $y$, and that $\kappa$ is a solution to the integral equation
\begin{equation*}
\kappa(t,x)^\top = b_2(t,x) + \int_{\U_2} g(t,u) \big( 1 - \exp \big(\kappa(t,x) g(t,u)\big) \big) \, \nu_2(\d u).
\end{equation*}
In this case, we see from \eqref{eq: defn function h} that $h^\top = \kappa$, and hence that the stochastic integral in \eqref{eq: defn I on product space} is given by
\begin{equation*}
\int_0^t H(s,X_{s-},Y_{s-}) \dd G_s = \int_0^t \kappa(s,X_{s-}) \dd Y_s.
\end{equation*}
It follows that, in this (very) special case, it is not necessary to consider the continuous and jump parts of the observation separately, and there exists a robust representation of the conditional distribution which is a continuous function of $\bY$, the It\^o rough path lift of $Y$.
\end{remark}

\begin{remark}\label{rem:reconstruction}
Filtering assumes that the characteristics of the signal and noise are known to the observer, including in particular the compensator $\nu_2(\d u) \dd t$. It is therefore possible to reconstruct the purely discontinuous process $Y^{d,n} := \int_0^\cdot \int_{\U_2} g(s,u) \1_{\{|g(s,u)| \geq \frac{1}{n}\}} \, \tN(\d s,\d u)$ using discrete observations of $Y$ path by path (see, e.g., part A or D of \cite[Theorem~3.3.1]{JacodProtter2012}). Hence, writing $Y^{c,n} := Y - Y^{d,n}$, we can find consistent estimators of $Y^d$ and $Y^c$ which in turn can be arbitrarily well approximated using discrete observations of $Y$. Therefore, Corollary~\ref{corollary: additive noise robustness requires only continuous and jump part} justifies the use of discrete approximations of $Y^d$ and $Y^c$, as long as there is a coherent way of lifting all these pairs of approximations to common rough paths which converge in the rough path topology to the lift of their corresponding limit.
\end{remark}

\appendix

\section{Further technical results}

\subsection{Proof of Proposition~\ref{proposition: exp(V) is Holder continuous}}\label{sec: proof of JN corollary}

For a normed vector space $(E,|\cdot|)$, a function $f \colon E \to \R$, and any $a, b, c, d \in E$, we write
\begin{equation}\label{eq: |f|_1,a,b,c,d}
\begin{split}
|f|_{1,\{a,b\}} &:= \sup_{\theta \in [0,1]} \big| f ( \theta a + (1 - \theta) b ) \big|,\\
|f|_{2,\{a,b,c,d\}} &:= \sup_{(\theta, \eta) \in [0,1]^2} \big| f \big( \eta ( \theta a + (1 - \theta) b ) + (1 - \eta) ( \theta c + (1 - \theta) d ) \big) \big|.
\end{split}
\end{equation}

\begin{lemma}\label{lemma: elementary estimate}
For some Banach space $(E,|\cdot|)$, let $g \in C^2(E;\R)$ and $a, b, c, d \in E$. (We consider standard Fr\'echet differentiability here.) Then
\begin{equation}\label{elementary estimate}
\begin{split}
&|g(a) - g(b) - g(c) + g(d)|\\
&\leq \big(|\D g|_{1,\{a,c\}} + |\D g|_{1,\{b,d\}} + |\D^2 g|_{2,\{a,b,c,d\}}\big) \big(|a-c| + |b-d|\big) \big( |c-d| \wedge 1 \big)\\
&\quad + |\D g|_{1,\{a,b\}} |a - b - c + d|.
\end{split}
\end{equation}
\end{lemma}

\begin{proof}
We first observe the straightforward bound
\begin{align*}
|g(a) - g(b) - g(c) + g(d)| &\leq |g(a) - g(c)| + |g(b) - g(d)|\\
&\leq |\D g|_{1,\{a,c\}} |a - c| + |\D g|_{1,\{b,d\}} |b - d|.
\end{align*}
We also have that
\begin{align*}
&g(a) - g(b) - g(c) + g(d)\\
&= \int_0^1 \D g \big( b + \theta (a-b) \big) (a-b) \dd \theta - \int_0^1 \D g \big( b + \theta (a-b) \big) (c-d) \dd \theta\\
&\quad + \int_0^1 \D g \big( b + \theta (a-b) \big) (c-d) \dd \theta - \int_0^1 \D g \big( d + \theta (c-d) \big) (c-d) \dd \theta\\
&= \int_0^1 \D g \big( b + \theta (a-b) ) (a - b - c + d) \dd \theta\\
&\quad + \int_0^1 \Big( \D g \big(b + \theta (a-b) \big) - \D g \big(d + \theta (c-d) \big) \Big) (c-d) \dd \theta,
\end{align*}
from which we infer that
\begin{align*}
&|g(a) - g(b) - g(c) + g(d)|\\
&\leq |\D g|_{1,\{a,b\}} |a - b - c + d| + |\D^2 g|_{2,\{a,b,c,d\}} \big(|a - c| + |b - d|\big) |c - d|.
\end{align*}
Combining the bounds above, we obtain \eqref{elementary estimate}.
\end{proof}

\begin{proof}[Proof of Proposition~\ref{proposition: exp(V) is Holder continuous}]
We note that, for any $\lambda > 0$ and any $m \in [1,\infty)$, we have
\begin{equation}\label{eq: sup V_u has exponential moments}
\Big\| \exp \Big( \lambda \sup_{u \in [0,T]} |V_u| \Big) \Big\|_{L^m} \leq \exp (\lambda \|V_0\|_{L^\infty}) \Big\| \E_0 \Big[ \exp \Big( m \lambda \sup_{u \in [0,T]} |\delta V_{0,u}| \Big) \Big] \Big\|_{L^\infty}^{\frac{1}{m}} < \infty,
\end{equation}
where, by Theorem~\ref{theorem: John-Nirenberg inequality}, the right-hand side may be bounded by a constant which depends only on $\lambda, m, p, \|V_0\|_{L^\infty}$ and $\|V\|_{\textup{BMO}^{p\textup{-var}},[0,T]}$.

By the elementary estimate $|\exp(x) - \exp(y)| \leq |x-y| (\exp(x) \vee \exp(y))$, and the version of H\"older's inequality in \eqref{eq: Holder's inequality for q,r,s norm} (with $\ell = \frac{r}{q}$), for any $(s,t) \in \Delta_{[0,T]}$, we have, for any $q \in [1,\infty)$ and $r \in [q,\infty)$, that
\begin{align*}
&\|\delta \exp(V)_{s,t}\|_{q,r,s} \leq \Big\| \E_s \Big[ \exp \Big( q \sup_{u \in [s,T]} |V_u| \Big) |\delta V_{s,t}|^q \Big]^{\frac{1}{q}} \Big\|_{L^r}\\
&= \Big\| \E_s \Big[ \exp \Big( q \sup_{u \in [s,T]} |V_u| \Big) |\delta V_{s,t}|^q \Big] \Big\|_{L^{\frac{r}{q}}}^{\frac{1}{q}} \leq \Big\| \exp \Big( 2q \sup_{u \in [s,T]} |V_u| \Big) \Big\|_{L^{\frac{r}{q}}}^{\frac{1}{2q}} \|\delta V_{s,t}\|_{2q,2r,s}.
\end{align*}
By \eqref{eq: sup V_u has exponential moments}, this implies that $\|\exp(V)\|_{p,q,r,[0,T]} \lesssim \|V\|_{p,2q,2r,[0,T]}$, which is finite by Corollary~\ref{corollary: V in VpLqinfty for all q}.

We now let $1 \leq \tilde{q} \leq r < q < \infty$. Adopting the notation introduced in \eqref{eq: |f|_1,a,b,c,d}, we let $\Lambda = |\exp|_{2, \{V_t, \tV_t, V_s, \tV_s\}} + |\exp|_{1, \{V_t, \tV_t\}} + |\exp|_{1, \{V_s, \tV_s\}} + |\exp|_{1, \{V_t, V_s\}}$.

Applying Lemma~\ref{lemma: elementary estimate}, for any $(s,t) \in \Delta_{[0,T]}$, we have that
\begin{align*}
&\big\|\delta \exp(V)_{s,t} - \delta \exp(\tV)_{s,t}\big\|_{\tilde{q},r,s}\\
&\leq \big\| \E_s \big[ \Lambda^{\tilde{q}} \big(|V_t - \tV_t| + |V_s - \tV_s|\big)^{\tilde{q}} |\delta \tV_{s,t}|^{\tilde{q}} \big]^{\frac{1}{\tilde{q}}} \big\|_{L^r} + \big\| \E_s \big[ \Lambda^{\tilde{q}} |\delta V_{s,t} - \delta \tV_{s,t}|^{\tilde{q}} \big]^{\frac{1}{\tilde{q}}} \big\|_{L^r}\\
&\lesssim \big\| \E_s \big[ \Lambda^{\tilde{q}} \big(|\delta V_{s,t} - \delta \tV_{s,t}|\big)^{\tilde{q}} |\delta \tV_{s,t}|^{\tilde{q}} \big]^{\frac{1}{\tilde{q}}} \big\|_{L^r} + \big\| \E_s \big[ \Lambda^{\tilde{q}} |V_s - \tV_s|^{\tilde{q}} |\delta \tV_{s,t}|^{\tilde{q}} \big]^{\frac{1}{\tilde{q}}} \big\|_{L^r}\\
&\quad + \big\| \E_s \big[ \Lambda^{\tilde{q}} |\delta V_{s,t} - \delta \tV_{s,t}|^{\tilde{q}} \big]^{\frac{1}{\tilde{q}}} \big\|_{L^r} =: I_1 + I_2 + I_3.
\end{align*}

By H\"older's inequality and \eqref{eq: sup V_u has exponential moments}, it is straightforward to see that $\|\Lambda\|_{L^m} < \infty$ for any $m \in [1,\infty)$, and it is also clear that $\|\delta \tV_{s,t}\|_{L^m} < \infty$ for any $m \in [1,\infty)$. We then have that
\begin{equation*}
I_1 \leq \big\| \E_s \big[ \Lambda^{\frac{\tilde{q}q}{q-\tilde{q}}} |\delta \tV_{s,t}|^{\frac{\tilde{q}q}{q-\tilde{q}}} \big] \big\|_{L^{\frac{r}{\tilde{q}}}}^{\frac{q-\tilde{q}}{\tilde{q}q}} \|\delta V_{s,t} - \delta \tV_{s,t}\|_{q,\frac{qr}{\tilde{q}},s} \lesssim \|\delta V_{s,t} - \delta \tV_{s,t}\|_{q,\frac{qr}{\tilde{q}},s},
\end{equation*}
and $I_3$ may be treated similarly. Finally, using H\"older's inequality again, we have that
\begin{align*}
I_2 &\leq \|V_s - \tV_s\|_{L^q} \big\| \E_s \big[ \Lambda^{\tilde{q}} |\delta \tV_{s,t}|^{\tilde{q}} \big]^{\frac{1}{\tilde{q}}} \big\|_{L^{\frac{qr}{q-r}}} \leq \|V_s - \tV_s\|_{L^q} \|\Lambda\|_{2\tilde{q},\frac{2qr}{q-r},s} \|\delta \tV_{s,t}\|_{2\tilde{q},\frac{2qr}{q-r},s}\\
&\lesssim \big(\|V_0 - \tV_0\|_{L^q} + \|V - \tV\|_{p,q,\frac{qr}{\tilde{q}},[0,T]}\big) \|\delta \tV_{s,t}\|_{2\tilde{q},\frac{2qr}{q-r},s}.
\end{align*}
Combining the above, we obtain the estimate in \eqref{eq: exp(V) - exp(tV) estimate}.
\end{proof}

\subsection{Some results on integrals against random measures}

The following result is a generalized conditional BDG inequality.

\begin{lemma}\label{lemma: conditional extended BDG}
Let $q \in [2,\infty)$, and let $M = (M_t)_{t \in [0,T]}$ be an $L^q$-integrable c\`adl\`ag real-valued local martingale with respect to a filtration $(\cF_t)_{t \in [0,T]}$, and let $\cG$ be a sub-$\sigma$-algebra of $\cF_0$. There exist constants $c_q, C_q$, which depend only on $q$, such that
\begin{equation}\label{eq: extended conditional BDG}
c_q \E \Big[ \langle M \rangle_t^{\frac{q}{2}} \vee A^{(\frac{q}{2})}_t \,\Big|\, \cG\Big] \leq \E \Big[ \sup_{s \in [0,t]} |M_s|^q \,\Big|\, \cG\Big] \leq C_q \E \Big[ \langle M \rangle_t^{\frac{q}{2}} \vee A^{(\frac{q}{2})}_t \,\Big|\, \cG\Big]
\end{equation}
holds almost surely for every $t \in [0,T]$, where $\langle M \rangle$ denotes the predictable quadratic variation of $M$, and, in the notation of \cite{HernandezHernandezJackal2022}, we write
\begin{equation*}
A^{(\ell)} := \Pi^\ast_p \bigg( \sum_{s \leq \cdot} |\Delta M_s|^{2\ell} \bigg),
\end{equation*}
where $\Pi^\ast_p$ denotes the dual predictable projection.
\end{lemma}

\begin{proof}
This is a consequence of the extended BDG inequality in \cite{HernandezHernandezJackal2022}. More precisely, for any $G \in \cG$, the process $(\tM_t)_{t \in [0,T]} := (M_t \1_G)_{t \in [0,T]}$ is a local martingale, and $\langle \tM \rangle = \langle M \rangle \1_G$. Also, $\widetilde{A}^{(\frac{q}{2})} := \Pi^\ast_p(\sum_{s \leq \cdot} |\Delta \tM_s|^q) = \Pi^\ast_p(\sum_{s \leq \cdot} |\Delta M_s|^q \1_G) = A^{(\frac{q}{2})} \1_G$. Then, by the standard BDG inequality combined with \cite[Theorem~2.1]{HernandezHernandezJackal2022}, we have that
\begin{equation*}
\E \Big[ \Big( \sup_{s \in [0,t]} |M_s|^q \Big) \1_G \Big] = \E \Big[ \sup_{s \in [0,t]} |\tM_s|^q \Big] \leq C_q \E \Big[ \langle \tM \rangle_t^{\frac{q}{2}} \vee \widetilde{A}^{(\frac{q}{2})}_t \Big] = C_q \E \Big[ \Big( \langle M \rangle_t^{\frac{q}{2}} \vee A^{(\frac{q}{2})}_t \Big) \1_G \Big].
\end{equation*}
Since this holds for all $G \in \cG$, we infer the second inequality in \eqref{eq: extended conditional BDG}, and the first inequality may be obtained similarly.
\end{proof}

\begin{remark}
Let $N$ be an integer-valued random measure on $\Omega \times [0,T] \times \U$ with compensator $\nu$. Let $\zeta \in G_{\textup{loc}}(N)$ be a real-valued function such that $\E [\int_0^T \int_{\U} |\zeta(s,u)|^q \, \nu(\d s,\d u)] < \infty$ for some $q \in [2,\infty)$, and let $M = \int_0^\cdot \int_{\U} \zeta(s,u) \, \tN(\d s,\d u)$. Then, with $\beta$ and $D$ defined as in \eqref{eq: definition D_i and beta^i}, we have by, e.g., \cite[Ch.~II, 1.15]{JacodShiryaev2003}, that
\begin{equation*}
\sum_{s \leq \cdot} |\Delta M_s|^q = \sum_{s \leq \cdot} |\zeta(s,\beta_s)|^q \1_D(s) = \int_0^\cdot \int_{\U} |\zeta(s,u)|^q \, N(\d s,\d u),
\end{equation*}
and hence by, e.g., \cite[Theorem~13.2.21]{Cohen2015}, that its dual predictable projection is given by
\begin{equation}\label{eq: dual predictable projection for Poisson martingale integral}
\Pi^\ast_p \bigg( \sum_{s \leq \cdot} |\Delta M_s|^q \bigg) = \int_0^\cdot \int_{\U} |\zeta(s,u)|^q \, \nu(\d s,\d u).
\end{equation}
\end{remark}

\begin{lemma}\label{lemma: application of conditional BDG for jump measures}
Let $N$ be an integer-valued random measure on $\Omega \times [0,T] \times \U$, such that its compensator $\nu$ satisfies $\nu(\{t\} \times \U) = 0$ for every $t \in [0,T]$, and write $\tN = N - \nu$ for the corresponding compensated random measure. Then the following hold.
\begin{itemize}
\item[(i)] Let $q \in [2,\infty)$ and $r \in [q,\infty]$. There exists a constant $C$, which depends only on $q$, such that, for any $\zeta \in G_{\textup{loc}}(N)$ with $\E [\int_0^T \int_{\U} |\zeta(s,u)|^q \, \nu(\d s,\d u)] < \infty$ and any $(s,h) \in \Delta_{[0,T]}$,
\begin{equation}\label{eq: quadratic variation bound for compensated poisson martingale integrals using BDG}
\begin{split}
\bigg\| &\sup_{v \in [s,h]} \bigg| \int_s^v \int_{\U} \zeta(t,u) \, \tN(\d t,\d u) \bigg| \bigg\|_{q,r,s}\\
&\leq C \bigg\| \E_s \bigg[ \bigg( \int_s^h \int_{\U} |\zeta(t,u)|^2 \, \nu(\d t,\d u) \bigg)^{\hspace{-2pt}\frac{q}{2}} \vee \int_s^h \int_{\U} |\zeta(t,u)|^q \, \nu(\d t,\d u) \bigg]^{\frac{1}{q}} \bigg\|_{L^r}.
\end{split}
\end{equation}
\item[(ii)] Let $\nu(\d t,\d u) = K(t,\d u) \dd A_t$ as in \eqref{eq: decomposition of nu}, and suppose that $A \in V^{\frac{p}{2}} L^{\frac{q}{2},\frac{r}{2}} \cap V^{\frac{p}{q}} L^{1,\frac{r}{q}}$ for some $2 \leq q \leq p < \infty$ and $r \in [q,\infty]$, and let $\zeta \in G_{\textup{loc}}(N)$ such that
\[\bigg\| \sup_{t \in [0,T]} \int_{\U} \big( |\zeta(t,u)|^2 \vee |\zeta(t,u)|^q \big) \, K(t,\d u) \bigg\|_{L^\infty} \leq L\]
for some finite constant $L > 0$. Then there exists a constant $C$, which depends only on $q$ and $L$, such that
\begin{equation*}
\bigg\| \int_0^\cdot \int_{\U} \zeta(t,u) \, \tN(\d t,\d u) \bigg\|_{p,q,r,[0,T]} \leq C \Big( \|A\|_{\frac{p}{2},\frac{q}{2},\frac{r}{2},[0,T]}^{\frac{1}{2}} + \|A\|_{\frac{p}{q},1,\frac{r}{q},[0,T]}^\frac{1}{q} \Big).
\end{equation*}
\end{itemize}
\end{lemma}

\begin{proof}
By considering the integral $\int_0^\cdot \int_{\U} \zeta(t,u) \, \tN(\d t,\d u)$ componentwise, we may assume without loss of generality that $\zeta$ is real-valued.

\emph{(i):}
By the conditional BDG inequality in Lemma~\ref{lemma: conditional extended BDG}, together with \eqref{eq: dual predictable projection for Poisson martingale integral}, we have that
\begin{align*}
&\bigg\| \E_s \bigg[ \sup_{v \in [s,h]} \bigg| \int_s^v \int_{\U} \zeta(t,u) \, \tN(\d t,\d u) \bigg|^q \bigg]^{\frac{1}{q}} \bigg\|_{L^r}\\
&\lesssim \bigg\| \E_s \bigg[ \bigg\langle \int_s^\cdot \int_{\U} \zeta(t,u) \, \tN(\d t,\d u) \bigg\rangle_{\hspace{-2pt}h}^{\hspace{-2pt}\frac{q}{2}} \vee \int_s^h \int_{\U} |\zeta(t,u)|^q \, \nu(\d t,\d u) \bigg]^{\frac{1}{q}} \bigg\|_{L^r},
\end{align*}
and the estimate in \eqref{eq: quadratic variation bound for compensated poisson martingale integrals using BDG} then follows by \cite[Ch.~II, Theorem~1.33]{JacodShiryaev2003}.

\emph{(ii):}
Using the bound in \eqref{eq: quadratic variation bound for compensated poisson martingale integrals using BDG}, for any $(s,v) \in \Delta_{[0,T]}$, we have that
\begin{align*}
&\bigg\| \int_s^v \int_{\U} \zeta(t,u) \, \tN(\d t,\d u) \bigg\|_{q,r,s}\\
&\lesssim \bigg\| \E_s \bigg[ \bigg( \int_s^v \int_{\U} |\zeta(t,u)|^2 \, K(t,\d u) \dd A_t \bigg)^{\hspace{-2pt}\frac{q}{2}} \vee \int_s^v \int_{\U} |\zeta(t,u)|^q \, K(t,\d u) \dd A_t \bigg]^{\frac{1}{q}} \bigg\|_{L^r}\\
&\leq L \Big( \|\delta A_{s,v}\|_{\frac{q}{2},\frac{r}{2},s}^{\frac{1}{2}} + \|\delta A_{s,v}\|_{1,\frac{r}{q},s}^{\frac{1}{q}} \Big) \leq L \Big( \|A\|_{\frac{p}{2},\frac{q}{2},\frac{r}{2},[s,v]}^{\frac{1}{2}} + \|A\|_{\frac{p}{q},1,\frac{r}{q},[s,v]}^{\frac{1}{q}} \Big)
\end{align*}
and, since $q \leq p$, the desired estimate then follows by superadditivity.
\end{proof}

\subsection{Some results on rough paths}\label{appendix: rough paths}

\begin{proof}[Proof of Proposition \ref{proposition: consistency rough and stochastic integrals for semimartingales}]
Let $t \in [0,T]$. By, e.g., \cite[Ch.~II, Theorem~21]{Protter2005}, there exists a sequence of partitions $(\cP^k)_{k \in \N}$ of the interval $[0,t]$ with vanishing mesh size, such that the limit
\begin{equation}\label{eq: lim Yu delta Xuv int Ys- dXs}
\lim_{k \to \infty} \sum_{[u,v] \in \cP^k} Y_u \delta X_{u,v} = \int_0^t Y_{s-} \dd X_s
\end{equation}
holds $(\P \otimes \bar{\P})$-almost surely.

By \cite[Lemma~4.35]{ChevyrevFriz2019}, we have that $\sum_{[u,v] \in \cP^k} Y'_{u-} \X_{u,v} \to 0$ as $k \to \infty$ as a limit in probability with respect to $\P \otimes \bar{\P}$, and by switching to a subsequence if necessary, we may assume that the limit
\begin{equation}\label{eq: lim sum Y'u- Xuv = 0}
\lim_{k \to \infty} \sum_{[u,v] \in \cP^k} Y'_{u-} \X_{u,v} = 0
\end{equation}
holds $(\P \otimes \bar{\P})$-almost surely.

We have by assumption that $(Y(\omega, \cdot),Y'(\omega, \cdot)) \in \cV_{X(\omega)}^{p,q,r,\bar{\Omega}}$ for $\P$-almost every $\omega \in \Omega$. For any such $\omega \in \Omega$, letting $\Xi_{s,t} = \Delta Y'_s(\omega, \cdot) \X_{s,t}(\omega)$, we have that
\begin{equation*}
\|\Xi_{s,t}\|_{L^q(\bar{\Omega})} \leq \delta(s)^{\frac{1}{p}} w(s,t)^{\frac{2}{p}},
\end{equation*}
where $\delta(s) := \|\Delta Y'_s(\omega, \cdot)\|_{L^q(\bar{\Omega})}^p$ and $w(s,t) := \|\X(\omega)\|_{\frac{p}{2},[s,t]}^{\frac{p}{2}}$. It then follows from \cite[Theorem~2.11]{FrizZhang2018} that $\sum_{[u,v] \in \cP^k} \Delta Y'_u(\omega, \cdot) \X_{u,v}(\omega) \to 0$ as $k \to \infty$ in probability with respect to $\bar{\P}$. Since this convergence holds for $\P$-almost every $\omega \in \Omega$, it follows that $\sum_{[u,v] \in \cP^k} \Delta Y'_u \X_{u,v} \to 0$ as $k \to \infty$ in probability with respect to $\P \otimes \bar{\P}$. By switching to a further subsequence if necessary, we may then assume that the limit
\begin{equation}\label{eq: lim Delta Y'_u Xuv = 0}
\lim_{k \to \infty} \sum_{[u,v] \in \cP^k} \Delta Y'_u \X_{u,v} = 0
\end{equation}
holds $\P \otimes \bar{\P}$-almost surely, and since $Y'_u = Y'_{u-} + \Delta Y'_u$, combining the limits in \eqref{eq: lim sum Y'u- Xuv = 0} and \eqref{eq: lim Delta Y'_u Xuv = 0}, we have that
\begin{equation}\label{eq: lim sum Y'u Xuv = 0}
\lim_{k \to \infty} \sum_{[u,v] \in \cP^k} Y'_u \X_{u,v} = 0
\end{equation}
also holds $\P \otimes \bar{\P}$-almost surely.

Combining \eqref{eq: lim Yu delta Xuv int Ys- dXs} and \eqref{eq: lim sum Y'u Xuv = 0}, we have in particular, by Fubini's theorem (e.g., \cite[Theorem~3.4.1]{Bogachev2007}), that for $\P$-almost every $\omega \in \Omega$, the limit
\begin{equation*}
\lim_{k \to \infty} \sum_{[u,v] \in \cP^k} Y_u(\omega, \cdot) \delta X_{u,v}(\omega) + Y'_u(\omega, \cdot) \X_{u,v}(\omega) = \bigg( \int_0^t Y_{s-} \dd X_s \bigg) (\omega, \cdot)
\end{equation*}
holds $\bar{\P}$-almost surely. However, by the definition of the rough stochastic integral (see Lemma~\ref{lemma: rough stochastic integral}), we also have that, for $\P$-almost every $\omega \in \Omega$,
\begin{equation*}
\lim_{k \to \infty} \sum_{[u,v] \in \cP^k} Y_u(\omega, \cdot) \delta X_{u,v}(\omega) + Y'_u(\omega, \cdot) \X_{u,v}(\omega) = \int_0^t Y_s(\omega, \cdot) \dd \bX_s(\omega)
\end{equation*}
exists as a limit in probability with respect to $\bar{\P}$. By the $\bar{\P}$-almost sure uniqueness of limits, for $\P$-almost every $\omega \in \Omega$, we have that
\begin{equation*}
\bigg( \int_0^t Y_{u-} \dd X_u \bigg)(\omega, \cdot) = \int_0^t Y_u(\omega, \cdot) \dd \bX_u(\omega),
\end{equation*}
$\bar{\P}$-almost surely. Since this holds for every $t \in [0,T]$, and both It\^o and rough stochastic integrals have almost surely c\`adl\`ag sample paths, we deduce the result.
\end{proof}

\begin{lemma}\label{lemma: |bX|_p,[0,T] leq N_alpha + jumps}
Let $p \in [2,3)$ and let $\bX \in \sV^p$ be a rough path. For any $\alpha > 0$, we have that
\begin{equation}\label{eq: bound on rough path norm in terms of N_alpha and jumps}
\|\bX\|_{p,[0,T]} \leq C \Big( N_{\alpha,[0,T]}\big(\|\bX\|_{p,[\cdot,\cdot)}^p\big) + 1 \Big)^{\hspace{-1pt}3} \Big( 1 + \sup_{t \in (0,T]} |\Delta \bX_t| \Big)^{\hspace{-1pt}2},
\end{equation}
where the constant $C$ depends only on $p$ and $\alpha$.
\end{lemma}

\begin{proof}
Let us write $N = N_{\alpha,[0,T]}(\|\bX\|_{p,[\cdot,\cdot)}^p)$ and $\{t_i\}_{i=0}^{N+1} = \{t_i(\alpha)\}_{i=0}^{N+1}$ for the partition defined in \eqref{eq: definition tau_i(alpha)}, with $t_{N+1} = T$. Let $\cP$ be an arbitrary partition of the interval $[0,T]$, and let $\bar{\cP} = \cP \cup \{t_i\}_{i=0}^{N+1}$. Similarly to the proof of \cite[Lemma~4]{FrizRiedel2013}, we bound
\begin{align*}
\sum_{[u,v] \in \cP} |\delta X_{u,v}|^p &\leq (N+1)^{p-1} \sum_{[s,t] \in \bar{\cP}} |\delta X_{s,t}|^p \leq (N+1)^{p-1} \sum_{i=0}^N \|X\|_{p,[t_i,t_{i+1}]}^p\\
&\leq (N+1)^{p-1} \sum_{i=0}^{N-1} \big( \|X\|_{p,[t_i,t_{i+1})} + |\Delta X_{t_{i+1}}| \big)^p \lesssim (N+1)^p \Big(\alpha + \sup_{t \in (0,T]} |\Delta X_t|^p \Big),
\end{align*}
where we used the fact that $\|X\|_{p,[t_i,t_{i+1}]} \leq \|X\|_{p,[t_i,t_{i+1})} + |\Delta X_{t_{i+1}}|$ (see \cite[Lemma~2.6]{AllanLiuProemel2021}), which implies that
\begin{equation*}
\|X\|_{p,[0,T]} \lesssim \Big( N_{\alpha,[0,T]}\big(\|\bX\|_{p,[\cdot,\cdot)}^p\big) + 1 \Big) \Big( 1 + \sup_{t \in (0,T]} |\Delta \bX_t| \Big).
\end{equation*}
A similar argument can then be used to bound $\|\X\|_{\frac{p}{2},[0,T]}$. This is more tedious, since $\X$ is not additive, and using Chen's relation introduces additional terms. In particular, we use the inequality $\|\X\|_{\frac{p}{2},[s,t]} \lesssim \|\X\|_{\frac{p}{2},[s,t)} + |\Delta \X_t| + \|X\|_{p,[s,t)}^2 + |\Delta X_t|^2$, which follows from a straightforward adaptation of the proof of \cite[Lemma~2.6]{AllanLiuProemel2021}. After some calculation, which we omit for brevity, we arrive at the estimate in \eqref{eq: bound on rough path norm in terms of N_alpha and jumps}.
\end{proof}

\subsection{Proof of Theorem~\ref{theorem: a p-variation kolmogorov result}}\label{section: Kolmogorov result}

\begin{proof}[Proof of Theorem~\ref{theorem: a p-variation kolmogorov result}]
Let us first consider the latter estimate. We note that when $q = p$ the claim follows easily by the linearity of expectation.

When $q < p$ the result essentially follows from Minkowski's integral inequality (see, e.g., \cite[Exercise~3.10.46]{Bogachev2007}), but for completeness we spell out the argument here. We first fix a (deterministic) partition $\cP = \{t_i\}_{i=0}^N$ of $[0,T]$. Letting $f(i,\omega) := |\delta Y_{t_i,t_{i+1}}(\omega)|$, we have that the function $f \colon \{0, 1, \ldots, N-1\} \times \Omega \to [0,\infty)$ is jointly measurable. Let us also write $\mu$ for the counting measure on $\{0, 1, \ldots, N-1\}$. By Minkowski's integral inequality, we then have that
\begin{equation*}
\begin{split}
\bigg( &\sum_{i=0}^{N-1} \|\delta Y_{t_i,t_{i+1}}\|_{L^q}^p \bigg)^{\hspace{-2pt}\frac{1}{p}} = \bigg( \int_{\{0, \ldots, N-1\}} \bigg( \int_{\Omega} |f(i,\omega)|^q \dd \P(\omega) \bigg)^{\hspace{-2pt}\frac{p}{q}} \dd \mu(i) \bigg)^{\hspace{-2pt}\frac{1}{p}}\\
&\leq \bigg( \int_{\Omega} \bigg( \int_{\{0, \ldots, N-1\}} |f(i,\omega)|^p \dd \mu(i) \bigg)^{\hspace{-2pt}\frac{q}{p}} \dd \P(\omega) \bigg)^{\hspace{-2pt}\frac{1}{q}} = \bigg\| \bigg( \sum_{i=0}^{N-1} |\delta Y_{t_i,t_{i+1}}|^p \bigg)^{\hspace{-2pt}\frac{1}{p}} \bigg\|_{L^q},
\end{split}
\end{equation*}
which implies the desired estimate.

\smallskip

Let us now take $p < q$ and consider the first estimate. Let $w$ be the control function given by $w(s,t) = \|Y\|_{p,q,[s,t]}^p$ for $(s,t) \in \Delta_{[0,T]}$. We may assume without loss of generality that $\tp \leq q$, and that $\Delta Y_T = 0$.

Let $d^n_i$ for $n \in \N$ and $i = 0, 1, \ldots, 2^n$ be the $w(\cdot, \cdot-)$-midpoints of the interval $[0,T]$, in the sense of \cite[Section~3.2]{Le2023}. For convenience, let us refer to $D_n = \{[d^n_i,d^n_{i+1}] : i = 0, 1, \ldots, 2^n-1\}$ as the $w$-dyadic partition at level $n$, and write $\cD = \cup_{n \in \N} D_n$ for the set of all $w$-dyadic points.

We may assume without loss of generality that the mesh size $|D_n| \to 0$ as $n \to \infty$. (Indeed, if not, we may simply add the function $(s,t) \mapsto \epsilon (t - s)$ to the control $w$, which will ensure that $|D_n| \to 0$, and then take $\epsilon \to 0$ at the end of the proof.) This means in particular that $\cD$ is dense in $[0,T]$.

Since $p < \tp$, we have, by a standard property of such midpoints (see, e.g., \cite[Lemma~7.8]{AllanPieper2026}), that for every $n \in \N$,
\begin{equation}\label{eq: midpoints bound for sum alt.}
\sum_{i=0}^{2^n-1} w(d^n_i,d^n_{i+1}-)^{\frac{\tp}{p}} \leq 2^{-(\frac{\tp}{p} - 1) n} w(0,T-)^{\frac{\tp}{p}}.
\end{equation}

Given arbitrary points $s, t \in \cD$ with $s < t$, the interval $[s,t]$ can be expressed as the finite union of essentially disjoint intervals of the form $[d^n_i,d^n_{i+1}]$, where no three intervals are in the same level $n$. In other words, for each $n \in \N$, there exists an index set $I_n^{s,t} \subset \{0, 1, \ldots, 2^n-1\}$, with $|I_n^{s,t}| \in \{0, 1, 2\}$ and $I_n^{s,t} = \emptyset$ for sufficiently large $n$, such that
\[ [s,t] = \bigcup_{n=1}^\infty \bigcup_{i \in I_n^{s,t}} [d^n_i,d^n_{i+1}], \]
where all the intervals are disjoint except for their endpoints. In particular, we can write
\[ \delta Y_{s,t} = \sum_{n=1}^\infty \sum_{i \in I_n^{s,t}} \delta Y_{d^n_i,d^n_{i+1}}. \]

Let $\gamma > \tp - 1$. Using the fact that $(\sum_{n=1}^\infty a_n)^{\tp} \leq C \sum_{n=1}^\infty n^\gamma a_n^{\tp}$ for any sequence $(a_n)_{n \in \N}$ of non-negative numbers, where the constant $C$ depends only on $\tp$ and $\gamma$ (see the proof of \cite[Lemma~2]{LedouxLyonsQian2002}), we then have that
\[ |\delta Y_{s,t}|^{\tp} \leq C \sum_{n=1}^\infty n^\gamma \bigg( \sum_{i \in I^{s,t}_n} |\delta Y_{d^n_i,d^n_{i+1}}| \bigg)^{\hspace{-2pt}\tp} \leq C 2^{\tp-1} \sum_{n=1}^\infty n^\gamma \sum_{i \in I^{s,t}_n} |\delta Y_{d^n_i,d^n_{i+1}}|^{\tp}. \]

Let us now take a partition $\{0 = t_0 < t_1 < \cdots < t_m = T\} \subset \cD$ of the interval $[0,T]$. We can apply the procedure above to each of the intervals $[t_j,t_{j+1}]$ in this partition. Noting that each dyadic point $d^n_{i+1} \in \cD$ appears at most once as the right-endpoint of one of the considered intervals (i.e., we need sum each jump $\Delta Y_{d^n_{i+1}}$ at most once), we see that
\begin{equation}\label{eq: pathwise bound for sums over partitions alt.}
\begin{split}
\sum_{j=0}^{m-1} |\delta Y_{t_j,t_{j+1}}|^{\tp} &\leq C 2^{2(\tp-1)} \bigg( \sum_{n=1}^\infty n^\gamma \sum_{i=0}^{2^n-1} |\delta Y_{d^n_i,d^n_{i+1}-}|^{\tp} +  \sum_{n =1}^{\infty} n^{\gamma}\sum_{i=0}^{2^{n-1}-1} |\Delta Y_{d^n_{2i+1}}|^{\tp} \bigg)\\
&=: C 2^{2(\tp-1)} \big( Z_1^{\tp} + Z_2^{\tp} \big).
\end{split}
\end{equation}
Using the bound in \eqref{eq: midpoints bound for sum alt.}, we have that
\begin{equation*}
\sum_{i=0}^{2^n-1} \|\delta Y_{d^n_i,d^n_{i+1}-}\|_{L^q}^{\tp} \leq \sum_{i=0}^{2^n-1} w(d^n_i,d^n_{i+1}-)^{\frac{\tp}{p}} \leq 2^{-(\frac{\tp}{p} - 1) n} w(0,T)^{\frac{\tp}{p}} = 2^{-(\frac{\tp}{p} - 1) n} \|Y\|_{p,q,[0,T]}^{\tp},
\end{equation*}
and hence, since $p < \tp \leq q$, we have that
\begin{equation}\label{eq: Z_1 estimate}
\begin{split}
\| Z_1 \|_{L^q}^{\tp} &= \bigg\| \sum_{n=1}^\infty n^\gamma \sum_{i=0}^{2^n-1} |\delta Y_{d^n_i,d^n_{i+1}-}|^{\tp} \bigg\|_{L^{\frac{q}{\tp}}} \leq \sum_{n=1}^\infty n^\gamma \sum_{i=0}^{2^n-1} \big\| |\delta Y_{d^n_i,d^n_{i+1}-}|^{\tp} \big\|_{L^{\frac{q}{\tp}}}\\
&= \sum_{n=1}^\infty n^\gamma \sum_{i=0}^{2^n-1} \| \delta Y_{d^n_i,d^n_{i+1}-} \|_{L^q}^{\tp} \leq \|Y\|_{p,q,[0,T]}^{\tp} \sum_{n=1}^\infty n^\gamma 2^{-(\frac{\tp}{p} - 1) n} < \infty.
\end{split}
\end{equation}

Noting that $d^n_{2i+1} \in (d^{n-1}_i, d^{n-1}_{i+1})$ for each $i \in \{0, \ldots, 2^{n-1}-1\}$, we have that
\begin{equation*}
\|\Delta Y_{d^n_{2i+1}}\|^{\tp}_{L^q} = \lim_{s \nearrow d^n_{2i+1}} \|\delta Y_{s,d^n_{2i+1}}\|_{L^q}^{\tp} \leq w(d^{n-1}_i, d^{n-1}_{i+1}-)^{\frac{\tp}{p}},
\end{equation*}
and hence that
\begin{equation}\label{eq: Z_2 estimate}
\begin{split}
\| Z_2 \|_{L^q}^{\tp} &\leq \sum_{n=1}^\infty n^\gamma \sum_{i=0}^{2^{n-1}-1} \| \Delta Y_{d^n_{2i+1}} \|_{L^q}^{\tp} \leq \sum_{n=1}^\infty n^\gamma \sum_{i=0}^{2^{n-1}-1} w(d^{n-1}_i, d^{n-1}_{i+1}-)^{\frac{\tp}{p}}\\
&\leq \sum_{n=1}^\infty n^\gamma 2^{-(\frac{\tp}{p} - 1) (n-1)} w(0,T)^{\frac{\tp}{p}} = \|Y\|_{p,q,[0,T]}^{\tp} \sum_{n=1}^\infty n^\gamma 2^{-(\frac{\tp}{p} - 1) (n-1)} < \infty.
\end{split}
\end{equation}
We thus have in particular that the right-hand side of \eqref{eq: pathwise bound for sums over partitions alt.} is almost surely finite.

It follows that almost every sample path of $Y$ has finite $\tp$-variation on the $w$-dyadic times $\cD$. In particular, these sample paths are regulated (in the sense that their left and right-limits exist at every point). We also recall that the dyadic times $\cD$ are dense in $[0,T]$. Thus, for every non-dyadic time $t \in (0,T) \setminus \cD$, the limit $\tY_t := \lim_{k \to \infty} Y_{t_k}$ exists almost surely whenever $(t_k)_{k \in \N} \subset \cD$ is a decreasing sequence of dyadic times with $t_k \searrow t$, and the limit does not depend on the choice of the sequence $(t_k)_{k \in \N}$. Setting $\tY_t = Y_t$ for all dyadic times $t \in \cD$, we obtain a process $\tY$ on $[0,T]$.

It is straightforward to see that, by construction, the sample paths of $\tY$ are almost surely c\`adl\`ag and have finite $\tp$-variation on $[0,T]$. Since $Y_{t_k} \to \tY_t$ almost surely, and $Y_{t_k} \to Y_t$ in probability as $k \to \infty$, we have that $\tY_t = Y_t$ almost surely for every $t \in [0,T]$, so that $\tY$ is indeed a modification of $Y$.

Finally, by combining \eqref{eq: pathwise bound for sums over partitions alt.} with the bounds in \eqref{eq: Z_1 estimate} and \eqref{eq: Z_2 estimate}, we deduce the estimate in \eqref{eq: Kolmogorov estimate}.
\end{proof}

\begin{example}\label{example: p<q is sharp for kolmogorov}
Let $(\Omega,\cF,\P)$ be a probability space, and let $p, \tp, q \geq 1$. For each $k \geq 1$, let $(E_{k,i})_{i = 1, \ldots, k} \subset \cF$ be a collection of events which form a partition of the sample space $\Omega$, such that $\P(E_{k,i}) = \frac{1}{k}$ for each $i = 1, \ldots, k$. We define random variables $(\zeta_{k,i})_{i = 1, \ldots, k}$ by
\[ \zeta_{k,i} = k^{-\frac{1}{\tp}} \1_{E_{k,i}} \]
for each $i = 1, \ldots, k$, and let $Z$ be the sequence of random variables given by
\[ Z := (0, \zeta_{1,1}, 0, \zeta_{2,1}, 0, \zeta_{2,2}, 0, \zeta_{3,1}, 0, \zeta_{3,2}, 0, \zeta_{3,3}, 0, \zeta_{4,1}, \ldots). \]
Then, for every $\omega \in \Omega$, we have that
\[ \|Z(\omega)\|_{\tp}^{\tp} = 2 \sum_{k=1}^\infty k^{-1} = \infty. \]

For any $k \leq m$, any $i = 1, \ldots, k$ and any $j = 1, \ldots, m$, noting that $\P(\zeta_{k,i} - \zeta_{m,j} \neq 0) \leq \frac{1}{k} + \frac{1}{m} \leq \frac{2}{k}$, we see that
\[ \|\zeta_{k,i} - \zeta_{m,j}\|_{L^q} \leq k^{-\frac{1}{\tp}} \Big(\frac{2}{k}\Big)^{\frac{1}{q}} = 2^{\frac{1}{q}} k^{-\frac{1}{q} - \frac{1}{\tp}}. \]
Hence, we can bound
\[ \|Z\|_{p,q}^p \leq \sum_{k=1}^\infty 2k \cdot 2^{\frac{p}{q}} k^{-\frac{p}{q} - \frac{p}{\tp}} \leq 2^{1+\frac{p}{q}} \sum_{k=1}^\infty k^{1 - \frac{p}{q} - \frac{p}{\tp}}. \]
The sum above is finite when
\begin{equation}\label{eq: pqtp condition}
\frac{1}{q} + \frac{1}{\tp} > \frac{2}{p}.
\end{equation}
Thus, whenever the condition in \eqref{eq: pqtp condition} is satisfied, we conclude that
\[ \|Z\|_{p,q} < \infty \qquad \text{but} \qquad \|Z(\omega)\|_{\tp} = \infty \text{~~for every~~} \omega \in \Omega. \]
In particular, if $p > q$ then we can take $\tp > p$ above, which implies that the estimate in \eqref{eq: Kolmogorov estimate} does not hold in general without the assumption that $p < q$.

Moreover, if $p < q$ then \eqref{eq: pqtp condition} requires that $\tp < p$, but if $p \approx q$ then we can also take $\tp \approx p$, indicating that in general \eqref{eq: Kolmogorov estimate} also cannot hold for $\tp < p$.
\end{example}

\bibliographystyle{plain}
\bibliography{refs}

@article{Clark1978,
author = {John M. C. Clark},
title = {The design of robust approximations to the stochastic differential equations of nonlinear filtering },
journal = {Communication Systems and Random Process Theory (Proc. 2nd NATO Advanced Study)},
volume= {25},
year= {1978},
pages = {721--734}
}

@incollection{Koch2014,
author= {Herbert Koch},
editor= {Michael Griebel},
title={{Adapted Function Spaces for Dispersive Equations}},
bookTitle={{Singular Phenomena and Scaling in Mathematical Models}},
year={2014},
publisher={Springer International Publishing},
pages={49--67},
}

@article{FrizSeegerKranich2022,
author = {Peter K. Friz and Benjamin Seeger and Pavel Zorin--Kranich},
title = {Besov rough path analysis},
volume = {339},
journal = {J. Differential Equations},
pages = {152--231},
year = {2022}
}

@article{ChevyrevFrizKorepanovMelbourneZhang2022,
author = {Ilya Chevyrev and Peter Friz and Alexey Korepanov and Ian Melbourne and Huilin Zhang},
title = {Deterministic homogenization under optimal moment assumptions for fast–slow systems. {Part 2}},
volume = {58},
journal = {Ann. Inst. Henri Poincaré Probab. Stat.},
number = {3},
pages = {1328--1350},
year = {2022},
}

@InProceedings{PisierXu1987,
author={Gilles Pisier and Quanhua Xu},
editor={Lindenstrauss, Joram and Milman, Vitali D.},
title={Random series in the real interpolation spaces between the spaces vp},
booktitle={Geometrical Aspects of Functional Analysis},
year={1987},
publisher={Springer Berlin Heidelberg},
pages={185--209},
}

@book{Pisier2016,
place = {Cambridge},
series = {Cambridge Studies in Advanced Mathematics},
title = {Martingales in {B}anach Spaces},
publisher = {Cambridge University Press},
author = {Gilles Pisier},
year = {2016},
collection = {Cambridge Studies in Advanced Mathematics}
}

@article {Gertner1978,
author = {Izidor Gertner},
title = {An alternative approach to nonlinear filtering},
Journal = {Stochastic Process. Appl.},
volume = {7},
year = {1978},
number = {3},
pages = {231--246},
}

@article {Poklukar2006,
author = {Darja Rupnik Poklukar},
title = {Nonlinear filtering for jump-diffusions},
journal = {J. Comput. Appl. Math.},
volume = {197},
year = {2006},
number = {2},
pages = {558--567}
}

@article {PopaSritharan2009,
author = {Silvia Popa and Sivaguru S. Sritharan},
title = {Nonlinear filtering of {I}t{\^o}{--L}{\'e}vy stochastic differential equations with continuous observations},
journal = {Commun. Stoch. Anal.},
volume = {3},
year = {2009},
number = {3},
pages = {313--330}
}

@book {CrisanRozovskiu2011,
title = {The {O}xford handbook of nonlinear filtering},
editor = {Dan Crisan and Boris Rozovski\u i},
publisher = {Oxford University Press, Oxford},
Year = {2011}
}

@book{JacodProtter2012,
author = {Jean Jacod and Philip Protter},
title = {Discretization of {P}rocesses},
volume = {67},
publisher = {Springer-Verlag},
year = {2012},
series= {Stochastic {M}odelling and {A}pplied {P}robability}
}

@book{LipsterShirayev1977,
author = {Robert S. Liptser and Albert N. Shiryayev},
TITLE = {Statistics of random processes. {I}},
series = {Applications of Mathematics},
volume = {5},
publisher = {Springer-Verlag, New York-Heidelberg},
year = {1977}
}

@book {Kallianpur1980,
author = {Gopinath Kallianpur},
title = {Stochastic filtering theory},
series = {Applications of Mathematics},
volume = {13},
publisher = {Springer-Verlag, New York-Berlin},
year = {1980}
}

@article{Kushner1997,
author = {Harold J. Kushner},
title = {Robustness and Convergence of Approximations to Nonlinear Filters for Jump-Diffusions},
volume = {16},
journal = {Matem{\'a}tica Aplicada e Computacional},
number = {2},
pages = {153--183},
year = {1997},
}

@article{FreyRunggaldier2010,
author = {R{\"u}diger Frey and Wolfgang Runggaldier},
title = {Pricing credit derivatives under incomplete information: {A} nonlinear-filtering approach},
volume = {14},
journal = {Finance Stoch.},
number = {4},
pages = {495--526},
year = {2010},
}

@article{Qiao2023,
author = {Huijie Qiao},
title = {Convergence of nonlinear filtering for multiscale systems with correlated {L\'e}vy noises},
journal = {Stoch. Dyn.},
volume = {23},
number = {2},
year = {2023},
pages={2350016},
}

@article{GermGyongy2025,
author = {Fabian Germ and Istv{\'a}n Gy{\"o}ngy},
title = {On partially observed jump diffusions {III}: regularity of the filtering density},
volume = {13},
journal = {Stoch. Partial Differ. Equ. Anal. Comput.},
number = {1},
pages = {531--583},
year = {2025},
}

@article{KliemannKochMarchetti1990,
  author={Wolfgang H. Kliemann and Giorgio Koch and Frederico Marchetti},
  journal={IEEE Transactions on Information Theory}, 
  title={On the unnormalized solution of the filtering problem with counting process observations}, 
  year={1990},
  volume={36},
  number={6},
  pages={1415--1425},
}

@article{GrigelionisMikulevicius1982,
author = {Bronius Grigelionis and Remigijus Mikulevi{č}ius},
title = {Robustness in nonlinear filtering theory},
volume = {22},
journal = {Litovsk. Mat. Sb.},
number = {4},
pages = {37--45},
year={1982},
}

@article{FernandoHausenblas2018,
author = {B. P. W. Fernando and E. Hausenblas},
title = {Nonlinear filtering with correlated {Lé}vy noise characterized by copulas},
volume = {32},
journal = {Braz. J. Probab. Stat.},
number = {2},
pages = {374--421},
year = {2018},
}

@article{CeciColaneri2012,
 author = {Claudia Ceci and Katia Colaneri},
 journal = {Adv. in Appl. Probab.},
 number = {3},
 pages = {678--701},
 publisher = {Applied Probability Trust},
 title = {Nonlinear filtering for jump diffusion observations},
 volume = {44},
 year = {2012}
}

@article{HorstZhang2025,
author    ={Ulrich Horst and Huilin Zhang},
title     ={Pontryagin {M}aximum {P}rinciple for rough stochastic systems and pathwise stochastic control}, 
year      ={2025},
journal   ={Preprint arXiv:2503.22959},
}

@article{FrizLeZhang2025,
author={Peter K. Friz and Khoa L{\^e} and Huilin Zhang},
title={Randomisation of rough stochastic differential equations}, 
year={2025},
journal={Preprint arXiv:2503.06622},
}

@article{FrizLeZhang2024,
title={Controlled rough {SDEs}, pathwise stochastic control and dynamic programming principles}, 
author={Peter K. Friz and Khoa L{\^e} and Huilin Zhang},
year={2024},
journal={Preprint arXiv:2412.05698},
}

@article{BuginiFrizStannat2024,
author = {Fabio Bugini and Peter K. Friz and Wilhelm Stannat},
title = {Parameter dependent rough {SDE}s with applications to rough {PDE}s},
journal = {Preprint arXiv:2409.11330},
year = {2024},
}

@article{Le2022b,
author = {Khoa L{\^e}},
title = {Maximal inequalities and weighted {BMO} processes},
year = {2022},
journal = {Preprint arXiv:2211.15550},
}

@incollection{Davis2005,
  author    = {Mark H. A. Davis},
  title     = {Pathwise nonlinear filtering with correlated noise},
  booktitle = {The Oxford Handbook of Nonlinear Filtering},
  publisher = {Oxford University Press},
  year      = {2011},
  pages     = {403--424}, 
  address   = {Oxford}, 
}

@article{ElliottKohlmann1981,
author = {Robert J. Elliott and Michael Kohlmann},
title = {Robust {F}iltering for {C}orrelated {M}ultidimensional {O}bservations},
volume = {178},
journal = {Math. Z.},
pages = {559--578},
year = {1981}
}

@article{Davis1980,
author = {Mark H. A. Davis},
title = {On a Multiplicative Functional Transformation Arising in Nonlinear Filtering Theory},
volume = {54},
journal = {Z. Wahrsch. Verw. Gebiete},
number = {2},
pages = {125--139},
year = {1980}
}

@article{DavisSpathopoulos1987,
author = {Mark H. A. Davis and M. P. Spathopoulos},
title = {Pathwise nonlinear filtering for nondegenerate diffusions with noise correlation},
volume = {25},
journal = {SIAM J. Control Optim.},
number = {2},
pages = {260--278},
year = {1987}
}

@article{Davis1982,
author = {Mark H. A. Davis},
title = {A pathwise solution of the equations of nonlinear filtering},
volume = {27},
journal = {Teor. Veroyatnost. i Primenen.},
number = {1},
publisher = {},
pages = {160--167},
year = {1982}
}

@article{ClarkCrisan2005 ,
author = {John M. C. Clark and Dan Crisan},
title = {On a robust version of the integral representation formula of nonlinear filtering},
volume = {133},
journal = {Probab. Theory Relat. Fields},
pages = {43--56},
year = {2005}
}

@article{Kushner1979,
author = {Harold J. Kushner},
title = {A robust discrete state approximation to the optimal nonlinear filter for a diffusion},
volume = {3},
journal = {Stochastics},
number = {2},
pages = {75--83},
year = {1979}
}

@article{CassOgrodnik2017,
author = {Thomas Cass and Marcel Ogrodnik},
title = {Tail estimates for {M}arkovian rough paths},
volume = {45},
journal = {Ann. Probab.},
number = {4},
pages = {2477--2504},
year = {2017}
}

@article{FrizOberhauser2010,
author = {Peter K. Friz and Harald Oberhauser},
title = {A generalized {F}ernique theorem and applications},
volume = {138},
journal = {Proc. Amer. Math. Soc.},
number = {10},
pages = {3679--3688},
year = {2010}
}

@article{FrizRiedel2013 ,
author = {Peter K. Friz and Sebastian Riedel},
title = {Integrability of (Non-)Linear Rough Differential Equations and Integrals},
volume = {31},
journal = {Stoch. Anal. Appl.},
number = {2},
pages = {336-358},
year = {2013}
}

@article{FrizGessGulisashviliRiedel2016,
author = {Peter K. Friz and Benjamin Gess and Archil Gulisashvili and Sebastian Riedel},
title = {The {J}ain{--M}onrad criterion for rough paths and applications to random {F}ourier series and non-{M}arkovian {H}{\"o}rmander theory},
volume = {44},
journal = {Ann. Probab.},
number = {1},
pages = {684--738},
year = {2016},
}

@article{CassLittererLyons2013 ,
author = {Thomas Cass and Christian Litterer and Terry Lyons},
title = {Integrability and tail estimates for {G}aussian rough differential equations},
volume = {41},
journal = {Ann. Probab.},
number = {4},
pages = {3026--3050},
year = {2013}
}

@article{LedouxLyonsQian2002,
author = {Michel Ledoux and Terry Lyons and Zhongmin Qian},
title = {L{\'e}vy area of {W}iener processes in {B}anach spaces},
journal = {Ann. Probab.},
volume = {30},
number = {2},
pages = {546--578},
year = {2002}
}

@article{Le2022a,
author = {Khoa L{\^e}},
title = {Quantitative {J}ohn--{N}irenberg inequality for stochastic processes of bounded mean oscillation},
year = {2022},
journal = {Preprint arXiv:2210.15736},
}

@article{AllanLiuProemel2021,
  author  = {Andrew L. Allan and Chong Liu and David J. Pr{\"o}mel},
  journal = {Stochastic Process. Appl.},
  title   = {C{\`a}dl{\`a}g rough differential equations with reflecting barriers},
  year    = {2021},
  pages   = {79--104},
  volume  = {142},
}

@article{HernandezHernandezJackal2022,
author = {Ma. Elena Hern{\'a}ndez-Hern{\'a}ndez and Saul D. Jackal},
title = {A generalisation of the {B}urkholder--{D}avis--{G}undy inequalities},
volume = {27},
journal = {Electron. Commun. Probab. },
number = {50},
pages = {1--8},
year = {2022},
}

@article{AllanPieper2026,
author = {Andrew L. Allan and Jost Pieper},
title = {Rough Stochastic Analysis with Jumps},
journal = {Electron. J. Probab.},
year = {2026},
volume = {31},
number = {77},
pages = {1--62}
}

@article{CoghiNilssenNüskenReich2023 ,
author = {Michele Coghi and Torstein Nilssen and Nikolas Nüsken and Sebastian Reich},
title = {Rough {McKean--Vlasov} dynamics for robust ensemble {Kalman} filtering},
volume = {33},
journal = {Ann. Appl. Probab.},
number = {6B},
pages = {5693--5752},
year = {2023},
}

@article{BankBayerFrizPelizzari2025,
author = {Peter Bank and Christian Bayer and Peter K. Friz and Luca Pelizzari},
title = {Rough {PDEs} for local stochastic volatility models},
year={2025},
journal={Math. Finance},
volume = {35},
number = {3},
pages = {661--681}
}

@article{Lyons1998,
  author  = {Terry J. Lyons},
  title   = {Differential equations driven by rough signals},
  journal = {Rev. Mat. Iberoamericana},
  year    = {1998},
  volume  = {14},
  pages   = {215--310},
}

@article{BuginiCoghiNilssen2024,
author = {Fabio Bugini and Michele Coghi and Torstein Nilssen},
title = {Malliavin calculus for rough stochastic differential equations},
year = {2024},
journal={Preprint arXiv:2402.12056},
eprint ={2402.12056},
}

@article{Le2023,
author = {Khoa L{\^e}},
title = {Stochastic sewing in {B}anach spaces},
volume = {28},
journal = {Electron. J. Probab.},
number = {26},
pages = {1-22},
year = {2023},
}

@article{CrisanDiehlFrizOberhauser2013,
author = {Dan Crisan and Joscha Diehl and Peter K. Friz and Harald Oberhauser},
title = {Robust filtering: Correlated noise and multidimensional observation},
volume = {23},
journal = {Ann. Appl. Probab.},
number = {5},
publisher = {Institute of Mathematical Statistics},
pages = {2139--2160},
year = {2013},
doi = {10.1214/12-AAP896},
URL = {https://doi.org/10.1214/12-AAP896}
}

@article{DavieGermGyongy2024,
author = {Alexander Davie and Fabian Germ and Istv{\'a}n Gy{\"o}ngy},
title = {On partially observed jump diffusions {II}: the filtering density},
volume = {12},
journal = {Stoch. Partial Differ. Equ. Anal. Comput.},
number = {3},
pages = {1628--1698},
year = {2024},
}

@article{MeyerProske2004,
author= {Thilo Meyer-Brandis and Frank Proske},
title= {Explicit solution of a non-linear filtering problem for {L}{\'e}vy processes with application to finance},
journal= {Appl. Math. Optim.},
volume= {50},
number= {2},
pages= {119--134},
year= {2004},
}

@article{ChevyrevFriz2019,
author= {Ilya Chevyrev and Peter K. Friz},
title= {Canonical {RDEs} and general semimartingales as rough paths},
journal= {Ann. Probab.},
volume= {47},
number= {1},
pages= {420-463},
year= {2019},
}

@article{FrizKranich2023,
title={Rough semimartingales and {$p$}-variation estimates for martingale transforms},
author={Peter K. Friz and Pavel Zorin-Kranich},
year={2023},
journal={Ann. Probab.},
volume = {51},
number = {2},
pages = {397--441}
}

@article{frizHocLe2021,
      title={Rough stochastic differential equations},
      author={Peter K. Friz and Antoine Hocquet and Khoa L{\^e}},
      year={2021},
      journal={Preprint arXiv:2106.10340},
      archivePrefix={arXiv},
      eprint={2106.10340},
      primaryClass={math.PR}
}

@article{FrizZhang2018,
title= {Differential equations driven by rough paths with jumps},
author= {Peter K. Friz and Huilin Zhang},
year= {2018},
journal= {J. Differential Equations},
volume = {264},
pages = {6226--6301}
}

@article{Qiao2021,
title= {Nonlinear filtering of stochastic differential equations with correlated {L\'e}vy noises},
author={Huijie Qiao},
year={2021},
journal={Stochastics},
pages={1156--1185},
volume={93},
number={8},
}

@article{QiaoDuan2015,
title= {Nonlinear filtering of stochastic dynamical systems with {L\'e}vy noises},
author= {Hujie Qiao and Jinqiao Duan},
year= {2015},
journal= {Adv. in Appl. Probab.},
pages= {902-918},
volume= {47},
}

@article{CeciColaneri2014,
author= {Claudia Ceci and Katia Colaneri},
title= {The {Z}akai equation of nonlinear filtering for jump-diffusion observations: {E}xistence and uniqueness},
journal= {Appl. Math. Optim.},
year= {2014},
pages= {47-82},
volume={69},
number={1},
}

@article{GermGyongy2025PartI,
title={On partially observed jump diffusions {I}: the filtering equations},
author={Fabian Germ and Istv{\'a}n Gy{\"o}ngy},
year={2025},
journal={Stoch. Partial Differ. Equ. Anal. Comput.},
volume = {to appear}
}

@article{DiehlOberhauserRiedel2015,
title = {A {Lé}vy area between {B}rownian motion and rough paths with applications to robust nonlinear filtering and rough partial differential equations},
journal = {Stochastic Process. Appl.},
volume = {125},
number = {1},
pages = {161--181},
year = {2015},
issn = {0304-4149},
doi = {https://doi.org/10.1016/j.spa.2014.08.005},
url = {https://www.sciencedirect.com/science/article/pii/S0304414914002002},
author = {Josha Diehl and Harald Oberhauser and Sebastian Riedel}
}

@book{JacodShiryaev2003,
author= {Jean Jacod and Albert N. Shiryaev},
title= {Limit theorems for stochastic processes},
publisher= {Springer},
edition = {2nd},
year= {2003},
series= {Grundlehren der mathematischen Wissenschaften},
}

@book{Cohen2015,
author= {Samuel N. Cohen and Robert J. Elliot},
title= {Stochastic Calculus and Applications},
publisher={Birkhäuser},
edition= {2nd},
year= {2015},
series= {Probability and its Applications},
}

@book{Bogachev2007,
address = {Berlin, Heidelberg},
isbn = {3-540-34513-2},
publisher = {Springer},
title = {Measure Theory},
year = {2007},
author = {Vladimir I. Bogachev},
}

@book{BainCrisan2009,
publisher = {Springer},
series = {Stochastic Modelling and Applied Probability 60},
title = {Fundamentals of stochastic filtering},
year = {2009},
author = {Alan Bain and Dan Crisan},
}

@book{MandrekarRudiger2014,
author= {Vidyadhar Mandrekar and Babara Rüdiger},
title= {Stochastic Integration in {B}anach Spaces: Theory and Applications},
volume= {73},
year= {2014},
series= {Probability Theory and Stochastic Modelling},
publisher= {Springer},
}

@book{Protter2005,
address = {Berlin},
title = {Stochastic integration and differential equations: A new approach},
edition={2nd},
isbn = {3540509968},
publisher={Springer},
language = {eng;ger},
series = {Stochastic Modelling and Applied Probability},
year = {2005},
author = {Protter, Philip E.},
}

@book{FrizVictoir2010,
author= {Peter K. Friz and Nicolas B. Victoir},
title= {Multidimensional Stochastic Processes as Rough Paths: Theory and Applications},
collection={Cambridge Studies in Advanced Mathematics},
year= {2010},
publisher= {Cambridge University Press},
}

@Book{FrizHairer2020,
  author    = {Peter K. Friz and Martin Hairer},
  publisher = {Springer},
  title     = {A Course on Rough Paths, With an Introduction to Regularity Structures},
  year      = {2020},
  edition   = {2nd},
}

\end{document}